\documentclass{article}
\usepackage[T1]{fontenc}
\usepackage{lmodern}
\usepackage{mathtools, amsmath, amssymb, xcolor, amsthm, float, animate, mathrsfs, algorithm2e, csquotes, enumerate}
\usepackage{booktabs}
\usepackage{pdfpages}
\usepackage[normalem]{ulem}
\mathtoolsset{showonlyrefs}
\usepackage{url}
\usepackage{hyperref}
\usepackage[export]{adjustbox}
\usepackage{geometry}

\newcommand{\figref}[1]{Figure~\ref{#1}}

\newcounter{dummy} 
\numberwithin{dummy}{section}

\newtheorem{remark}[dummy]{Remark}
\newtheorem{theorem}[dummy]{Theorem}
\newtheorem{lemma}[dummy]{Lemma}
\newtheorem{example}[dummy]{Example}

\newtheorem{cor}[dummy]{Corollary}
\newtheorem{proposition}[dummy]{Proposition}

\usepackage{appendix}

\DeclareMathOperator{\R}{\mathbb R}
\DeclareMathOperator{\N}{\mathbb N}
\let\P\relax
\DeclareMathOperator{\P}{\mathcal P}
\DeclareMathOperator{\C}{\mathcal C}
\DeclareMathOperator{\F}{\mathcal F}
\DeclareMathOperator{\id}{id}
\DeclareMathOperator{\NN}{\mathcal N}
\DeclareMathOperator{\erf}{erf}
\newcommand{\lebesgue}{\Lambda}
\DeclareMathOperator{\dom}{dom}
\DeclareMathOperator{\conv}{conv}
\DeclareMathOperator{\ran}{Range}
\DeclareMathOperator{\loc}{loc}
\DeclareMathOperator{\supp}{supp}
\DeclareMathOperator*{\argmin}{arg\,min}
\renewcommand{\d}{\mathop{}\!\mathrm{d}}

\let\vareps\varepsilon

\usepackage[draft, todonotes={textsize=tiny}]{changes}
\setuptodonotes{color=red, backgroundcolor=red!60!white,
  bordercolor=red, tickmarkheight=10pt}
  
\definechangesauthor[name=Robert, color=cyan]{RB}

\definechangesauthor[name=Richard, color=orange]{RD}

\title{Wasserstein Gradient Flows of MMD Functionals with Distance Kernel
and 
Cauchy Problems on Quantile Functions}
\author{Richard Duong\footnote{Institute of Mathematics,
	   Technische Universität Berlin,
	   Stra{\ss}e des 17.\ Juni 136, 
	   10623 Berlin, Germany,
	   {\ttfamily\{duong, stein, beinert, steidl\}@math.tu-berlin.de},
      \url{https://tu.berlin/imageanalysis}.
	} 
\and
Viktor Stein\footnotemark[1]
\and
Robert Beinert\footnotemark[1]
\and
Johannes Hertrich\footnote{Univesité Paris Dauphine-PSL and Inria Mokaplan, Paris, France,
\texttt{johannes.hertrich@dauphine.psl.eu}.}
		\and Gabriele Steidl\footnotemark[1] 
 }
\date{\today}

\begin{document}
\maketitle

\begin{abstract}
    We give a comprehensive description of Wasserstein gradient flows of maximum mean discrepancy (MMD) functionals $\mathcal F_\nu \coloneqq \text{MMD}_K^2(\cdot, \nu)$ towards given target measures $\nu$ on the real line, where we focus on the negative
    distance kernel $K(x,y) \coloneqq -|x-y|$.
    In one dimension, the Wasserstein-2 space can be isometrically embedded
    into the cone $\mathcal C(0,1) \subset L_2(0,1)$ of quantile functions 
    leading to a characterization of Wasserstein gradient flows
    via the solution of an associated Cauchy problem on $L_2(0,1)$.
    Based on the construction of an appropriate counterpart of $\mathcal F_\nu$ on $L_2(0,1)$
    and its subdifferential,
    we provide a solution of the Cauchy problem.
    For discrete target measures $\nu$, this results in a piecewise linear solution formula.
    We prove invariance and smoothing properties of the flow on subsets of $\mathcal C(0,1)$.
    For certain $\mathcal F_\nu$-flows this implies that initial point measures instantly become absolutely continuous, and stay so over time.
    Finally, 
    we illustrate the behavior of the flow by various numerical examples using an implicit Euler scheme, which is easily computable by a bisection algorithm. For continuous targets $\nu$, also the explicit Euler scheme can be employed, although with limited convergence guarantees.
\end{abstract}

\section{Introduction}
Wasserstein gradient flows have long been studied in stochastic analysis 
and have recently received increasing interest in machine learning,
leading to intriguing research questions that 
often fall outside the scope of the existing theory, see, e.g., \cite{CHHRS2023,HHABCS2023,LBADDP2022}.
In this paper, we concentrate on
Wasserstein gradient flows of MMD functionals 
$\mathcal F_\nu \coloneqq \text{MMD}_K^2(\cdot,\nu)$ towards given target measures $\nu$. 
Wasserstein gradient flows of MMDs with smooth kernels $K$ like, e.g., the Gaussian one, preserve absolutely continuous measures as well as empirical measures, so that in the latter case, 
just the movement of particles has to be taken into account
\cite{AKSG2019}.
For non-smooth kernels like, e.g., the negative distance kernel, 
the properties of the measure can heavily vary during the flow. 
\figref{fig:gradient_flow_dirac_dirac_P2R} shows the simple example
of a Wasserstein gradient flow on the line starting from an initial measure $\mu_0 = \delta_{-1}$ towards the target measure $\nu = \delta_0$, where the behavior of the flow changes completely, see also Example~\ref{ex:delta0_to_uniform}.

\begin{figure}[tb]
      \centering
      \includegraphics[width=0.7\textwidth, trim=0pt 20pt 0pt 30pt, clip]{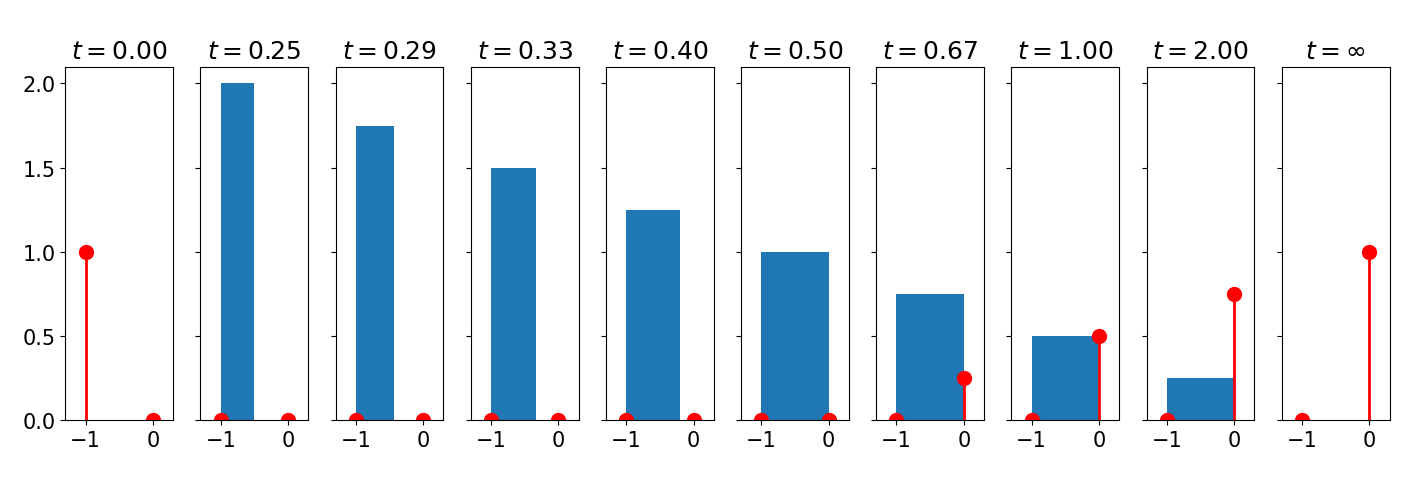}
      \vspace{-5pt}
      \caption{
      Wasserstein gradient flow of $\mathcal F_\nu$
      from $\delta_{-1}$ towards $\nu = \delta_0$.
      The absolutely continuous part is visualized by its density in blue (area equals mass) and the atomic part by the red dotted vertical line (height equals mass).
      The flow changes immediately from a point measure to a uniform one
      with increasing support. It stays absolutely continuous until approaching
      the target support point $0$. Then it becomes the sum of an absolutely continuous measure and a discrete one, where the weight of the latter increases in time, see \cite{HBGS2023}.}
      \label{fig:gradient_flow_dirac_dirac_P2R}
\end{figure}

In general, the characterization of Wasserstein gradient flows of MMD functionals with non-smooth kernels appears to be complicated. For the interaction energy, which is only one part of the MMD functional,
we refer to \cite{HGBS2022} and in relation with potential theory to \cite{CMSVW2024};
for functionals which are the sum of the interaction energy and a potential energy from $\C^1(\R^d)$, which are moreover geodesically convex, see \cite{CLW2016}.
Existence and global convergence results of MMD flows with the Coulomb kernel were proven in \cite{BV2023}.

However, the flow examination can be simplified in one dimension, since in this case the
Wasserstein space $\P_2(\R)$ can be isometrically embedded
into the cone $\mathcal C(0,1) \subset L_2(0,1)$ of quantile functions of measures.
Using this isometry, gradient flows of the negative squared Wasserstein distance were studied in \cite{NS2009}. Concerning the MMD functional, 
the flow of just the interaction energy, and also for a system of two interacting measures, was considered in
\cite{BCDP15,CDEFS2020}, see Remark~\ref{rem:carrillo}. 
The recent preprint \cite{FFR2024} investigates a non-local conservation law, where the interaction kernel is Lipschitz continuous and supported on the negative half line.
Another special case, namely with a single point target measure, was treated in \cite{HBGS2023}.
In \cite{dFFHM2014}, the quantile embedding was used to study the convergence of the gradient flow of kernel-based attraction-repulsion functionals. 
Discretized versions of $\F_\nu$ in the one-dimensional case and their convergence properties to the continuous setting were considered in \cite{DB2022,DLJ2023,FHS2013}.
Our work is also different from the recent article \cite{HMVV2024}, where the functional is a porous-medium type approximation of the interaction energy for the Riesz (or Coulomb) kernel $- | x - y |^{r}$ with $r \in (- 1, 0)$; and the kernel $- | x - y |^r$ for $r > 1$ (but not $r = 1$) is covered by the results in \cite{CCH2014}.
Furthermore, connections between continuity equations -- or, more generally, scalar conservation laws in one dimension -- and the corresponding PDEs for quantile functions associated to various Wasserstein-$p$ distances have been investigated in, e.g., \cite{BD2008,FR2011,LT2004}, but focusing on the longtime behavior of solutions.
Here, we refer to the overview paper \cite{F2016}.

In this paper, we consider the flow of the \emph{whole} MMD functional with repulsion term \textit{and attraction term} for \emph{arbitrary} starting and target measures, concentrating on the negative distance kernel $K(x,y) \coloneqq -|x-y|$.
Taking the isometric embedding into account, a functional $F_\nu$ on $\mathcal C(0,1)$ can be associated to the
MMD functional $\F_\nu$ on the Wasserstein space. We construct a continuous extension of this functional $F_\nu$ to the space $L_2(0,1)$ and compute its subdifferential. Then, starting with a measure $\mu_0$ with quantile function $Q_{\mu_0}$,
we characterize the Wasserstein gradient flow of $\F_\nu$ via $\gamma_t \coloneqq (g(t))_\# \Lambda_{(0,1)}$, where
$g$ is the unique strong solution of the associated Cauchy problem on $L_2(0,1)$: 
\begin{align*}
    \begin{cases}
        \partial_t g(t) \in - \partial F_\nu (g(t)), \quad t \in (0,\infty), \\
         g(0) = Q_{\mu_0}.
    \end{cases}
\end{align*}
Indeed, the strong solution of the Cauchy problem exists and stays within the cone $\mathcal C(0,1)$, i.e., \textit{quantile functions remain quantile functions}.
We identify further interesting invariant subset of $\mathcal C(0,1)$,
in which the $F_\nu$-flow remains once it has started there.
In particular, we prove that the $F_\nu$-flow preserves continuity and Lipschitz properties of the initial datum $Q_{\mu_0}$ under suitable conditions on the target $Q_{\nu}$. This includes cases where point masses \textit{cannot} form during the flow -- in contrast to the example given by Figure~\ref{fig:gradient_flow_dirac_dirac_P2R}.
Also, the phenomenon, where initial point masses immediately explode into absolutely continuous measures, is covered by our smoothing result, see Theorem~\ref{thm:regularisation-Lip}.
This smoothing effect is driven by the repulsive nature of the \textit{interaction energy} part of our MMD functional, see also \cite{BCDP15, CDEFS2020, FFR2024}. But note that in our case, the \textit{potential energy} part in conjunction with a \textit{fixed} target measure $\nu$ plays a fundamental role, whether or not this smoothing property can come into effect, and remain over time -- see again Figure~\ref{fig:gradient_flow_dirac_dirac_P2R}.

Moreover, we discuss the explicit pointwise solution $g_s$ of the Cauchy problem
at continuity points $s \in (0,1)$ of the quantile function $Q_\nu$ and show that
the family of these functions determines the solution of the Cauchy problem
via $[g(t)] (s) = g_s(t)$.
Furthermore, we derive the Euler backward and forward schemes which are easy to implement due to our explicit representation of $\partial F_\nu$, and we illustrate the flow by some numerical examples.

\paragraph{Outline of the paper.}
In Section~\ref{sec:prelim_1}, we recall Wasserstein gradient flows, and especially flows on the line in Section~\ref{sec:prelim}. 
We take great care of the relation between quantile functions and cumulative distribution functions (CDFs) of probability measures and recall properties of maximal monotone operators
on Hilbert spaces which we need later.
Finally, we formulate our central characterization of
Wasserstein gradient flows by the solution of a Cauchy problem in $L_2(0,1)$.
The MMD and the functional $\F_\nu \coloneqq \text{MMD}_K^2(\cdot,\nu)$ with the distance kernel is introduced in Section~\ref{sec:mmd}.
We determine a continuous functional $F_\nu$ on $L_2(0,1)$ whose restriction to $\mathcal C(0,1)$ is associated with $\F_\nu$, meaning that it fulfills $\F_\nu(\mu) =F_\nu(Q_\mu)$ for all measures $\mu$ in the Wasserstein space.
We compute the subdifferential of $F_\nu$ and show that the related Cauchy problem
produces indeed a flow that stays within the cone $\mathcal C(0,1)$.
In Section~\ref{sec:explicit}, we provide a solution of the pointwise Cauchy problem
and show that it also determines the overall solution. In particular, an explicit solution formula is given for discrete target measures $\nu$.
Section~\ref{sec:inv} deals with invariance and smoothing properties of our gradient flows. 
We show that certain $F_\nu$-flows preserve (and improve) Lipschitz properties of the initial quantile $Q_{\mu_0}$ by describing the time evolution of the Lipschitz constants.
Also, we prove for general targets $\nu$ that the support of the starting measure stays convex and grows monotonically.
Finally, Section~\ref{sec:numerics} briefly introduces Euler backward and forward schemes and illustrates properties of the flow by numerical examples.
The appendix collects auxiliary and additional material.
\\[2ex]

\section{Wasserstein Gradient Flows}\label{sec:prelim_1}
Let $\mathcal M(\R^d)$ denote the space of $\sigma$-additive, signed Borel measures and 
$\mathcal P(\R^d)$ the set of probability measures on $\R^d$.
For $\mu \in \mathcal M(\R^d)$ and a measurable map $T\colon\R^d \to \R^n$,
the \emph{push-forward} of $\mu$ via $T$ is given by
$T_{\#}\mu \coloneqq \mu \circ T^{-1}$.
We consider the \emph{Wasserstein space}
$
  \P_2(\R^d) 
  \coloneqq 
  \{ \mu \in \P(\R^d) \colon \int_{\R^d}\|x\|_2^2 \d \mu(x) < \infty \}
$
equipped with the \emph{Wasserstein distance} 
$W_2\colon\P_2(\R^d) \times \P_2(\R^d) \to [0,\infty)$,
\begin{equation}        \label{def:W_2}
  W_2^2(\mu, \nu)
  \coloneqq 
  \min_{\pi \in \Gamma(\mu, \nu)} 
  \int_{\R^d\times \R^d}
  \|x - y\|_2^2
  \d \pi(x, y),
  \qquad \mu,\nu \in \P_2(\R^d),
\end{equation}
where 
$
\Gamma(\mu, \nu)
\coloneqq
\{ \pi \in \P_2(\R^d \times \R^d):
(\pi_{1})_{\#} \pi = \mu,\; (\pi_{2})_{\#} \pi = \nu\}
$ 
and $\pi_i(x) \coloneqq x_i$, $i = 1,2$ for $x = (x_1,x_2)$, see e.g., \cite{Vil03,BookVi09}.
Further, $\| \cdot \|_2$ denotes the Euclidean norm on $\R^d$.
The set of optimal transport plans $\pi$ realizing the minimum in \eqref{def:W_2} is denoted by $\Gamma^{\rm{opt}}(\mu, \nu)$.

A curve $\gamma \colon I \to \P_2(\R^d)$ on an interval $I \subset \R$, is called a \emph{geodesic} 
if there exists a constant $C \ge 0$ 
such that 
\begin{equation}
    \label{eq:geodesic}
    W_2(\gamma_{t_1}, \gamma_{t_2}) = C |t_2 - t_1|, \qquad \text{for all } t_1, t_2 \in I.
\end{equation}
There also exists the notion of \emph{generalized geodesics},
 which in one dimension coincides with that of geodesics, so we do not introduce it here.
 
The Wasserstein space is a geodesic space, meaning that any two measures $\mu, \nu \in \P_2(\R^d)$ can be connected by a geodesic.
For $\lambda \in \R$, a function $\F\colon \P_2(\R^d) \to (-\infty,\infty]$ is called 
\emph{$\lambda$-convex along geodesics} if, for every 
$\mu, \nu \in \dom(\F) \coloneqq \{\mu \in \P_2(\R^d): \F(\mu) < \infty\}$,
there exists at least one geodesic $\gamma \colon [0, 1] \to \P_2(\R^d)$ 
between $\mu$ and $\nu$ such that
\begin{equation} \label{def:lambda_convex}
    \F(\gamma_t) 
    \le
    (1-t) \, \F(\mu) + t \, \F(\nu) 
    - \tfrac{\lambda}{2} \, t (1-t) \, W_2^2(\mu, \nu), 
    \qquad t \in [0,1].
\end{equation}
In the case $\lambda = 0$, we just speak about convex functions.

Let $L_{2,\mu}$ denote the Bochner space of (equivalence classes of) 
functions $\xi:\R^d \to \R^d$ with
$\|\xi \|_{L_{2,\mu}} ^2 \coloneqq \int_{\R^d} \|\xi(x)\|_2^2 \d \mu(x) < \infty$.
The \emph{(regular) tangent space} at $\mu \in \P_2(\R^d)$ is given by
\begin{align} \label{tan_reg}
    {\text T}_{\mu}\mathcal P_2(\R^d)
    &\coloneqq 
      \overline{
      \left\{ \lambda (T- I): (I ,T)_{\#} \mu \in \Gamma^{\text{opt}} (\mu , T_{\#} \mu), \; \lambda >0
      \right\} }^{L_{2,\mu}}.
\end{align}
For a proper and lower semicontinuous (lsc) function $\F \colon \P_2(\R^d) \to (-\infty, \infty]$ 
and $\mu \in \P_2(\R^d)$, 
the \emph{reduced Fr\'echet subdifferential $\partial \F(\mu)$ at $\mu$} 
consists of all $\xi \in L_{2,\mu}$ satisfying
{\small
\begin{equation}\label{need}
  \F(\nu) - \F(\mu)
    \ge 
    \inf_{\pi \in \Gamma^{\text{opt}}(\mu,\nu)}
    \int_{\R^{2d}}
    \langle \xi(x), y - x \rangle
    \d \pi (x, y)
    + o(W_2(\mu,\nu))
\end{equation}
}%
for all $\nu \in \P_2(\R^d)$.

A curve $\gamma\colon I \to \P_2(\R^d)$ is \emph{(locally) $p$-absolutely continuous} for $p \in [1, \infty]$ \cite[Def.~1.1.1]{BookAmGiSa05} if there exists $m \in L_{p}(I, \R)$ (resp. $L_{p, \loc}(I, \R)$) such that
\begin{equation*}
    W_2(\gamma_t, \gamma_s)
    \le \int_{s}^{t} m(r) \d{r} \qquad
    \text{for all } s, t \in I, s < t,
\end{equation*}
and we omit the $p$ if $p = 1$.
By \cite[Thm.~8.3.1]{BookAmGiSa05}, if $\gamma$ is \emph{absolutely continuous}, then
there exists a Borel velocity field $v$ of functions $v_t \colon \R^d \to \R^d$ with $\int_I \| v_t \|_{L_{2,\gamma_t}} \d t < \infty$ such that the \emph{continuity equation}
\begin{equation} \label{eq:CE}
  \partial_t \gamma_t + \nabla_x \cdot ( v_t \, \gamma_t) = 0
\end{equation}
holds on $I \times \R^d$ in the sense of distributions, i.e., for all $\varphi \in C_{\mathrm c}^\infty(I \times \R^d)$
it holds
\begin{equation} \label{eq:CE_distr}
    \int_I \int_{\R^d} \partial_t \varphi(t, x) + v_t(x) \cdot \nabla_x \, \varphi(t, x) \d \gamma_t(x) \d{t}
    = 0.
\end{equation}
The velocity field $v_t$ can be chosen to have minimal norm $\|v_t\|_{L_2, \gamma_t}$ among all velocity fields satisfying \eqref{eq:CE}. This unique optimal vector field fulfills $v_t \in \text{T}_{\gamma_t} \mathcal P_2(\R^d)$.

A locally $2$-absolutely continuous curve 
  $\gamma \colon (0,\infty) \to \P_2(\R^d)$ 
  with tangent velocity field $v_t \in \text{T}_{\gamma_t} \mathcal P_2(\R^d)$
  is called \emph{Wasserstein gradient flow 
  with respect to} $\F\colon \P_2(\R^d) \to (-\infty, \infty]$
  if 
  \begin{equation}\label{wgf}
      v_t \in - \partial \F(\gamma_t) \quad \text{for a.e. } t > 0.
  \end{equation}

We have the following theorem which holds also true in $\R^d$ when switching to so-called generalized geodesics, 
see \cite[Thm.~11.2.1]{BookAmGiSa05}.

\begin{theorem}[Existence and uniqueness of Wasserstein gradient flows] \label{thm:WGF}
    Let $\F \colon \P_2(\R) \to (- \infty, \infty]$ be bounded from below, lower semicontinuous (lsc) and $\lambda$-convex along geodesics and $\mu_0 \in \overline{\dom(\F)}$.
    Then there exists a unique Wasserstein gradient flow $\gamma \colon (0, \infty) \to \P_2(\R)$ with respect to $\F$ with $\gamma(0+) = \mu_0$.
    Furthermore, the piecewise constant curve 
    constructed from the iterates of the minimizing movement scheme
    \begin{equation} \label{jko}
        \mu_{n+1} \coloneqq \argmin_{\mu \in \P_2(\R^d)} \Big\{ \F(\mu) + \frac{1}{2 \tau} W_2^2(\mu_{n}, \mu) \Big\}, \qquad \tau > 0,
    \end{equation}
    i.e., $\gamma_\tau$ defined by $\gamma_\tau|_{(n \tau,(n+1)\tau]} \coloneqq \mu_n$, $n=0,1,\ldots$,
    converges locally uniformly to $\gamma$ as $\tau \downarrow 0$.

    If $\lambda > 0$, then $\F$ admits a unique minimizer $\bar{\mu} \in \P_2(\R)$ and we observe exponential convergence:
    \begin{equation*}
        W_2(\gamma_t, \bar{\mu})
        \le e^{- \lambda t} W_2(\mu_0, \bar{\mu})
        \quad \text{and} \quad
        \F(\gamma_t) - \F(\bar{\mu})
        \le e^{-2 \lambda t} \big( \F(\mu_0) - \F(\bar{\mu}) \big).
    \end{equation*}
    If $\lambda = 0$ and $\bar{\mu}$ is a minimizer of $\F$, then we have
    \begin{equation*}
        \F(\gamma_t) - \F(\bar{\mu})
        \le \frac{1}{2 t} W_2^2(\mu_0, \bar{\mu}).
    \end{equation*}
\end{theorem} 

\section{Wasserstein Gradient Flows in 1D}\label{sec:prelim}
In the following, we recall that $\P_2(\R)$ can be isometrically embedded into $L_2(0, 1)$ via so-called quantile functions. 
This reduces Wasserstein gradient flows in one dimension to the
consideration of gradient flows in the Hilbert space $L_2(0, 1)$, or more precisely, in the cone of quantile functions. We will see in the main Theorem~\ref{thm:L2_representation} of this section, that we have finally
to deal with a Cauchy inclusion problem.


The \emph{cumulative distribution function} (CDF) of $\mu \in \P(\R)$
is given by
\begin{equation*}
    R_{\mu}^+ \colon \R \to [0, 1], \quad
    R_{\mu}^+(x) \coloneqq \mu\big( (-\infty, x] \big), 
\end{equation*}
and its \emph{quantile function} 
by 
\begin{equation*}
    Q_{\mu} \colon (0, 1) \to \R, \quad
    Q_{\mu}(s) \coloneqq \min\{ x \in \R: R_{\mu}^+(x) \ge s \}.
\end{equation*}
\begin{remark}\label{l:DLOW-eq-lem}
    It is easy to check that $Q_\mu$ is strictly increasing if and only if $R_\mu^+$ is continuous. Further, $Q_\mu$ is continuous if and only if $R_\mu^+$ is strictly increasing on $\overline{(R_\mu^+)^{-1}(0,1)}$. 
\end{remark}
We will also need the function
\begin{equation*}
    R_{\mu}^{-} \colon \R \to [0, 1], \quad
    R_{\mu}^{-}(x) \coloneqq \mu\big( (- \infty, x) \big) .
\end{equation*}
The functions $R_\mu^+, R_\mu^-$ and $Q_\mu$ are monotonically increasing with only countably many discontinuities. 
If $\mu(\{x\}) = 0$ for all $x \in \R$, then the functions 
 $R_\mu \coloneqq R_{\mu}^{-} = R_{\mu}^{+}$ are continuous.
In general, the points of continuity $x \in \R$ of $R_\mu^+$ and $R_\mu^-$ coincide, and there it holds $R_\mu^-(x) = R_\mu^+(x)$. Generally, we have $R_\mu^- \le R_\mu^+$ such that 
$$
[R_\mu^-(x), R_\mu^+(x)] \ne \emptyset \hspace*{4mm} \text{ for all } x \in \R.
$$
Both $R_\mu^-$ and $Q_\mu$ are left-continuous and, since they are increasing, also lower semicontinuous (lsc), whereas $R_\mu^+$ is right-continuous and, since it is increasing, upper semicontinuous (usc).
Further, we have for $x \in \R$ that 
\begin{equation} \label{eq:bounds}
 R_{\mu}^+(x)
        = \begin{cases}
            \max\{ s \in (0, 1): Q_{\mu}(s) \le x \}, & \text{if } x \in [\inf Q_{\mu}(s) , \sup Q_{\mu}(s)), \\
            1, & \text{if } x \ge \sup Q_{\mu}(s), \\
            0, & \text{if } x < \inf Q_{\mu}(s),
        \end{cases}
\end{equation}
where the infimum and supremum is taken over all $s \in (0, 1)$, and
\begin{equation} \label{eq:bounds_1}
R_{\mu}^{-}(x)= \sup \{s \in (0, 1):  Q_{\mu}(s) < x \} \quad \text{for} \quad x > \inf_{s \in (0, 1)} Q_{\mu}(s).
\end{equation}
Note that formula \eqref{eq:bounds} is different from an erroneous one in \cite{RoRo14}.
Moreover, it holds the Galois inequalities which state that
\begin{gather}
    \label{eq:galois}
    Q_\mu(s) \le x \quad \text{if and only if} \quad s \le R_\mu^+(x).
   \end{gather}
The inequalities \eqref{eq:galois} state a generalized inversion relation between $R_\mu^+$ and $Q_\mu$, see Figure \ref{fig:CDF_vs_Quantile}.   
\begin{figure}[t]
   \centering
    \includegraphics[height=4.5cm,valign=c]{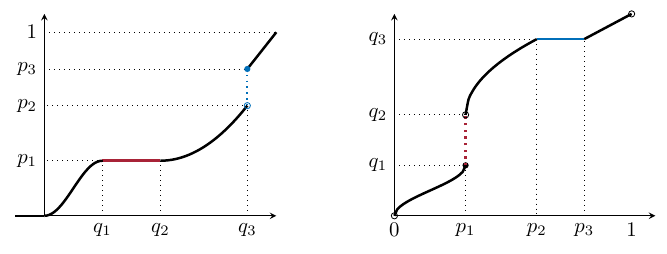}
   \caption{Left: the CDF $R_{\mu}^+$ of a probability measure $\mu \in \P(\R)$.
   Right: the corresponding quantile function $Q_{\mu}$. Intervals of constancy translate to jumps and vice-versa.}
   \label{fig:CDF_vs_Quantile}
\end{figure}

The quantile functions of measures in $\mathcal P_2(\R)$ form a closed, convex cone
$$
\mathcal C (0,1) := \{ f \in L_2(0,1): f ~ \text{is increasing\footnotemark (a.e.)} \},
$$
\footnotetext{For convenience, we define "increasing" as \emph{non-decreasing} ($\ge$), and we distinguish it from "strictly increasing" ($>$).}
see also Remark~\ref{rem:cone} in the appendix. For a further overview on quantile functions in convex analysis, see also \cite{RoRo14}.

By the following theorem, see, e.g., \cite[Thm.~2.18]{Vil03}, the mapping $\mu \mapsto Q_\mu$ is an isometric embedding 
of $\P_2(\R)$ into $L_2( 0,1 )$.

\begin{theorem}\label{prop:Q}
    For $\mu, \nu \in \P_2(\R)$,
    the quantile function $Q_{\mu} \in \mathcal C(0,1)$ 
		satisfies $\mu = (Q_{\mu})_{\#} \lebesgue_{(0,1)}$
  with the Lebesgue measure $\lebesgue$ on $(0,1)$
   	and 
    \begin{equation}
        W_2^2(\mu, \nu) = \int_{0}^1 |Q_{\mu}(s) - Q_{\nu}(s)|^2 \d s.
    \end{equation}
\end{theorem}

Thus, instead of working with 
$\F: \P_2(\R) \to (-\infty,\infty]$, we can just deal with associated
functions 
$$F \colon L_2( 0,1 ) \to (-\infty,\infty] \quad \text{with} \quad
F (Q_\mu) \coloneqq \F(\mu).
$$
Note that this relation determines 
$F$ only on $\mathcal C(0,1)$ and several extension to the whole space
$L_2( 0,1 )$ are possible.
One possibility would be to extend $F$ outside of $\mathcal C(0,1)$ by $\infty$. Yet, later in this work we will deal with a continuous extension of a specific functional which is everywhere finite on $L_2(0,1)$.

Now, instead of the (reduced) Fr\'echet subdifferential \eqref{need}, 
we will use the \emph{regular subdifferential} of functions 
$F \colon L_2( 0,1 ) \to (-\infty,\infty]$ defined by 
\begin{equation} \label{eq:subdiff_1}
    \partial F(u) \coloneqq 
    \bigl\{ v \in L_2( 0,1 ): F(w) \ge F(u) + \langle v, w-u \rangle +o(\|w-u\|) \; ~\text{for all } w\in L_2(0,1)\bigr\}.
\end{equation}
If $F$ is convex, then the $o$-term can be skipped.

The following theorem collects well-known properties of subdifferential operators
on Hilbert spaces \cite{BC2011}. 
Here, the \textit{domain} of a multivalued operator $A \colon L_2(0, 1) \to 2^{L_2(0, 1)}$ is denoted by $\dom(A) \coloneqq \{ u \in L_2(0, 1): A u \ne \emptyset \}$.
 
\begin{theorem} \label{general}
Let $F: L_2(0,1) \to (-\infty,\infty]$ be proper and convex.
Then $\partial F: L_2(0,1) \to 2^{L_2(0,1)}$ is a monotone operator, i.e.
for every $u_i$ and $v_i \in \partial F( u_i)$, $i=1,2$ it holds
$$
\langle u_1-u_2, v_1 - v_2 \rangle \ge 0.
$$
If $F$ is in addition lsc, then $\partial F$ is maximal monotone and thus for any $\vareps >0$, 
$$\text{{\rm Range}} (I + \vareps \partial F) = \bigcup\limits_{u \in L_2(0,1)} (I + \vareps \partial F)(u) = L_2(0,1).$$
Moreover, $\partial F$ is a closed operator, $\overline{\dom(F)} = \overline{\dom(\partial F)}$
and the resolvent 
$$J_{\varepsilon}^{\partial F} \coloneqq (I + \vareps \partial F)^{-1} : L_2(0,1) \to L_2(0,1)$$ 
is single-valued.
\end{theorem}

Concerning the $\lambda$-convexity and lower semicontinuity 
of $\mathcal F$, we have the following proposition, whose proof is given in the appendix. 

\begin{proposition}    \label{prop:conv-geo} 
    If $F: \mathcal C(0,1) \to (-\infty,\infty]$ is $\lambda$-convex, then 
    $\mathcal F(\mu) \coloneqq F(Q_\mu)$, $\mu \in \P_2(\R)$, is $\lambda$-convex along geodesics. If $F$ is lsc, then the same holds true for $\mathcal F$.
\end{proposition}

The next theorem is central for our further considerations.
It characterizes Wasserstein gradient flows by the solution of a Cauchy problem,
where we have to ensure that the solution remains in the cone $\mathcal C(0,1)$.
To this end, recall that given an 
operator $A: \dom(A) \subseteq L_2(0,1) \to 2^{L_2(0,1)}$ and an initial function $g_0 \in L_2(0,1)$, a \textit{strong solution} $g: [0, \infty) \to L_2(0,1)$ of the Cauchy problem
\begin{align}\label{eq:cauchy-strong}
    \begin{cases}
    \partial_t g(t) + Ag(t) ~\ni~ 0,  \quad t \in (0,\infty), \\
    g(0) ~=~ g_0,
    \end{cases}
\end{align}
is a function $g \in W^{1}_1\big((0,T];L_2(0,1) \big) \cap C\big([0,\infty); L_2(0,1)\big)$
for any $T > 0$ which meets the initial condition and solves the differential inclusion in \eqref{eq:cauchy-strong} pointwise for a.e. $t > 0$, where $\partial_t g(t)$ denotes the strong derivative of $g$ at $t$.

\begin{theorem}\label{thm:L2_representation}
    Let $F \colon L_2(0,1) \to (-\infty,\infty]$
    be a proper, convex and lsc function. Assume that for all (small) $\varepsilon > 0$, the resolvent $J_{\varepsilon}^{\partial F}$ maps $\C(0,1)$ into itself.
    Then, for any initial datum $g_0 \in \mathcal C(0,1) \cap \dom(\partial F)$, 
    the Cauchy problem 
    \begin{align}\label{eq:cauchy}
    \begin{cases}
    \partial_t g(t) + \partial F (g(t)) ~\ni~ 0, \quad t \in (0,\infty), \\
    g(0) ~=~ g_0,
    \end{cases}
    \end{align}
    has a unique strong solution $g: [0, \infty) \to \mathcal C(0, 1)$ which is expressible by the exponential formula
    \begin{equation} \label{eq:form}
            g(t) = \text{e} ^{-t \partial F} g_0 \coloneqq \lim_{n \to \infty} \left( I + \frac{t}{n} \partial F\right)^{-n} g_0.
        \end{equation}
    The curve $\gamma_t \coloneqq (g(t))_{\#} \lebesgue_{(0,1)}$ has quantile functions $Q_{\gamma_t} = g(t)$ and 
    is a Wasserstein gradient flow of $\F$ with $\gamma(0+) = (g_0)_{\#} \lebesgue_{(0,1)}$.
\end{theorem}

\begin{proof}
    1. By assumption, $\partial F$ is the subdifferential of a proper, convex and lsc function, and hence maximal monotone by Theorem~\ref{general}. By standard results of semigroup theory, see e.g., Theorem~\ref{theorem:BrezisRegularity}, there exists a strong solution
    $g: [0, \infty) \to L_2(0, 1)$ of \eqref{eq:cauchy} satisfying the exponential formula \eqref{eq:form}. It remains to show that starting in $g_0 \in \mathcal C(0,1)$, the solution remains in the cone $\mathcal C(0,1)$. Indeed, since $J_{\varepsilon}^{\partial F}$ maps $\C(0,1)$ into itself, we can conclude that $\left( I + \frac{t}{n} \partial F \right)^{-n} g_0 \in \mathcal C(0,1)$ for all $n \in \N$. Since $\mathcal C(0,1)$ is closed in $L_2(0,1)$, this also holds true when the limit
    in \eqref{eq:form} is taken, hence $g(t) \in \C(0,1)$ for all $t \ge 0$.
    
    2. Let $\gamma_t \coloneqq (g(t))_{\#} \lebesgue_{(0,1)}$.
    First, we show that $\gamma$ is locally $2$-absolutely continuous.
     By Theorem~\ref{theorem:BrezisRegularity}, the function $g$ is Lipschitz continuous, so there exists a $L > 0$ such that for all $s, t \ge 0$ with $s < t$ we have
    \begin{equation*}
        W_2(\gamma_t, \gamma_s)
        = \| g(t) - g(s) \|_{L_2(0, 1)}
        \le L (t - s)
        = \int_{s}^{t} L \d{\tau},
    \end{equation*}
    where the first equality is due to Theorem~\ref{prop:Q}.
    Hence, the curve $\gamma \colon [0, \infty) \to (\P_2(\R), W_2)$ is Lipschitz continuous, and in particular, locally $2$-absolutely continuous.
    
    Next, we show that the velocity field $v_t \in \text{T}_{\gamma_t} \mathcal P_2(\R)$ from \eqref{eq:CE} fulfills $v_t\in -\partial \F(\gamma_t)$.
    To calculate $v_t$, 
    we exploit \cite[Prop~8.4.6]{BookAmGiSa05}
    stating that,
    for a.e.~$t \in (0,\infty)$,
    the velocity field satisfies
    \begin{align}
        0
        &=\lim_{\tau\downarrow0}\frac{W_2(\gamma_{t+\tau},(I+\tau \, v_t)_\#\gamma_t)}{\tau} \\
        & =\lim_{\tau\downarrow0}\frac{W_2(g(t + \tau)_\#\lebesgue_{(0,1)},
        \bigl(g(t) +\tau \left(v_t\circ g(t) \right) \bigr)_\#\lebesgue_{(0,1)})}{\tau}.
    \end{align}
    First, let us assume that there exists $\tilde \tau > 0$ such that $I + \tilde\tau v_t$ is monotonically increasing. In the Appendix \ref{sec:appendix}, we will give the arguments in the general case. 
    Hence, for sufficiently small $0 < \tau < \tilde \tau$,
    the mappings $I + \tau v_t$ are still monotonically increasing.
    Consequently,
    the functions $g(t) + \tau (v_t \circ g(t))$ are also monotonically increasing,
    and their left-continuous representatives are quantile functions.
    Employing the isometry to $L_2(0,1)$,
    we hence obtain
    \begin{equation}
    0
         = \lim_{\tau \downarrow0} \Big\| \frac{g(t+\tau)-g(t)}{\tau}- v_t \circ g (t) \Big\|_{L_2(0,1)}
        = \|\partial_t g(t) - v_t \circ g(t) \|_{L_2(0,1)}.
    \end{equation}
    Thus, by construction \eqref{eq:cauchy} of $g$, we see that
    $v_t \circ g(t) \in - \partial F \left( g(t) \right)$ for a.e.~$t$.
    In particular, 
    for any $\mu\in\P_2(\R)$,
    we obtain
    \begin{align}
       0 \,\le \, & F(Q_\mu) - F\left(g(t) \right)
        + \int_0^1 \left( v_t\circ g(t)\right)(s) \, \big(Q_\mu(s)- \left(g(t)\right)(s) \big) \d s \\
        =\, &\mathcal F(\mu)-\mathcal F(\gamma_t) +\int_{\R\times\R} v_t(x) \, (y-x) \d \pi(x,y),
    \end{align}
    where $\pi\coloneqq (g(t),Q_\mu)_\#\lebesgue_{(0,1)}$.
    By Theorem~\ref{prop:Q}, the plan
    $\pi$ is optimal between $\gamma_t$ and $\mu$, see also \cite[Thm.~2.18]{Vil03}.
    By \cite[Thm.~16.1(i),(ii)]{M2023}, we also know that $\pi$ is unique, so \eqref{need} yields that $v_t\in -\partial \F(\gamma_t)$. 
\end{proof}

By the same arguments as in the proof of Theorem~\ref{thm:L2_representation}, we have the following corollary concerning invariant subsets of $\C(0,1)$ of $F$-flows.

\begin{cor} \label{remark:GeneralizedRangeCondition}
Let $D \subset \C(0, 1)$ be a closed subset
and let $F \colon L_2(0,1) \to (-\infty,\infty]$
be a proper, convex and lsc function. 
 Assume that for all (small) $\varepsilon > 0$, the resolvent $J_{\varepsilon}^{\partial F}$ maps $D$ into itself.
Then, the solution of the Cauchy problem \eqref{eq:cauchy}
starting in $g_0 \in D$ fulfills 
$g(t) \in D$ for all $t \ge 0$.
\end{cor}
\section{Flows of MMD with Distance Kernel}\label{sec:mmd} 
In this paper, we are mainly interested in Wasserstein gradient flows 
of MMD functionals 
$\mathcal F_\nu: \P_2(\R) \to [0, \infty)$ 
for the negative distance kernel. After introducing these functionals, we will define an
associated functional $F_\nu: L_2(0,1) \to \R$ with
$F_\nu(Q_\mu) = \mathcal F_\nu(\mu)$. More precisely, we will extend 
$F_\nu$ from $\mathcal C(0,1)$ to $L_2(0,1)$ such that its subdifferential can be
easily computed.


MMDs are defined with respect to kernels $K\colon\R^d \times \R^d \to \R$.
In this paper, we are interested in the negative distance kernel 
\begin{equation} \label{eq:riesz}
K(x,y) \coloneqq - |x-y|,
\end{equation}
which is symmetric and \emph{conditionally positive definite} of order one.
Then we define
    \begin{align}         \label{eq:DK2}
        \text{MMD}_K^2(\mu, \nu) &\coloneqq 
        \frac12 \int_{\R^d \times \R^d} K(x,y) \d \mu(x) \d \mu(y) 
        -
        \int_{\R^d \times \R^d} K(x,y) \d \mu(x) \d \nu(y) \\
        & \quad +
        \frac12 \int_{\R^d \times \R^d} K(x,y) \d \nu(x) \d \nu(y).
\end{align}
The square root of the above formula defines a distance on $\P_2(\R^d)$ for 
many kernels of interest including
the negative distance kernel \eqref{eq:riesz}.
In particular, we have that
$\text{MMD}_K(\mu, \nu) \ge 0$ with equality if and only if $\mu = \nu$.
Fixing the target measure $\nu$, the third summand becomes a constant
and we may consider the \emph{MMD functional}
\begin{equation}  \label{eq:dis-decomp_all}
    \F_\nu(\mu)    
    \coloneqq \frac12 \int_{\R^d \times \R^d} K(x,y) \d \mu(x) \d \mu(y) 
    -
    \int_{\R^d \times \R^d} K(x,y) \d \mu(x) \d \nu(y). 
\end{equation}
The first summand is known as \emph{interaction energy}, while
 the second one is called \emph{potential energy} of 
 $V_\nu \coloneqq \int_{\R^d} K(\cdot,y) \,\d\nu(y)$.

 In the following, we are exclusively interested in $d=1$ and the negative distance kernel \eqref{eq:riesz}.
Note that the MMD of the negative distance kernel 
is also known as energy distance \cite{szekely2002,szekely_energy} and that in one dimension we have a relation to the Cramer distance
$$
\mathrm{MMD}_K^2(\mu,\nu)=\int_{-\infty}^\infty (R_\mu(x)-R_\nu(x))^2 \d x,
$$ 
see \cite{mmd_energy_eq}.
More precisely, we will deal with Wasserstein gradient flows of
\begin{equation}    \label{eq:dis-decomp}
 \F_\nu(\mu) = - \frac12 \int_{\R \times \R} |x-y| \d \mu(x) \d \mu(y) 
 + \int_{\R \times \R} |x-y| \d \mu(x) \d \nu(y) .
\end{equation}
In dimensions $d \ge 2$, neither the interaction energy nor
the whole $\text{MMD}_K^2$ functional with the negative distance kernel are $\lambda$-convex along geodesics, see \cite{HGBS2022},
so that Theorem~\ref{thm:WGF} does not apply.
We will see that this is different on the real line. 
Note that in 1D, but not in higher dimensions,
$\lambda$-convexity along geodesics implies the stronger property of $\lambda$-convexity along so-called generalized geodesics. 

Next, we propose an associated functional of \eqref{eq:dis-decomp} on 
the \emph{whole space} $L_2(0,1)$, 
which is determined on $\mathcal C(0,1)$ by $\mathcal F_\nu (\mu) = F_\nu(Q_\mu)$.
We consider the functional $F_{\nu} \colon L_2(0, 1) \to \R$ given via 
\begin{align} \label{eq:Fnu}
    F_{\nu}(u)
    &\coloneqq \int_{0}^{1} \left( (1 - 2 s)  u(s) 
    + \int_{0}^{1} | u(s) - Q_{\nu}(t) | \d{t} \right) \d{s}\\
    &=
    \int_{0}^{1} j\left(s,u(s) \right) \d{s},
\end{align}
where $j \colon (0, 1) \times \R \to \R$ is given by 
\begin{equation}\label{def_H}
    j(s, u) \coloneqq (1 - 2 s) u 
    + H(u), \quad 
    H(u) \coloneqq \int_{0}^{1} | u - Q_{\nu}(t) | \d{t}.
\end{equation}

Indeed, by the following lemma the functional $F_{\nu}$
is associated with $\mathcal F_\nu$.

\begin{lemma}\label{lem:extended_fun}
For $\F_\nu$ in \eqref{eq:dis-decomp},
the functional $F_\nu \colon L_2(0,1) \to \R$ defined by \eqref{eq:Fnu}
fulfills $F_\nu(Q_\mu)=\F_\nu(\mu)$ for all $\mu\in\P_2(\R)$.
\end{lemma}
\begin{proof}
The following calculation was derived in \cite[Lemma 1]{HBGS2023}. For convenience, we restate it here: for all $\mu \in \P_2(\R)$ it holds
\begin{align}   
    &\text{MMD}_K^2(\mu, \nu)
    = - \frac12 \int_{\R \times \R} |x - y| \,
    (\d \mu(x) - \d \nu(x))(\d \mu(y) - \d \nu(y))
    \\
    &=
    - \frac12 \int_0^1 \int_0^1 
    |Q_\mu (s)- Q_\mu(t)| 
    - 2 |Q_\mu(s)-Q_\nu(t)| + |Q_\nu(s)-Q_\nu(t)| \, \d s \, \d t
    \\
    & = \int_0^1 \int_t^1
    Q_\mu(t)-Q_\mu(s) + Q_\nu(t)-Q_\nu(s) \, \d s \, \d t
    + \int_0^1 \int_0^1 |Q_\mu(s) - Q_\nu(t) | \, \d s \, \d t
    \\[-0pt]
    & = \int_0^1 \int_t^1
    Q_\mu(t) + Q_\nu(t) \, \d s \, \d t
    - \int_0^1 \int_0^s Q_\mu(s) + Q_\nu(s) \,\d t \, \d s
    + \int_0^1 \int_0^1 |Q_\mu(s) - Q_\nu(t) | \, \d s \, \d t
    \\
    &= \int_0^1 
    \Bigl((1-2s) (Q_\mu(s) + Q_\nu(s)) 
    + \int_0^1 |Q_\mu(s)-Q_\nu(t)| \, \d t \Bigr) \, \d s.
\end{align} 
Finally noticing that $\int_0^1 (1-2s) Q_\nu(s) \d{s} = -\frac{1}{2}\int_{\R \times \R} |x-y| \d \nu(x) \d \nu(y)$ yields the claim.
\end{proof}

By the next lemma, whose proof is given in the appendix, the functional $F_{\nu}$ has further desirable properties.

\begin{lemma} \label{lemma:FnuGamma0}
    The functional $F_{\nu} \colon L_2(0, 1) \to \R$ in \eqref{eq:Fnu} is convex and continuous. 
\end{lemma}

The subdifferential of $F_\nu$ can be computed explicitly.
First note that by working in $L_2$, our subdifferential $\partial F_{\nu}(u)$ is always nonempty, in contrast to the $\P_2$-subdifferentials $\partial \F_{\nu}(\mu)$, \emph{cf.}\ \cite[Theorem 5.1]{BCDP15}.
Second, we emphasize that the extension of our functional's domain to the \emph{whole space} $L_2(0,1)$ crucially enables the explicit representation of the subdifferential. We explicitly avoid the restriction to the smaller domain $\C(0,1)$, as this would enlarge the subdifferential and conceal its explicit form in the general case, \emph{cf.}\ \cite[Proposition 2.10]{BCDP15}.
This explicit form will be vital for the remaining sections of this paper.

\begin{lemma}[Subdifferential of $F_{\nu}$] \label{thm:subdiffF_nu}
    For $u \in L_2(0, 1)$, it holds
    \begin{align} \label{eq:subdiff}
    \partial F_{\nu}(u) &= \big\{ f \in L_2(0,1): \\
    & \quad \; \;
    f(s) \in 2 \big[ R_{\nu}^-\big(u(s)\big), R_{\nu}^+\big(u(s)\big)\big] - 2s
    ~\text{for a.e.} \; s \in (0, 1) \big\}.
    \end{align}
    In particular, we have $\dom(\partial F_\nu) = L_2(0,1)$.
\end{lemma}

\begin{proof}
It holds $f \in \partial F_{\nu}(u)$ if and only if
    \begin{equation*}
        f(s) \in \partial \big[ j(s, \cdot) \big]\big(u(s)\big)
        \quad \text{for a.e. } s \in (0, 1),
    \end{equation*}
see, e.g., \cite[Cor.~1B]{R1971} or \cite[Thm.~10.39]{SGGHL2009}.
Now,
\begin{equation} \label{sum}
    \partial \big[ j(s, \cdot) \big](u)
    = (1 - 2 s) + \partial H(u),
\end{equation}
so that it remains to consider the second summand.
Recall that for the convex, one-dimensional function $H$, we have
$\partial H (u) = [D_-H(u),D_+H(u)]$ 
with the one-sided derivatives 
\begin{equation*}
    D_{\pm}H(u)
    \coloneqq \lim_{\lambda \downarrow 0} \frac{H(u \pm \lambda) - H(u)}{\lambda}.
\end{equation*}
Let $h_t(u) \coloneqq |u- Q_\nu(t)|$ in the definition \eqref{def_H} of $H$. Then we conclude by Lebesgue's dominated convergence theorem that
 \begin{align}
     D_{\pm } H(u)
     &=
     \lim_{\lambda \downarrow 0} 
     \frac{H(u \pm \lambda) - H(u)}{\lambda}
    =
     \lim_{\lambda \downarrow 0}
     \int_0^1 
     \frac{h_t(u\pm \lambda) - h_t(u)}{\lambda}
     \d t
     \\
     &=
     \int_0^1 \lim_{\lambda \downarrow 0}
     \frac{h_t(u\pm \lambda) - h_t(u)}{\lambda}
     \d t          
     = 
     \int_0^1
     D_{\pm} h_t(u) 
     \d t,
 \end{align}
and consequently
 \begin{equation}
     \partial H(u)
     =
     \Bigl[
     \int_0^1 D_- h_t(u) \d t,
     \int_0^1 D_+ h_t(u) \d t
     \Bigr].
 \end{equation}
 Next, we have 
 \begin{align*}
 D_- h_t(u) &= D_+ h_t(u) = 1, \qquad \; \; \; \text{if } u > Q_{\nu}(t),\\
 D_- h_t(u) &= D_+ h_t(u) = -1, \qquad \, \text{if } u < Q_{\nu}(t),\\
 D_- h_t(u) &= -1, \; D_+ h_t(u) = 1, \quad \text{if } u = Q_{\nu}(t),
 \end{align*}
so that we obtain by \eqref{eq:bounds} and \eqref{eq:bounds_1} the value
$$ 
\int_0^1 D_- h_t(u) \d t 
= 
\int_0^{R_\nu^- (u)} 1 \d{t}
+
\int_{R_\nu^- (u)}^{R_\nu^+ (u)} - 1 \d{t}
+
\int_{R_\nu(u)^+}^{1} -1 \d{t}
=
2 R_\nu^- (u) - 1,
$$
and similarly, $\int_0^1 D_+ h_t(u) \d t = 2 R_\nu^+ (u) - 1$.
This implies 
\begin{equation} \label{helper1}
\partial H(u) = 2[ R_\nu^- (u), R_\nu^+ (u)] - 1,
\end{equation}
and by \eqref{sum} finally
\begin{equation*}
   \partial \big[ j(s, \cdot) \big](u)
   = 2 \big[ R_{\nu}^-\big(u(s)\big), R_{\nu}^+\big(u(s)\big)\big] - 2s.\qedhere 
\end{equation*}
\end{proof}

By the following lemma, $J_{\varepsilon}^{\partial F_\nu}$ fulfills the invariance condition from Theorem~\ref{thm:L2_representation}.

\begin{lemma}    \label{l:RC-C}
    Let $F_\nu$ 
    be defined by \eqref{eq:Fnu}. Then,
    $J_{\varepsilon}^{\partial F_\nu}$ maps $\C(0,1)$ into itself
    for all $\varepsilon >0$.
\end{lemma}

\begin{proof}
Let $\varepsilon >0$ be arbitrarily fixed and $h \in \C(0,1)$.
By Theorem~\ref{general} and Lemma~\ref{lemma:FnuGamma0}, we have that
$\ran(I + \frac{\vareps}{2} \partial F_\nu) = L_2(0, 1)$, so we obtain the existence of $u \in L_2(0,1)$ 
fulfilling $h \in \left(I + \frac{\vareps}{2} \partial F_\nu\right) u$.
The explicit representation of $\partial F_\nu$ in Lemma~\ref{thm:subdiffF_nu} yields
\begin{equation} \label{to_hold}
    u(s) + \vareps [R_\nu^-(u(s)), R_\nu^+(u(s))] ~\ni~ h(s) + \vareps s \hspace*{3mm} \text{ for a.e. } s \in (0,1).
\end{equation} 
We have to show that $u \in \C(0, 1)$.
Since $h + \vareps (\cdot) \in \C(0,1)$, 
there exists a null set $\NN \subset (0,1)$ 
such that outside of $\NN$, $h + \vareps (\cdot)$ is increasing and \eqref{to_hold} holds.
Assume that $u \notin \C(0,1)$.
Then, there exist $s_1,s_2 \in (0,1)$ outside of $\NN$ with $s_1 < s_2$ 
and $u(s_1) > u(s_2)$.
But since $R_{\nu}^-(u(s_1)) - R_{\nu}^+(u(s_2)) = \nu( (u(s_2), u(s_1) ) \ge 0$, it follows that
\begin{align*}
    h(s_1) + \vareps s_1 \ge u(s_1) + \vareps R_\nu^- (u(s_1))
     > u(s_2) + \vareps R_\nu^+ (u(s_2))
     \ge h(s_2) + \vareps s_2,
\end{align*}
contradicting that $h + \vareps (\cdot)$ 
is increasing outside of $\NN$, and the proof is finished. 
\end{proof}


Combining the results from Lemmas~\ref{lemma:FnuGamma0}
and~\ref{l:RC-C}, we can apply Theorem~\ref{thm:L2_representation} to $F_\nu$.
By Proposition~\ref{prop:conv-geo}, the properties of $F_\nu$ carry over to $\mathcal F_\nu$, thus Theorem~\ref{thm:WGF} applies to $\F_\nu$. Together, 
we obtain the following theorem.

\begin{theorem} \label{main_mmd}
Let $\mathcal F_\nu$ and $F_\nu$ be defined 
by \eqref{eq:dis-decomp} and \eqref{eq:Fnu}, respectively, and let $\mu_0 \in \P_2(\R)$.
Then the Cauchy problem
\begin{align}\label{eq:cauchy2}
    \begin{cases}
        \partial_t g(t) \in -  \partial F_\nu (g(t)),   \quad  t \in (0,\infty), \\
         g(0) = Q_{\mu_0},
    \end{cases}
\end{align}
has a unique strong solution $g \colon [0, \infty) \to \mathcal C(0, 1)$,
and the associated curve
$\gamma_t \coloneqq (g(t))_\# \Lambda_{(0,1)}$
is the unique Wasserstein gradient flow of $\mathcal F_\nu$ with $\gamma(0+) = (Q_{\mu_0})_{\#} \lebesgue_{(0,1)}$.
\end{theorem}
Note that due to Lemma \ref{l:RC-C}, there is no need to enforce the cone constraint $g(t) \in \C(0,1)$ as usually done via indicator functionals, \emph{cf.}\ \cite[Sec.\ 2.2]{BCDP15}. Instead, this is automatically given via \eqref{eq:form}.

Finally, let us briefly have a look at the results of the paper \cite{BCDP15} (based on \cite{B2013}) concerning only the interaction energy part of the MMD functional.

\begin{remark} \label{rem:carrillo}
The authors of \cite{BCDP15} showed a similar result as Theorem~\ref{main_mmd}, but only for the interaction energy part $\mathcal E(\mu)=\frac12\int_{\R^d\times\R^d}K(x,y)\d\mu(x)\d\mu(y)$ of the MMD \eqref{eq:DK2}.
More precisely, they derived two equivalent criteria for a curve $\gamma\colon(0,\infty)\to\P_2(\R)$ to be a Wasserstein gradient flow with respect to $\mathcal E$.
The first characterization is that the distributions functions $R_t\coloneqq R_{\gamma(t)}^+$ solve
\begin{equation*}
    \partial_t R_t(x) + \partial_x (R_t^2(x) - R_t(x))
    = 0    
\end{equation*}
subject to some minimum entropy condition. 
Second, this can be characterized via quantile functions. 
That is, $\gamma$ is a Wasserstein gradient flow with respect to $\mathcal E$
if and only if its quantile functions $Q_t\coloneqq Q_{\gamma(t)}$ solve the $L_2$-subgradient inclusion
\begin{equation*}
    -\frac{\d}{\d t} Q_t \in \partial E(Q_t), \quad\text{where}\quad E(g)=\begin{cases}\int_{0}^{1} (1 - 2 s) g(s) \d{s},&\text{if }g\in \C(0,1),\\
    \infty,&\text{otherwise,}
    \end{cases}
\end{equation*}
for almost every $t > 0$.
Based on this representation, they can explicitly compute the Wasserstein gradient flow with respect to $\mathcal E$ via the quantile function $Q_{\gamma(t)}(s) = Q_{\gamma(0)}(s) + t (2 s - 1)$ and prove that $\gamma(t)$ is absolutely continuous for all $t>0$.
These results were extended in \cite{CDEFS2020}, where the authors consider gradient flows for the functional $\mathrm{MMD}_K^2\colon\P_2(\R)\times\P_2(\R)\to\R$ and the related PDE system. In contrast to our setting, these gradient flows are defined on $\P_2(\R)\times\P_2(\R)$, i.e., they consider two interacting measures (species), while on our side, the second entry is fixed as the target measure $\nu$.
Using again the embedding of the Wasserstein space into $L^2(0,1)$, they prove that gradient flows exist and that they remain absolutely continuous whenever the initial measures are absolutely continuous.
Stability of stationary states of singular interaction energies and differentiable potentials on $\R$ with bounded, compactly supported initial data is studied in \cite{FR2011}.
\end{remark}
To our best knowledge, the explicit form \eqref{eq:subdiff} of the $L_2$-subdifferential of the (associated) MMD functional \emph{in presence of the potential energy part} is new to the literature, see the above Remark \ref{rem:carrillo}. It has been substantial for proving the invariance result of Lemma \ref{l:RC-C} used in Theorem \ref{main_mmd}, and will stay so for
our following considerations, including (pointwise) solution formulas, the characterization of invariant subsets, and the numerical calculation of the flow. We like to note that the addition of the attraction towards the target measure $\nu$, in combination with our \emph{non-smooth} kernel $K$, facilitates a far greater control and flexibility of the MMD flow, making it more accessible to applications, e.g., in generative modeling \cite{HHABCS2023, HWAH2023}. This includes present work of the authors to apply the $1$D MMD flow in a high-dimensional setting for image generation; for preliminary results see \cite{DCFS2025}.

\section{Explicit Solution of the Cauchy Problem} \label{sec:explicit}
Having the sub\-differential of $F_\nu$ in \eqref{eq:subdiff} in mind,
we first aim to find a solution of the pointwise Cauchy problem, i.e., for fixed $s\in (0,1)$ we are interested in
\begin{equation}
    \label{eq:pw-ode}
    \begin{cases}
        \partial_t g_s(t)
        \in
        2 s - 2 [R_{\nu}^- (g_s(t)), R_{\nu}^+ (g_s(t))], \quad \text{for a.e. } \; t > 0,
        \\
        g_s(0) = Q_{\mu_0}(s),
    \end{cases}
\end{equation}
satisfying $g_s(t) = [g(t)] (s)$.
Since the proposed method works in the same way
for the left- and right-continuous version of the CDF,
we use $R_\nu^\pm$ to indicate one of them.

\begin{lemma}[Pointwise solution]
    \label{lem:pt-ode}
    Let $\nu, \mu_0 \in \P_2(\R)$.
    For any fixed continuity point $s \in (0,1)$ of $Q_\nu$, 
    the curve
    \begin{equation}
        \label{eq:expl-sol}
        g_s(t)
        \coloneqq
        \begin{cases}
            \Phi_{\nu, s}^{-1}\big(2 t \big),
            &
            t < T_s,
            \\
            Q_\nu(s),
            &
            \text{otherwise},
        \end{cases}
    \end{equation}
    is a strong solution of \eqref{eq:pw-ode},
    where $T_s \coloneqq \frac{1}{2} \Phi_{\nu, s}\left(Q_{\nu}(s) \right)$ and
    $$
    \Phi_{\nu, s}(x) 
    \coloneqq 
    \int_{Q_{\mu_0}(s)}^x \left(s - R_{\nu}^{\pm}(z) \right)^{-1} \d z
    $$ 
    with $x \in [\min\{ Q_{\mu_0}(s), Q_{\nu}(s)\}, \max\{Q_{\mu_0}(s), Q_{\nu}(s)\}]$.
    In particular,
    the inclusion in \eqref{eq:pw-ode} holds true
    for every $t \in (0, \infty)$,
    where $\dot g_s(t)$ exists.
    \end{lemma}

    \begin{proof}
1.        First, assume that $s \in (0,1)$ is \emph{any} point such that  $Q_{\mu_0}(s) < Q_{\nu}(s)$. Since $R^- \le R^+$ and by
        the Galois inequality \eqref{eq:galois}, we have
        for any $z < Q_{\nu}(s)$ that $R_{\nu}^{\pm}(z) < s$. 
        Hence $\left(s - R_{\nu}^{\pm}(z) \right)^{-1}$ exists and
        $\Phi_{\nu, s}(x)$ is finite for all $x < Q_{\nu}(s)$. 
        For $x= Q_\nu(s)$ we distinguish two cases.
        \\[1ex]
        \emph{Case 1:} $\Phi_{\nu, s}(Q_{\nu}(s)) = \infty$.
        Fix any $\hat{x} < Q_{\nu}(s)$. Then, $\Phi_{\nu, s} : [Q_{\mu_0}(s), \hat{x}] \to [0, \infty)$ is absolutely continuous, and its derivative exists a.e. and is given by 
        $$
        \Phi_{\nu, s}'(x) = \left(s - R_{\nu}^{\pm}(x) \right)^{-1} > 0, \quad \text{for a.e. } x \in [Q_{\mu_0}(s), \hat{x}].
        $$
        Hence, the function $\Phi_{\nu, s} : [Q_{\mu_0}(s), \hat{x}] \to [0, \Phi_{\nu, s}(\hat{x})]$ is invertible, and by a result of Zareckii \cite{BC2009,nathanson1955}, its inverse 
        is also absolutely continuous with derivative
        $$
        (\Phi_{\nu, s}^{-1})'(t) = \frac{1}{\Phi_{\nu, s}'(\Phi_{\nu, s}^{-1}(t))} 
        = s - R_{\nu}^{\pm}(\Phi_{\nu, s}^{-1}(t)), \quad \text{for a.e. } t \in [0,\Phi_{\nu, s}(\hat{x})].
        $$
        Hence, by definition \eqref{eq:expl-sol}, the function $g_s$ is absolutely continuous on $[0,\frac12 \Phi_{\nu, s}(\hat{x}))$, and it holds 
        \begin{equation}
        \dot g_s(t) = 
        2 s - 2 R_{\nu}^{\pm} (g_s(t)), \quad \text{for a.e. } t \in [0,\frac12 \Phi_{\nu, s}(\hat{x})).
    \end{equation}
    Considering the above arguments, 
    we notice that $g_s$ is differentiable at $t$
    if and only if $g_s(t)$ is a continuity point of $R_\nu^{\pm}$,
    i.e.\ $R_\nu^+(g_s(t)) = R_\nu^-(g_s(t))$.
    Since $\Phi_{\nu, s}(\hat{x}) \uparrow \infty$ for $\hat{x} \uparrow Q_\nu(s)$, we see that $g_s$ is a strong solution of \eqref{eq:pw-ode}.
    \\[1ex]
    \emph{Case 2:} $\Phi_{\nu, s}(Q_{\nu}(s)) < \infty$. Then above arguments with $\hat{x} = Q_\nu(s)$ show that $g_s$ is absolutely continuous on $[0,T_s)$, and it holds 
        \begin{equation}
        \dot g_s(t) = 
        2 s - 2 R_{\nu}^{\pm} (g_s(t)), \quad \text{for a.e. } t \in [0,T_s).
    \end{equation}
    By the continuity of $\Phi_{\nu, s}^{-1} : [0, 2 T_s] \to [Q_{\mu_0}(s), Q_\nu(s)]$ and by construction of $g_s$, we conclude that $g_s$ is continuous in $T_s$. Since $g_s \equiv Q_\nu(s)$ on $[T_s, \infty)$, it is absolutely continuous on the whole half-line $[0,\infty)$.  
    On the one hand,
    \eqref{eq:galois} immediately yields $s \le R_\nu^+(Q_\nu(s))$,
    and on the other hand,
    its negation $s > R_\nu^+(x) \ge R_\nu^-(x)$ for $x < Q_\nu(s)$
    implies $s \ge R_\nu^-(Q_\nu(s))$
    due to the left-continuity of $R_\nu^-$.
    For this reason,
    we finally have
       \begin{equation}
        \dot g_s(t) = 0 \in 
        2 s - 2 [ R_{\nu}^{-} (Q_\nu(s)), R_{\nu}^{+} (Q_\nu(s)) ] \quad \text{for all } t \in [T_s, \infty),
    \end{equation}    
    which means that $g_s$ is a strong solution of \eqref{eq:pw-ode}.  
    \\[1ex]
2.     
Next, assume that $s \in (0,1)$ is a \textit{continuity} point of $Q_\nu$ such that
$Q_{\mu_0}(s) > Q_{\nu}(s)$.
Then Galois inequality \eqref{eq:galois} together with the fact that $s$ is a continuity point of $Q_\nu$ implies for $z > Q_\nu(s)$ that $R_{\nu}^{\pm}(z) > s$. Now, the rest follows analogously as in the first part. 
\\[1ex]
3. For the case $Q_{\mu_0}(s) = Q_{\nu}(s)$, we clearly get the constant solution, and the proof is done.
\end{proof}

Next
we show that the curve $g: [0,\infty) \to \mathcal C(0,1)$ is differentiable almost everywhere
and solves the original $L_2$ problem \eqref{eq:cauchy2}.

\begin{theorem}[$L_2$ solution]    \label{thm:pw-l2}
    The family 
    \begin{equation} \label{son_pontwise}
    \{g_s \text{ in } \eqref{eq:expl-sol} \colon \, s \in (0,1) \; \text{ is a continuity point of } \; Q_\nu \}
    \end{equation} 
    strongly solves \eqref{eq:cauchy2}
    via $[g(t)](s) \coloneqq g_s(t)$.
\end{theorem}

\begin{proof}
    Since \eqref{eq:expl-sol} is a strong solution of \eqref{eq:pw-ode}, it holds for a.e. $t > 0$ that
    \begin{equation*}
        \lvert \dot g_s(t) \rvert
        \le
        \lvert 2s - 2 R_{\nu}^\pm (g_s(t)) \rvert
        \le 4.
    \end{equation*}
   Since $g_s$ is absolutely continuous in $t$, this implies
   that $g_s: [0,\infty) \to \R$ is Lipschitz continuous for a.e. fixed $s \in (0,1)$, and therefore, $g:[0,\infty) \to L_2(0,1)$ is  Lipschitz continuous. In particular, $g$ is differentiable at a.e. $t>0$, i.e., for any differentiability point $\hat{t} > 0$ outside a null set, the sequence of functions
   $$
   \partial_t g(\hat{t}) \coloneqq \lim_{h\to 0} \frac{g(\hat{t} + h) - g(\hat{t})}{h}
   $$
   converges in $L_2(0,1)$.
   In particular, this holds true for any
   \textit{positive} zero sequence $(h_n)_{n \in \N}$. Then, there exists a subsequence $(h_{n_k})_k$ and a null set $N_{\hat{t}} \subset (0,1)$ such that
   \begin{equation*}
   [\partial_t g(\hat{t})](s) = \lim_{k\to \infty} \frac{g_s(\hat{t} + h_{n_k}) - g_s(\hat{t})}{h_{n_k}} \quad \text{for all } s\in (0,1)\setminus \NN_{\hat{t}}.
   \end{equation*}
   Now, fix $s \in (0,1)\setminus N_{\hat{t}}$ which is also a continuity point of $Q_\nu$. By Theorem~\ref{theorem:BrezisRegularity}, the strong solution $g_s$ of \eqref{eq:pw-ode} solves
   \begin{equation*}
        \frac{d^+}{dt} g_s(t)
        \in
        2 s - 2 [R_{\nu}^- (g_s(t)), R_{\nu}^+ (g_s(t))]
   \end{equation*}
   even for \textit{every} $t > 0$, where $\frac{d^+}{dt}$ denotes the right derivative. Altogether, for the choice $t = \hat{t}$, it follows
   \begin{align*}
    [\partial_t g(\hat{t})](s) &= \lim_{k\to \infty} \frac{g_s(\hat{t} + h_{n_k}) - g_s(\hat{t})}{h_{n_k}}  \\
    &= \frac{d^+}{dt} g_s(\hat{t})
       ~ \in ~
        2 s - 2 [R_{\nu}^- (g_s(\hat{t})), R_{\nu}^+ (g_s(\hat{t}))].
   \end{align*}
   This proves that $g$ is a strong solution of \eqref{eq:cauchy2}, and we are done.
\end{proof}

Next, we apply the explicit solution formula
\eqref{eq:expl-sol}
to describe the flow 
from an arbitrary starting measure $\mu_0 \in \mathcal P_2(\R)$
to a discrete measure $\nu$. 

\begin{table}
    \centering\footnotesize
    \begin{tabular}{llll}
         \toprule
         & $Q_{\mu_0}(s) \le Q_\nu(s)$& $Q_{\mu_0}(s) \ge Q_\nu(s)$& 
         \\
         \midrule
         $\ell_s$
         &  $W_{\ell_s-1} < s < W_{\ell_s}$
         & $W_{\ell_s-1} < s < W_{\ell_s}$
         &
         \\
         $k_s$
         & $x_{k_s} \le Q_{\mu_0}(s) < x_{k_s + 1}$
         & $x_{k_s - 1} < Q_{\mu_0}(s) \le x_{k_s}$
         &
         \\
         $x_{s,j}$
         & $x_{k_s + j}$
         & $x_{k_s - j}$
         & $j = 1, \dots, \lvert \ell_s - k_s \rvert$
         \\
         $R_{s,j}$
         & $W_{k_s + j}$
         & $W_{k_s - j - 1}$& $j = 0, \dots, \lvert \ell_s - k_s \rvert - 1$
         \\
         \bottomrule
    \end{tabular}
    \caption{Quantities from  Corollary~\ref{cor:pt-meas-tar},
    where $W_k \coloneqq \sum_{j=1}^k w_j$ and $x_0 \coloneqq -\infty$, $x_{n+1} \coloneqq \infty$.}
    \label{tab:pt-meas-tar}
\end{table}

\begin{cor}[Point measure target]    \label{cor:pt-meas-tar}
    Let 
    \begin{equation}    \label{eq:pt-meas}
    \nu \coloneqq \sum_{j=1}^n w_j \delta_{x_j}    
\end{equation}
with weights $0 < w_j \le 1$ fulfilling $\sum_{j=1}^n w_j = 1$
and $x_1 < x_2 < \dots < x_n$.
Then the strong solution of \eqref{eq:cauchy2}     is given by
    \begin{equation*}
        [g(t)](s)
        \coloneqq
        \begin{cases}
            Q_{\mu_0}(s) + 2 \, (s - R_{s,0}) \, t,
            & t \in [t_{s,0}, t_{s,1}),
            \\
            x_{s,j} + 2 \, (s - R_{s,j}) \, (t - t_{s,j}),
            & t \in [t_{s,j}, t_{s,j+1}),
            \\
            Q_\nu(s),
            & t \ge t_{s,\lvert \ell_s - k_s \rvert} ,
        \end{cases}
    \end{equation*}
    where
     \begin{equation*}
            t_{s,0} \coloneqq 0, 
            \quad
            t_{s,1} \coloneqq \tfrac{x_{s,1} - Q_{\mu_0}(s)}{2(s - R_{s,0})},
            \quad
            t_{s,j+1} \coloneqq t_{s,j} + \frac{x_{s,j+1} - x_{s,j}}{2(s - R_{s,j})},
    \end{equation*}
    for $j \in \{ 1, \ldots, | \ell_s - k_s | - 1 \}$, 
    and the quantities $\ell_s,k_s,x_{s,j},R_{s,j}$  are given in Table~\ref{tab:pt-meas-tar}.
\end{cor}

\begin{proof}
    Recall that $g_s$ is either monotonically increasing or monotonically decreasing
    until it hits the target quantile function $Q_\nu(s)$.
    The values $R_{s,j}$ in Table~\ref{tab:pt-meas-tar}
    denote the values of the constant plateaus,
    which the solution passes.
    The values $x_{s,j}$ denote the locations
    where the derivative jumps.
    Solving \eqref{eq:pw-ode} piecewise gives the stated explicit form.
\end{proof}

In \figref{fig:discrete_nu} we illustrate that from a pointwise view,
the flow $g_s$ changes only linearly over time.
However, since all quantities in Table~\ref{tab:pt-meas-tar},
especially the time points $t_{s,j}$
corresponding to the discontinuities of the derivative,
depend non-linearly on $s \in (0,1)$,
the actual evolution of the quantile function itself,
becomes highly non-linear. This is illustrated in the following example.

\begin{figure}[t]
    \centering
    \includegraphics[height=3.4cm]{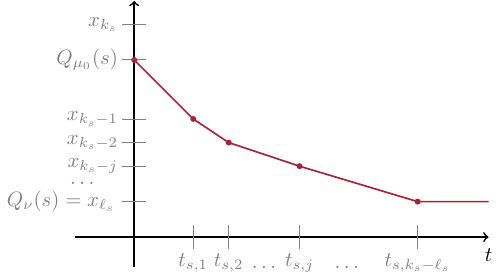}
    \includegraphics[height=3.4cm]{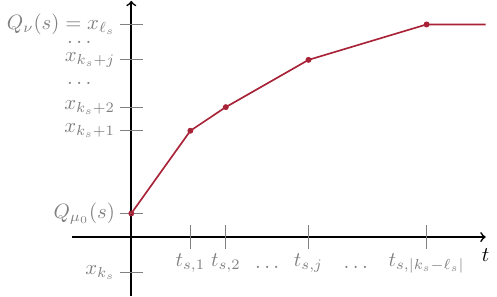}
    \caption{The map $[0, \infty) \to \R$, $t \mapsto [g(t)](s)$ for fixed $s \in (0, 1)$ with $Q_{\mu}(s) < Q_{\nu}(s)$ \textbf{(left)} and  $Q_{\mu}(s) > Q_{\nu}(s)$ \textbf{(right)}.
    The slopes of the affine linear pieces, determined by $s - R_{s, j}$, decrease for increasing $t$, that is, for increasing $j$.}
    \label{fig:discrete_nu}
\end{figure}

\begin{example}   \label{example:TwoToThree}
    Let
    \begin{equation*}
        \mu_0
        \coloneqq 
        \tfrac{1}{3}  \delta_{-1}
        + \tfrac{1}{3}  \delta_{1/2} 
        + \tfrac{1}{3}  \delta_{2} 
        \quad\text{and}\quad
        \nu
        \coloneqq 
        \tfrac{1}{4} \delta_0 
        + \tfrac{3}{4}  \delta_1.
    \end{equation*}
    The evolution of the quantile function can be computed as
    \begin{equation}
        [g(t)](s)
        =
        \begin{cases}
            \min\{2st - 1, 0\},
            & 
            s \in (0, \frac14),
            \\
            \eqref{eq:excl-int},
            &
            s \in (\frac14, \frac13),
            \\
            \min\{(\frac12 - \frac{t}4) + 2st, 1 \},
            &
            s \in (\frac13,\frac23),
            \\
            \max\{(2-2t) - 2st, 1 \},
            &
            s \in (\frac23, 1),
        \end{cases}
    \end{equation}
    where
    \begin{equation}
        \label{eq:excl-int}
        [g(t)] \big\vert_{(\frac14, \frac13)}(s)
        = 
        \begin{cases}
            2 s t - 1, 
            & 
            s \le t_1^{-1}(t),
            \\
            \frac{1}{4 s} - (\frac{t}{2} + 1) + 2 s t , 
            & 
            t_1^{-1}(t) \le s \le t_2^{-1}(t),
            \\
            1, 
            & 
            t_2^{-1}(t) \le s,
        \end{cases}
    \end{equation}
    with the monotonically decreasing, continuous functions 
    $t_1(s) \coloneqq t_{s,1}$
    and $t_2(s) \coloneqq t_{s,2}$.
    Except for $s \in (1/4,1/3)$,
    the quantile function is piecewise linear,
    which corresponds to uniform measures and point measures.
    For $s \in (1/4, 1/3)$, it becomes non-linear.
    The flow $g(t)_\# \Lambda_{(0,1)}$ is illustrated in Figure~\ref{fig:pt-target}.
    \end{example}

\begin{figure}[H]
    \includegraphics[width=.22\textwidth,page=1]{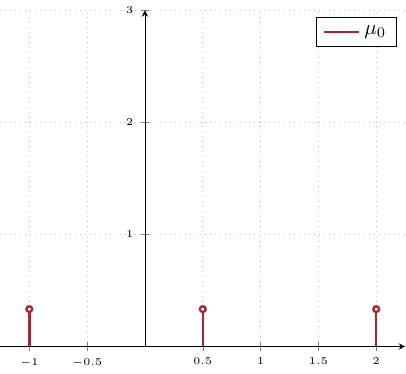}
    \hspace{0.2cm}
    \includegraphics[width=.22\textwidth,page=2]{imgs/Three_Dirac_to_Two_Diracs/Three_Dirac_to_Two_Diracs.pdf}
    \hspace{0.2cm}
    \includegraphics[width=.22\textwidth,page=4]{imgs/Three_Dirac_to_Two_Diracs/Three_Dirac_to_Two_Diracs.pdf}
    \hspace{0.2cm}
    \includegraphics[width=.22\textwidth,page=6]{imgs/Three_Dirac_to_Two_Diracs/Three_Dirac_to_Two_Diracs.pdf}
    \\[2ex]
    \includegraphics[width=.22\textwidth,page=14]{imgs/Three_Dirac_to_Two_Diracs/Three_Dirac_to_Two_Diracs.pdf}
    \hspace{0.2cm}
    \includegraphics[width=.22\textwidth,page=18]{imgs/Three_Dirac_to_Two_Diracs/Three_Dirac_to_Two_Diracs.pdf}
    \hspace{0.2cm}
    \includegraphics[width=.22\textwidth,page=22]{imgs/Three_Dirac_to_Two_Diracs/Three_Dirac_to_Two_Diracs.pdf}
    \hspace{0.2cm}
    \includegraphics[width=.22\textwidth,page=29]{imgs/Three_Dirac_to_Two_Diracs/Three_Dirac_to_Two_Diracs.pdf}
        \caption{The Wasserstein gradient flow between the point measures
        in Example~\ref{example:TwoToThree}. 
        The gray regions illustrate the density of the absolute continuous part of the flow,
        whereas the red spikes illustrate the discrete part. 
        In the lower row,
        the mass from $\delta_{-1}$ splits up,
        and a portion moves towards $\delta_{1}$.
        The movement speed of this portion, however, significantly decreases
        due to its attraction by $\delta_0$. 
        Figuratively,
        the mass sticks to $\delta_0$
        and can escape only slowly.
        For the corresponding quantile functions, see \figref{fig:pt-target_q}.}
        \label{fig:pt-target}
\end{figure}

\section{Invariant Subsets and Smoothing Properties of \texorpdfstring{$F_\nu$}{F}-Flows} \label{sec:inv}
In this section, we are interested in \textit{invariant subsets of} $\C(0,1)$, i.e., subsets in which $F_\nu$-flows remain once starting there. We also prove a \textit{smoothing} result, where the $F_\nu$-flow immediately becomes more regular than the initial datum.

Note that \cite[p.~131, Prop.~4.5]{B1973} generally characterizes conditions for closed subsets to be invariant. 
However, we take a more refined approach involving the resolvent $J_{\varepsilon}^{\partial F_\nu}$, the exponential formula \eqref{eq:form}, and our explicit calculation \eqref{eq:subdiff} of $\partial F_\nu$. 
This approach will yield more precise results -- and is not limited to closed subsets.

As a starting point, recall that $J_{\varepsilon}^{\partial F_\nu}$ maps $\C(0,1)$ into itself by Lemma~\ref{l:RC-C}. By \eqref{eq:subdiff} this means: For all $\varepsilon > 0$ and any $h \in \C(0,1)$, there exists $u \in \C(0,1)$ such that
\begin{equation} \label{eq:RC-start}
    u(s) + \vareps [R_\nu^-(u(s)), R_\nu^+(u(s))] ~\ni~ h(s) + \vareps s \hspace*{3mm} \text{for a.e. } s \in (0,1).
\end{equation} 

To simplify the following arguments and notations,
we identify an (equivalence class) $v \in \C(0,1)$ with its unique left-continuous and increasing version, i.e., $v \coloneqq Q_{\mu_v}$, where $\mu_v \coloneqq v_{\#} \lebesgue_{(0,1)}$, such that $v$ is uniquely defined \textit{everywhere} on $(0,1)$ (and not only up to a null set)
\footnote{Further, we make the agreement to always exclude denominators of (difference) quotients being zero, implicitly.}.

With the above identification of $v \in \C(0,1)$ with its quantile function, the following lemma shows that \eqref{eq:RC-start} holds \textit{everywhere} on $(0,1)$.

\begin{lemma}\label{lem:helper}
    Let $\nu \in \P_2(\R)$ be arbitrary.
    Then, for any $h \in \C(0,1)$ there exists $u \coloneqq J_{\varepsilon /2}^{\partial F_\nu} h \in \C(0,1)$ such that
    \begin{equation} \label{eq:RC-everywhere}
    u(s) + \vareps [R_\nu^-(u(s)), R_\nu^+(u(s))] ~\ni~ h(s) + \vareps s \hspace*{3mm} \text{ for all } \hspace*{3mm} s \in (0,1).
\end{equation} 
\end{lemma}

The proof is given in the appendix.
\\

Now, let $- \infty \le a \le b \le \infty$ and $L > 0$. We will study the following closed subsets of $\mathcal C(0,1)$:
\begin{enumerate}
\item[I)] bounded quantile functions:
\begin{align*}
D_{[a,b]} &:=\{Q_\mu \in \C(0,1) : \overline{Q_\mu(0,1)} \subseteq [a,b]\} =
\{Q_\mu \in \C(0,1):  \supp \mu \subseteq [a,b] \} ,
\end{align*}
\item[II)] quantile functions admitting a lower $L$-Lipschitz condition:
\begin{align*}
D_{L}^- &:= \{Q_\mu \in \C(0,1) : L \le \frac{Q_\mu(s_1) - Q_\mu(s_2)}{s_1-s_2} ~\text{ for all } s_1,s_2 \in (0,1)\},
\end{align*}
\item[III)]   $L$-Lipschitz continuous quantile functions:
\begin{align*}
D_{L}^+ &:=\{Q_\mu \in \C(0,1) : 
\frac{Q_\mu(s_1) - Q_\mu(s_2)}{s_1-s_2} \le L  ~\text{ for all } s_1,s_2 \in (0,1)\}.
\end{align*}
\end{enumerate}
We will also consider the following subset which is not closed in $\C(0,1)$:
\begin{enumerate}
    \item[IV)] continuous quantile functions:
    \begin{align*}
    D_{c} &:=\{Q_\mu \in \C(0,1) : Q_\mu \text{ is continuous}\}\\
    & = \{Q_\mu \in \C(0,1) : R_\mu^+ \text{ is strictly increasing on } \overline{(R_\mu^+)^{-1}(0,1)} \}\\
    &= \{Q_\mu \in \C(0,1):  \supp \mu \text{ is convex} \}.
    \end{align*}
\end{enumerate}	
 
Note that the  subset of quantile functions of absolutely continuous measures 
with respect to the Lebesgue measure is also not closed in $\mathcal C(0,1)$.

Next, we provide some examples of probability measures $\nu \in \P_2(\R)$ and discuss to which of the above sets their quantile functions belong.

\begin{example}
    \begin{enumerate}[i)]
        \item
        For any atomic measure $\nu$ with a finite number of atoms we have $Q_{\nu} \in D_{[a, b]}$ for some $a, b \in \R$ and $Q_{\nu} \not\in D_{L}^{-}$ for any $L > 0$.
        
        \item 
        If $\nu$ defines a uniform distribution on $[a, b]$, then its quantile function is $s \mapsto a + s (b - a)$.
        Hence $Q_{\nu} \in D_{b - a}^{+} \cap D_{b - a}^{-} \cap D_{[a, b]} \cap D_c$.
        
        \item
        For the normal distribution $\nu \sim \mathcal{N}(\mu, \sigma^2)$, we have $\supp \nu = \R$, so $Q_\nu \in D_c$, and that $Q_{\nu}(s) = \mu + \sqrt{2} \sigma \erf^{-1}(2 s - 1)$ is not Lipschitz continuous.
        However, $Q_{\nu} \in D_L^{-}$ with $L \coloneqq (2 \pi \sigma^2)^{\frac{1}{2}}$.
        The same is true for the Laplace distribution or a mixture of two normal distributions.
        \item
        The exponential, Pareto and folded norm distributions, see \eqref{eq:folded_norm}, have a support that is unbounded in only one direction, so their quantile functions belong to $D_c \cap D_{[a, \infty)}$ for some $a \in \R$.
        As above, these quantile functions do not belong to $D_L^+$, but to $D_L^{-}$.
    \end{enumerate}
    \medskip\noindent
     Note that if $\mu \in \P_2(\R)$ is absolutely continuous and $Q_{\mu} \not\in D_{[a, b]}$ for any $a, b \in \R$, then we can not have $Q_{\mu} \in D_L^+$ for any $L > 0$.
\end{example}

By the following Proposition~\ref{l:DLOW-eq}, the above sets $D^\pm_L$ can be described in terms of CDFs 
instead of quantile functions, see \figref{fig:Q_Lipschitz}.

\begin{figure}[t]
    \centering
    \includegraphics[width=.49\textwidth, page=1, valign=c]{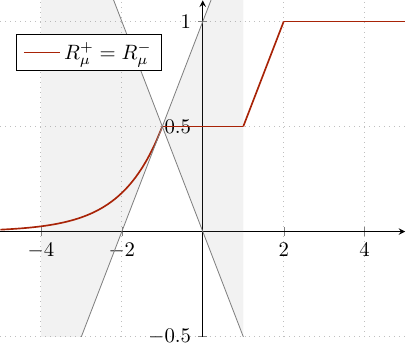}
    \includegraphics[width=.49\textwidth, page=2, valign=c]{imgs/Q_Lipschitz}
    \caption{ \textbf{Left:} 
     $R_{\mu}^+$ is $\frac{1}{2}$-Lipschitz continuous, since there is a \enquote{double cone} (white) with extremal rays of slope $\frac{1}{2}$ such that if it is centered at any point of the graph of $R_{\mu}^+$, then the   graph has no points within this double cone.
    \textbf{Right:} 
    $Q_{\mu}$ admits a lower $2$-Lipschitz condition, since the white double cone with extremal rays of slope $2$ contains the whole graph of $Q_\mu$.}
    \label{fig:Q_Lipschitz}
\end{figure}

\begin{proposition}\label{l:DLOW-eq}
  Let $Q_\mu \in \C(0,1)$. Then the following holds true:
  \begin{itemize}
      \item[i)] 
       $Q_\mu \in D_{L}^-$ if and only if $R_\mu = R_\mu^+$ is $\frac1L$-Lipschitz continuous, i.e., 
              \begin{equation}
     \frac{R_\mu(x) - R_\mu(y)}{x-y} \le \frac{1}{L} \quad
     \text{ for all } x, y \in \R.
   \end{equation} 
In this case, $\mu$ is absolutely continuous, and the above holds iff its density $f_\mu$ fulfills
   $$
   f_\mu \le \frac{1}{L} \quad\text{a.e. on } \R.
   $$
      \item[ii)]  $Q_\mu \in D_{L}^+$ if and only if $R_\mu^+$ admits a lower $\frac1L$-Lipschitz condition on $\overline{(R_\mu^+)^{-1}(0,1)}$, i.e.,
  \begin{equation}
     \frac{1}{L} \le \frac{R_\mu^+(x) - R_\mu^+(y)}{x-y} \quad\text{ for all } x, y \in \overline{(R_\mu^+)^{-1}(0,1)}.
   \end{equation}  
   If $\mu$ is absolutely continuous, then the above holds iff its density $f_\mu$ fulfills
   $$
   \frac{1}{L} \le f_\mu  \quad\text{a.e. on } \conv (\supp \mu ).
   $$
  \end{itemize}
  \end{proposition}

The proof is given in the appendix.

The following proposition says that the support of flows cannot escape the convex hull of the support of the target measure $\nu$ once they start there.

\begin{proposition}[Invariance of $D_{[a,b]}$]\label{p:rangeDab}
    Let $\nu \in \P_2(\R)$ be such that $\overline{\conv} (\supp  \nu) \subseteq [a,b]$.
    Then,  $J_{\varepsilon}^{\partial F_\nu}$ maps $D_{[a,b]}$ into itself for all $\vareps > 0$.
    Hence,
    the solution of the Cauchy problem \eqref{eq:cauchy2}
    starting in $g_0 \in D_{[a,b]}$ fulfills 
    $g(t) \in D_{[a,b]}$ for all $t \ge 0$.
\end{proposition}
  
\begin{proof}
    Let $h \in D_{[a,b]}$ and consider a solution $u\in \C(0,1)$ of \eqref{eq:RC-everywhere}.
    W.l.o.g., let $a,b$ be finite (otherwise there is nothing to show). Assume there exists $s^* \in (0,1)$ such that
    \begin{equation*}
        u(s^*) < a \le \min (\supp  \nu).
    \end{equation*}
    Then, it holds $R_{\nu}^-(u(s^*)) = R_{\nu}^+(u(s^*)) = 0$ and by \eqref{eq:RC-everywhere} that
    \begin{equation*}
        a > u(s^*) =
        h(s^*) + \vareps s^* \ge a + \vareps s^*,
    \end{equation*}
    which is a contradiction. Now, assume there is $s^* \in (0,1)$ such that
    \begin{equation*}
        u(s^*) > b \ge \max (\supp  \nu).
    \end{equation*}
    Then, we obtain $R_{\nu}^-(u(s^*)) = R_{\nu}^+(u(s^*)) = 1$ and by \eqref{eq:RC-everywhere} further
    \begin{equation*}
        b < u(s^*) =
        h(s^*) + \vareps (s^*-1) \le b + \vareps (s^*-1),
    \end{equation*}
    again a contradiction. Together, we have proved $u(s) \in [a,b]$ for \textit{every} $s \in (0,1)$, which shows $u \in D_{[a,b]}$. The remaining claim follows by Corollary~\ref{remark:GeneralizedRangeCondition}.
\end{proof}

  In order to prove invariance results for $D_L^-$ and $D_L^+$, 
  we define\footnote{Notice that in the definitions of $L_{\text{low}}(Q_\mu)$ and ${\rm Lip}(Q_\mu)$, \textit{``for all''} can be replaced by \textit{``for almost all''}, yielding the same constants since $Q_\mu$ is left-continuous.} 
  for given $Q_\mu \in \C(0,1)$ the largest lower Lipschitz constant 
  \begin{align*}
  L_{\text{low}}(Q_\mu) := \max \{L \ge 0: \frac{Q_\mu(s_1) - Q_\mu(s_2)}{s_1-s_2} \ge L ~ \text{ for all } s_1,s_2 \in (0,1)\} \ge 0,
  \end{align*}
  where $L_{\text{low}}(Q_\mu) = 0$ is explicitly allowed. Note that by Proposition~\ref{l:DLOW-eq}, it holds $L_{\text{low}}(Q_\mu) > 0$ if and only if $R_\mu$ is Lipschitz continuous with constant $L_{\text{low}}(Q_\mu)^{-1}$.
  Further, consider the usual smallest Lipschitz constant
  \begin{align*}
  {\rm Lip}(Q_\mu) := \min \{L \ge 0: \frac{Q_\mu(s_1) - Q_\mu(s_2)}{s_1-s_2} \le L ~ \text{ for all } s_1,s_2 \in (0,1)\} \le \infty.
  \end{align*}
  By Proposition~\ref{l:DLOW-eq}, it holds $0 < {\rm Lip}(Q_\mu) < \infty$ if and only if $R_\mu^+$ admits a lower ${\rm Lip}(Q_\mu)^{-1}$-Lipschitz condition on $\overline{(R_\mu^+)^{-1}(0,1)}$.\\

  Now, we can formulate an invariance and smoothing property of $F_\nu$-flows. In addition, we can accurately describe how the Lipschitz constants of the $F_\nu$-flow evolve over time. We start with the lower constant.
  
  \begin{theorem}\label{thm:regularisation-Lip}
  Let $\nu \in \P_2(\R)$ with $L_{{\rm low}}(Q_\nu) > 0$. In particular, it is assumed that $R_\nu$ is Lipschitz continuous.
  Consider any initial value $g_0 = Q_{\mu_0} \in \C(0,1)$, where $L_{{\rm low}}(g_0) = 0$ is explicitly allowed.
   Then, the strong solution $g:[0,\infty) \to \C(0,1)$ of the Cauchy problem \eqref{eq:cauchy2} enjoys for all $t > 0$ the smoothing property
  \begin{equation}\label{eq:smoothing1}
     L_{{\rm low}}( g(t) )  \ge  L_{{\rm low}}( g_0 ) \cdot e^{-\frac{2t}{L_{{\rm low}}( Q_\nu )}}
    + L_{{\rm low}}( Q_\nu ) \cdot (1 - e^{-\frac{2t}{L_{{\rm low}}( Q_\nu) }}) > 0.
    \end{equation}
    In particular, the CDF $R_{\gamma_t}$ of the associated Wasserstein gradient flow $\gamma_t$ is Lipschitz continuous for any $t > 0$ with Lipschitz constant
    \begin{equation}\label{eq:smoothing2}
    {\rm Lip}(R_{\gamma_t}) \le \Big( L_{{\rm low}}( g_0 ) \cdot e^{-\frac{2t}{L_{{\rm low}}( Q_\nu )}}
    + L_{{\rm low}}( Q_\nu ) \cdot (1 - e^{-\frac{2t}{L_{{\rm low}}( Q_\nu )}}) \Big)^{-1} < \infty.
    \end{equation}
    \end{theorem}

    \begin{proof}
   i)
   First, let $h \in \C(0,1)$ and $\vareps > 0$. By \eqref{eq:RC-everywhere} and since $R_\nu^+ = R_\nu^- = R_\nu$ is continuous, there exists a solution $u\in \C(0,1)$ of
    \begin{equation*}
        h(s) + \vareps s = u(s) + \vareps R_{\nu}(u(s))
        \quad \text{for all } s \in (0, 1).
    \end{equation*}
    By the Lipschitz continuity of $R_\nu$, it holds for all $s_1, s_2 \in (0,1)$, $s_1 > s_2$, that
    \begin{align*}
    \frac{u(s_1) - u(s_2)}{s_1 - s_2}
    &= \frac{h(s_1) - h(s_2)}{s_1 - s_2} + \vareps - \vareps \frac{R_\nu (u(s_1)) - R_\nu(u(s_2))}{s_1 - s_2}  \\
    &\ge L_{\text{low}}(h) + \vareps - \vareps \frac{1}{L_{\text{low}}(Q_\nu)}\cdot \frac{u(s_1) - u(s_2)}{s_1 - s_2},
    \end{align*}
    which yields
    \begin{align*}
    \frac{u(s_1) - u(s_2)}{s_1 - s_2} \ge  \frac{L_{\text{low}}(h) + \vareps}{q}, \quad q \coloneqq  1 + \frac{\vareps}{L_{\text{low}}(Q_\nu)}.
    \end{align*}
    This means $ L_{\text{low}}(u) \ge \frac{L_{\text{low}}(h) + \vareps}{q} > 0$. (Notice the improvement $L_{\text{low}}(u) > 0$, even if $L_{\text{low}}(h) = 0$.) We have proved for any $h \in \C(0,1)$ and any $\vareps > 0$ that
    \begin{equation}
     L_{\text{low}}( (I + \frac{\vareps}{2}\partial F_\nu)^{-1} h) 
     \ge 
     \frac{L_{\text{low}}(h) + \vareps}{q}.   
    \end{equation}
    Now, fix $n \in \N$ and observe inductively that
    \begin{align*}
     L_{\text{low}}( (I + \frac{\vareps}{2}\partial F_\nu)^{-n} g_0) 
     &\ge \frac{L_{\text{low}}((I + \frac{\vareps}{2}\partial F_\nu)^{-(n-1)} g_0)}{q} + \frac{\vareps}{q} 
     \\
     &\ge
     \frac{L_{\text{low}}((I + \frac{\vareps}{2}\partial F_\nu)^{-(n-2)} g_0)}{q^2} 
     + \frac{\vareps}{q^2} 
     + \frac{\vareps}{q}
     \ge \ldots\\
     &\ge 
     \frac{L_{\text{low}}( g_0)}{q^n} + \vareps \sum_{k=1}^n \frac{1}{q^k}\\
     &= 
     \frac{L_{\text{low}}( g_0)}{q^n}
     +  \frac{\vareps}{q} \cdot \frac{1 - q^{-n}}{1 - q^{-1}}\\
     &= \frac{L_{\text{low}}( g_0)}{q^n}
     + L_{\text{low}}(Q_\nu) (1- q^{-n}).
    \end{align*}
    Now, for $t > 0$ and choosing $\vareps \coloneqq \frac{2t}{n}$, it follows
    \begin{align*}
    L_{\text{low}}( (I + \frac{t}{n}\partial F_\nu)^{-n} g_0)
    &\ge \frac{L_{\text{low}}( g_0)}{(1 + \frac{2t}{n L_{\text{low}}(Q_\nu)})^n}
     + L_{\text{low}}(Q_\nu) (1- (1 + \frac{2t}{n L_{\text{low}}(Q_\nu)})^{-n})\\
    &\overset{n \uparrow \infty}{\longrightarrow} L_{\text{low}}( g_0) \cdot e^{-\frac{2t}{L_{\text{low}}(Q_\nu)}}
    + L_{\text{low}}(Q_\nu) \cdot (1 - e^{-\frac{2t}{L_{\text{low}}(Q_\nu)}}).
    \end{align*}
    Finally, the solution $g: [0,\infty) \to \C(0,1)$ of the Cauchy problem \eqref{eq:cauchy2} is given by the exponential formula (and $L_2$-limit) \eqref{eq:form} which implies 
    \begin{equation}
     L_{\text{low}}( g(t) )   \ge  L_{\text{low}}( g_0) \cdot e^{-\frac{2t}{L_{\text{low}}(Q_\nu)}}
    + L_{\text{low}}(Q_\nu) \cdot (1 - e^{-\frac{2t}{L_{\text{low}}(Q_\nu)}}) > 0
    \end{equation}
    for all $t > 0$, which concludes the proof.
\end{proof}
    
Note that even if the {\rm initial} measure satisfies $L_{{\rm low}}( g_0 ) = 0$ (e.g., atomic measures), the Lipschitz continuity of the {\rm target} CDF, i.e., $L_{{\rm low}}( Q_\nu ) > 0$, is enough to force the  flow's CDFs to immediately become Lipschitz continuous for any $t > 0$, and the Lipschitz constants exponentially improve for $t \to \infty$ towards the target Lipschitz constant $L_{{\rm low}}( Q_\nu )^{-1}$ as described in \eqref{eq:smoothing2}.


 The following example demonstrates that the bound \eqref{eq:smoothing2} of the Lipschitz constants in Theorem~\ref{thm:regularisation-Lip} is sharp. Furthermore, it features an explicit flow from a Dirac point towards a target uniform distribution, where the flow stays uniformly distributed at any given time point.

\begin{example} \label{ex:delta0_to_uniform_0}
 Consider as target measure $\nu$ the uniform distribution on $[0,1]$, and as initial measure $\mu_0 \coloneqq \delta_0$ the Dirac measure at $0$. Then, for the corresponding CDFs, we get that
      \begin{equation}
          R_\nu (x) = 
          \begin{cases}
            0, & ~x \le 0, \\
            x, & ~0 < x < 1 , \\
            1, & ~x \ge 1,
        \end{cases}
      \end{equation}
      is $1$-Lipschitz continuous, whereas
      \begin{equation}
          R_{\mu_0}^{+}(x) = 
          \begin{cases}
            0, & ~x < 0, \\
            1, & ~x \ge 0,
        \end{cases}
      \end{equation}
      is a jump function. Still, the $\F_\nu$-Wasserstein gradient flow $\gamma_t$ immediately regularizes to uniform distributions for $t > 0$ given by
      \begin{equation}
          R_{\gamma_t}(x) = 
          \begin{cases}
            0, & ~x \le 0, \\
            (1-e^{-2t})^{-1} x, & ~0 < x < 1-e^{-2t} , \\
            1, & ~x \ge 1-e^{-2t},
        \end{cases}
      \end{equation}
      which are $(1-e^{-2t})^{-1}$-Lipschitz continuous.
      Thus, the upper bound of the Lipschitz constants \eqref{eq:smoothing2}  is sharp.
    \end{example}

Next, we deal with the upper Lipschitz constant.

    \begin{theorem}\label{thm:regularisation-Lip_1}
    Let $\nu \in \P_2(\R)$ such that $\supp \nu = [a,b]$ and ${\rm Lip}(Q_\nu) < \infty$. Consider an initial value $g_0 = Q_{\mu_0} \in \C(0,1)$ with convex support $\supp {\mu_0} \subseteq [a,b]$ and ${\rm Lip}(g_0) < \infty$.
    In particular, it is assumed that $\supp {\mu_0}, \, \supp {\nu}$ are convex.
    Then, the strong solution $g:[0,\infty) \to \C(0,1)$ of the Cauchy problem \eqref{eq:cauchy2} fulfills for all $t > 0$ the invariance property
    \begin{equation}\label{eq:invariance1}
    {\rm Lip}( g(t) )   \le  {\rm Lip}( g_0) \cdot e^{-\frac{2t}{{\rm Lip}(Q_\nu)}}
    + {\rm Lip}(Q_\nu) \cdot (1 - e^{-\frac{2t}{{\rm Lip}(Q_\nu)}}) < \infty.
    \end{equation}
    In particular, the CDF $R_{\gamma_t}^+$ of the associated Wasserstein gradient flow $\gamma_t$ admits a lower Lipschitz constant on $\overline{(R_{\gamma_t}^+)^{-1}(0,1)}$ for all $t > 0$, and the constant is
    \begin{equation}\label{eq:invariance2}
    L_{\rm low}(R_{\gamma_t}^+) \ge  \Big( {\rm Lip}( g_0) \cdot e^{-\frac{2t}{{\rm Lip}(Q_\nu)}}
    + {\rm Lip}(Q_\nu) \cdot (1 - e^{-\frac{2t}{{\rm Lip}(Q_\nu)}}) \Big)^{-1} > 0.
    \end{equation}
\end{theorem}

  \begin{proof}
  If ${\rm Lip}(Q_\nu) = 0$, then by assumption, it holds $\supp {\mu_0} \subseteq \supp {\nu} = \{a\}$. Hence, the solution of \eqref{eq:cauchy2} is given by $g(t) \equiv a$ for all $t \ge 0$, and trivially satisfies \eqref{eq:invariance1}. So, we can assume that ${\rm Lip}(Q_\nu) > 0$.
    Now, let $h \in D_{[a,b]}$ with ${\rm Lip}(h) < \infty$, and $\vareps > 0$. By Proposition~\ref{p:rangeDab}, there exists a solution $u \in D_{[a,b]}$ of \eqref{eq:RC-everywhere}. 
    Let $s_1, s_2 \in (0,1)$, $s_1 > s_2$. W.l.o.g., we can assume that $u(s_2) < u(s_1)$.  
    Hence, we can choose $\kappa > 0$ small enough such that $a \le u(s_2) < u(s_1) - \kappa \le b$. Now, by \eqref{eq:RC-everywhere}, it holds
    \begin{equation}
        u(s_1) \le h(s_1) + \vareps s_1 - \vareps R_\nu^-(u(s_1)) 
    \end{equation}
    and 
    \begin{equation}
        u(s_2) \ge h(s_2) + \vareps s_2 - \vareps R_\nu^+(u(s_2)). 
    \end{equation}
    By Proposition~\ref{l:DLOW-eq}, $R_\nu^+$ admits a lower ${\rm Lip}(Q_\nu)^{-1}$-Lipschitz condition on $\overline{(R_\nu^+)^{-1}(0,1)} = \supp \nu = [a,b]$, so that
    \begin{align*}
    \frac{u(s_1) - u(s_2)}{s_1 - s_2}
      &\le \frac{h(s_1) - h(s_2)}{s_1 - s_2} + \vareps - \vareps \frac{R_\nu^- (u(s_1)) - R_\nu^+(u(s_2))}{s_1 - s_2}  \\
    &\le \frac{h(s_1) - h(s_2)}{s_1 - s_2} + \vareps - \vareps \frac{R_\nu^+ (u(s_1) - \kappa) - R_\nu^+(u(s_2))}{s_1 - s_2}  \\
    &\le \frac{h(s_1) - h(s_2)}{s_1 - s_2} + \vareps - \vareps \frac{1}{{\rm Lip}(Q_\nu)} \cdot \frac{(u(s_1)-\kappa) - u(s_2)}{s_1 - s_2}.
    \end{align*}
    Letting $\kappa \downarrow 0$ leads to
    \begin{align*}
      \frac{u(s_1) - u(s_2)}{s_1 - s_2}
       \le {\rm Lip}(h) + \vareps -  \frac{\vareps}{{\rm Lip}(Q_\nu)}\cdot \frac{u(s_1) - u(s_2)}{s_1 - s_2},
    \end{align*}
    which yields
    \begin{align*}
    \frac{u(s_1) - u(s_2)}{s_1 - s_2} \le  \frac{{\rm Lip}(h) + \vareps}{q}, \quad q \coloneqq  1 + \frac{\vareps}{{\rm Lip}(Q_\nu)}.
    \end{align*}
    This shows ${\rm Lip}(u) \le \frac{{\rm Lip}(h) + \vareps}{q} < \infty$. In other words, for any $h \in D_{[a,b]}$ with ${\rm Lip}(h) < \infty$ and $\vareps > 0$, it holds 
    \begin{equation}
     {\rm Lip}( (I + \frac{\vareps}{2}\partial F_\nu)^{-1} h) 
     \le 
     \frac{{\rm Lip}(h) + \vareps}{q}.   
    \end{equation}
    Now, setting $h := g_0$, we obtain the desired estimate \eqref{eq:invariance1} in the same lines as in the previous proof, and we are done.
  \end{proof}

We mention that formally, the conditions ${\rm Lip}(Q_\nu) < \infty$ and ${\rm Lip}(Q_{\mu_0}) = \infty$ bring no improvement in the estimate \eqref{eq:invariance1} for ${\rm Lip}(g(t))$.
Further, by the following example, the support assumption in Theorem~\ref{thm:regularisation-Lip_1}
cannot be skipped.

 \begin{example} \label{ex:delta0_to_uniform}
  In our example in the introduction,
    we considered the Wasserstein gradient flow
    from $\mu_0 = \delta_{-1}$ 
    to $\nu = \delta_0$ with ${\rm Lip}(Q_{\mu_0}) = {\rm Lip}(Q_\nu) = 0$. 
    By Corollary~\ref{cor:pt-meas-tar}, the quantile functions  given by  
        \begin{equation} \label{eq:delta-1_to_delta0}
        [g(t)](s)
        = 
        \min \{2st - 1, 0 \},
    \end{equation}
    are piecewise linear and sharply contained in $D_{2t}^+$, see 
    Figure~\ref{fig:qunat}. In particular, the Lipschitz constants cannot be
bounded.
    The corresponding Wasserstein gradient flow reads as
    \begin{equation}\label{eq:initioal}
        \gamma_t 
        =
        \begin{cases}
            \delta_{-1},
            &
            t = 0, 
            \\
            \tfrac{1}{2t} \Lambda_{[-1,-1+2t]},
            &
            0 \le t \le \frac12,
            \\
            \tfrac{1}{2t} \Lambda_{[-1,0]} 
            + \bigl(1 - \tfrac{1}{2t} \bigr) \delta_0,
            &
            \frac12 < t,
        \end{cases}
    \end{equation}
    and was already depicted in Figure~\ref{fig:gradient_flow_dirac_dirac_P2R}.
           
    \begin{figure}[t]
        \centering
        \includegraphics[width=.4\textwidth]{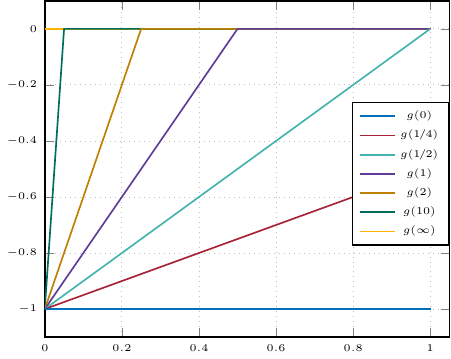}
        \caption{The functions $g(t)$ from \eqref{eq:delta-1_to_delta0} for $t \in \{ 0, \frac{1}{4}, \frac{1}{2}, 1, 2, 10 \}$.
        }\label{fig:qunat}
    \end{figure} 
  
  \end{example}

 As a consequence of Theorems~\ref{thm:regularisation-Lip},~\ref{thm:regularisation-Lip_1} and their proofs, we immediately obtain the following invariance properties of $D_L^-$ and $D_{[a,b]} \cap D_L^+$ with respect to the $F_\nu$-flow. 

   \begin{cor}[Invariance of $D_L^-$ and $D_{L}^+ \cap D_{[a,b]}$]\label{c:rangeDL}
  Let $\nu \in \P_2(\R)$. 
   \begin{enumerate}[i)]
\item   Assume $Q_\nu \in D_{L}^-$. 
    Then, the solution of the Cauchy problem \eqref{eq:cauchy2}
    starting in $g_0 \in D_{L}^-$ fulfills 
    $g(t) \in D_{L}^-$ for all $t \ge 0$.
    More precisely, it holds \eqref{eq:smoothing1}.
    Moreover, $J_{\varepsilon}^{\partial F_\nu}$ maps $D_L^-$ into itself for all $\varepsilon >0$.
    
    In particular,
    if the initial measure $\mu_0$ has a $L_0$-Lipschitz continuous CDF $R_{\mu_0}$, then the flow's CDF $R_{\gamma_t}$ remains Lipschitz (and thus absolutely) continuous with constant $\le \max \{L_0, L^{-1}\}$ for all $t \ge 0$.    
\item
    Assume  $\supp \nu = [a,b]$ and  $Q_\nu \in D_L^+$.
    Then, the solution of the Cauchy problem \eqref{eq:cauchy2}
    starting in $g_0 \in D_{[a,b]} \cap D_L^+$ fulfills 
    $g(t) \in D_{[a,b]} \cap D_L^+$ for all $t \ge 0$.
    More precisely, it holds \eqref{eq:invariance1}.
  \\ 
  Moreover, $J_{\varepsilon}^{\partial F_\nu}$ maps $D_{[a,b]} \cap D_L^+$ into itself for all $\varepsilon >0$.
    
    In particular,
    if the initial  measure $\mu_0$ has convex support $\supp \mu_0 \subseteq [a,b]$ and is in $D_L^+$, then $\gamma_t$ has convex support for all $t \ge 0$ and no 'gaps' of empty mass can form\footnote{ This particular invariance property of $F_\nu$-flows concerning the convexity of supports, we will generalize to arbitrary target measures $\nu \in \P_2(\R)$ further below.}.  
    \end{enumerate}
  \end{cor}

 Note that in part i), no point masses can arise during the flow -- given that the target measure is Lipschitz continuous -- as opposed to the example of Figure~\ref{fig:gradient_flow_dirac_dirac_P2R}.

	Finally, we study the set $D_c$ of continuous quantile functions. By the following lemma,
 it is also invariant with respect to the resolvent $J_{\varepsilon}^{\partial F_\nu}$.
 
		\begin{lemma}\label{l:rangeCont}
			Let $\nu \in \P_2(\R)$ be arbitrary. Then, $J_{\varepsilon}^{\partial F_\nu}$ maps $D_{c}$ into itself for all $\varepsilon >0$.
		\end{lemma}
  
		\begin{proof}
			For $h \in D_c$, we consider a solution $u \in \C(0,1)$ of \eqref{eq:RC-everywhere}. 
			Suppose that $u$ is not continuous at some $s_0 \in (0, 1)$.
			Since $u$ is left-continuous and increasing, we have $\lim_{t \downarrow s_0} u(t) > u(s_0)$. But then, \eqref{eq:RC-everywhere} and the continuity of $h$ imply
			\begin{align*}
			h(s_0) + \vareps s_0
			& = \lim_{t \downarrow s_0}\left(  h(t) + \vareps t \right)
			\ge 
   \lim_{t \downarrow s_0} \left( u(t) + \vareps R_{\nu}^-(u(t)) \right)\\
			& > u(s_0) + \vareps R_{\nu}^+(u(s_0))
			\ge h(s_0) + \vareps s_0,
			\end{align*}
			a contradiction.
			Thus, $u \in D_{c}$.
		\end{proof}
  
		Since $D_c$ is not closed with respect to the $L_2$-norm, Corollary~\ref{remark:GeneralizedRangeCondition} cannot be applied directly. Still, we have the following invariance and monotonicity result. Note that there are no restrictions on the target measure $\nu \in \P_2(\R)$.
		
    \begin{theorem}[Invariance of $D_c$ \& monotonicity of the support]\label{p:rangeDc}
	Let $\nu \in \P_2(\R)$ and $g_0 \in D_{c}$.
        Then the following holds true:
        \begin{enumerate}[i)]
		\item
            The solution $g$ of the Cauchy problem \eqref{eq:cauchy2}
            starting in $g_0 \in D_{c}$ fulfills $g(t) \in D_{c}$ for all $t \ge 0$.
            
		\item
            The ranges fulfill $\overline{g(t_1) (0,1)} \subseteq \overline{g(t_2) (0,1)}$ for all $0 \le t_1 \le t_2$. 
        \end{enumerate}
    \end{theorem}

In other words the theorem says: if the initial measure $\mu_0$ has convex support, then $\supp \gamma_t$ stays convex for all $t\ge 0$, and we have the monotonicity $\supp \gamma_{t_1} \subseteq \supp \gamma_{t_2}$ for all $0 \le t_1 \le t_2$.
  
	    \begin{proof}
	    i) Fix $t > 0$. For the initial datum $g_0 \in D_c$, Lemma~\ref{l:rangeCont} yields	
	    \begin{equation*}
	    g_n \coloneqq (I + \frac{t}{n} \partial F_\nu)^{-n} g_0 \in D_c \quad \text{ for all } n \in \N.
	    \end{equation*}
	    We verify that $(g_n)_{n \in \N}$ fulfills the assumptions of the Arzel\`a-Ascoli theorem.
     \\
	    1) Fix $s \in (0,1)$ and $n \in \N$. By \eqref{eq:RC-everywhere}, it holds for $u \coloneqq (I + \frac{t}{n} \partial F_\nu)^{-1} g_0$ that 
	    \begin{align*}
	    u(s) &\le  g_0(s) + \frac{2t}{n} s - \frac{2t}{n} R_\nu^-(u(s)) \le g_0(s) + \frac{2t}{n} s,
	    \\
	    u(s) &\ge  g_0(s) + \frac{ 2t}{n} s - \frac{2t}{n} R_\nu^+(u(s)) \ge g_0(s) + \frac{2t}{n} s - \frac{2t}{n},
	    \end{align*}
	    so that
	    \begin{equation*}
	    |u(s)| \le |g_0(s)| + \frac{2t}{n} (s + 1).
	    \end{equation*}
	    Iteratively, it follows 
	    \begin{equation*}
	    |g_n(s)| \le |g_0(s)| + 2 t (s + 1).
	    \end{equation*}
	    Since $n \in \N$ was arbitrary, $(g_n)_n$ is pointwise bounded.
     \\
	    2) Fix $0 < s_2 < s_1 < 1$ and $n \in \N$. W.l.o.g. we can assume for $u$ as above that $u(s_2) < u(s_1)$. Since $R_\nu^-(u(s_1)) - R_\nu^+(u(s_2)) \ge 0$, it follows by \eqref{eq:RC-everywhere} that
	    \begin{align*}
	     |u(s_1) - u(s_2)|
	    &\le g_0(s_1) - g_0(s_2) + \frac{2t}{n} (s_1 - s_2) - \frac{2t}{n} \big(R_\nu^-(u(s_1)) - R_\nu^+(u(s_2)) \big)\\
	    &\le |g_0(s_1) - g_0(s_2)| + \frac{2t}{n} |s_1 - s_2|.
	    \end{align*}
	    Again, it follows inductively that
	    \begin{equation*}
	    |g_n(s_1) - g_n(s_2)| \le |g_0(s_1) - g_0(s_2)| + 2t |s_1 - s_2|.
	    \end{equation*}
	    Since $n \in \N$ was arbitrary, $(g_n)_{n \in \N}$ is equicontinuous.
     \\[1ex]
	    Now let $K \subset (0,1)$ be compact. By Arzel\`a-Ascoli's theorem, (for a subsequence) $(g_n)$ converges uniformly on $K$ to a continuous function $\tilde{g} \colon K \to \R$. By the exponential formula \eqref{eq:form}, the $L_2$-limit of $(g_n)$ is already given by $g(t)$. By the uniqueness of the limit, it follows that $g(t)|_K \equiv \tilde{g}$ is continuous on $K$. Since $K$ was an arbitrary compact subset of $(0,1)$, we obtain that $g(t)$ is continuous on $(0,1)$, which means $g(t) \in D_c$ as claimed.
     \\[1ex]
	    ii) For the monotonicity claim, let $0 \le t_1 \le t_2 < \infty$ be arbitrary fixed. 
        The strong solution $g$ of the Cauchy problem \eqref{eq:cauchy2} has the Bochner integral representation
	    \begin{equation*}
	    g(t_2) = \int_{t_1}^{t_2} \partial_t g(t) \d t + g(t_1),
	    \end{equation*}
	    which allows a pointwise evaluation
	    \begin{equation*}
	    g(t_2)(s) = \int_{t_1}^{t_2} \partial_t g(t)(s) \d t + g(t_1)(s) \quad \text{for a.e. } s \in (0,1).
	    \end{equation*}
	    Since $\partial_t g(t) \in -\partial F_\nu(g(t))$ for a.e. $t>0$, and noting that $0 \le R_\nu^- \le R_\nu^+ \le 1$, 
        it follows for a.e. $s \in (0,1)$ that
	    \begin{align*}
	    g(t_2)(s) &\le \int_{t_1}^{t_2} 2s - 2R_\nu^-\big( g(t)(s) \big) \d t + g(t_1)(s) \\
	    &\le 2s(t_2 - t_1) + g(t_1)(s).
	    \end{align*}
	    This yields
	    \begin{equation}\label{eq:left-limit}
	    \lim_{s \downarrow 0} g(t_2)(s) \le \lim_{s \downarrow 0} g(t_1)(s).
	    \end{equation}
	    Analogously, it holds for a.e. $s \in (0,1)$ that
	    \begin{align*}
	    g(t_2)(s) &\ge \int_{t_1}^{t_2} \Big( 2s - 2R_\nu^+\big( g(t)(s) \big) \Big) \d t + g(t_1)(s) \\
	    &\ge 2(s-1) (t_2 - t_1) + g(t_1)(s),
	    \end{align*}
	    which implies
	    \begin{equation}\label{eq:right-limit}
	    \lim_{s \uparrow 1} g(t_2)(s) \ge \lim_{s \uparrow 1} g(t_1)(s).
	    \end{equation}
	    Combining \eqref{eq:left-limit} and \eqref{eq:right-limit}, we have proved that
	    \begin{equation}\label{eq:left-right-limit}
	    \lim_{s \downarrow 0} g(t_2)(s) \le \lim_{s \downarrow 0} g(t_1)(s) \le \lim_{s \uparrow 1} g(t_1)(s) \le \lim_{s \uparrow 1} g(t_2)(s).
	    \end{equation}
	    By part i), we know that $g(t_2) \in D_c$. The intermediate value theorem finally implies that $\overline{g(t_1) (0,1)} \subseteq \overline{g(t_2) (0,1)}$, and we are done.
	    \end{proof}

        \begin{remark}
            Theorem~\ref{p:rangeDc} states that the support of the flow $\gamma_t$ monotonically increases with the time $t$, given that the initial support $\supp \mu_0$ is convex. We leave it as an open problem, whether this still holds in general without the convexity assumption on $\supp \mu_0$.
        \end{remark}

Many of the results of Section \ref{sec:inv} can be derived via the pointwise solution \eqref{eq:expl-sol} of Section \ref{sec:explicit}. This can be found in the Appendix~\ref{app:diff}. But note that the application of the resolvent $J_{\varepsilon}^{\partial F_\nu}$ proves these results \emph{also} for the implicit Euler steps \eqref{impl} used in our numerical section.

\section{Numerical Experiments}\label{sec:numerics}

In this section, we first discuss a backward and a forward Euler scheme
for the numerical computation of the one-dimensional MMD flow with the negative distance kernel and give some examples afterwards.
Note that a direct calculation of the flow via the explicit pointwise formula \eqref{eq:expl-sol} should also be feasible in our case. Yet, we like to stress that Euler schemes are applicable to a broader range of (quantile) Cauchy problems, where pointwise solution formulas might be absent, see e.g., \cite{DRSS2025}, so that we pursue the more general Euler approach.

\subsection{Euler Schemes}
\paragraph{Implicit (backward) Euler scheme.} The minimizing movement scheme \eqref{jko}, see also JKO scheme \cite{JKO1998}, 
$$
\mu_{n + 1} \coloneqq \argmin_{\mu \in \P_2(\R^d)} \Big\{\F_\nu(\mu) + \frac{1}{2 \tau} W_2^2(\mu_{n}, \mu) \Big\}, \qquad \tau > 0,
$$
can be rewritten by 
Theorem~\ref{prop:Q} 
in terms of quantile functions as
\begin{equation} \label{iter}
g_{n+1} = 
\argmin_{g \in \mathcal C(0,1)} \Big\{F_\nu (g) + \frac{1}{2 \tau} \int_0^1 |g - g_{n} |^2 \d s\Big\}, \qquad \tau > 0.
\end{equation}
Because $F_\nu$ is proper, convex and lsc, the solution of this problem is given by (also see \eqref{eq:form})
\begin{equation} \label{impl}
    g_{n + 1} = (I + \tau \partial F_{\nu})^{-1}(g_n)
\end{equation}
which is by Lemma~\ref{thm:subdiffF_nu} and Lemma~\ref{lem:helper} equivalent to
\begin{equation} \label{eq:implicit_Euler_pointwise}
    g_n(s) + 2 \tau s
    \in g_{n + 1}(s) + 2 \tau [R_{\nu}^{-}(g_{n + 1}(s)), R_{\nu}^{+}(g_{n + 1}(s))]
    \qquad \text{for all } s \in (0, 1).
\end{equation}
Note that by Lemma~\ref{lem:helper} the functions $g_n$ and $g_{n + 1}$ 
are quantile functions and thus increasing,
so that $g_n + 2 \tau I$ is strictly increasing.
Figure~\ref{fig:one_step_implicit_Euler} gives a visual intuition for solving \eqref{eq:implicit_Euler_pointwise} using bisection.

\begin{figure}[H]
    \centering
    \includegraphics[width=0.5\textwidth]{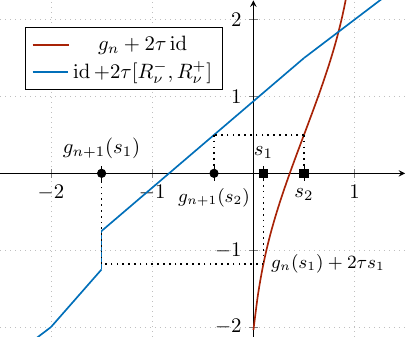}
    \caption{Implicit Euler step visualized for $g_n = Q_{\gamma_n}$ with $\gamma_n \sim \NN(0, 1)$ and $\tau = \tfrac{1}{2}$ and $\nu$ being a mixture of two uniform distributions.
    Starting from any point $s_1 \in (0,1)$ and
    $y \coloneqq (g_n + 2 \tau I)(s_1)$,
    we are searching for $x \coloneqq g_{n+1}(s) \in \R$ 
    such that $y \in x + 2 \tau [R_{\nu}^{-}(x), R_{\nu}^{+}(x)]$.
    This is also illustrated for a point $s_2 \in (0,1)$.
    The next point $x$ can be easily found by the bisection method.
    }
    \label{fig:one_step_implicit_Euler}
\end{figure}

For fixed $\tau > 0$, we perform $n=1,2,\ldots$ implicit Euler steps \eqref{impl} via solving \eqref{eq:implicit_Euler_pointwise}.
By Theorem~\ref{thm:WGF}, we see that $(g_n)_\# \Lambda_{(0,1)}$ is an approximation of the Wasserstein gradient flow 
in the time interval $(n\tau, (n+1)\tau]$.
Further, for fixed $\tau$ and $n \to \infty$, the sequence
$(g_n)_{n \in \N}$ converges weakly to the target $Q_\nu$ in $L_2(0,1)$, see \cite{BC2011,PS2010}, and
the corresponding measures
 $\mu_n \coloneqq (g_n)_\# \Lambda_{(0,1)}$ converge narrowly
to the minimizer $\nu$ of $\mathcal F_\nu$, see \cite[Thm 6.7]{NS2022}. 
The scheme \eqref{impl} resembles a proximal point algorithm,
and for convergence results for more general so-called quasi $\alpha$-firmly nonexpansive mappings, 
we refer to \cite{BS2024}. Finally note that the implicit Euler Scheme \eqref{eq:implicit_Euler_pointwise} works for \emph{arbitrary} target measures $\nu \in \P_2(\R)$ including discrete ones.


\paragraph{Explicit (forward) Euler scheme.} 
For constant step size $\tau > 0$, we will also consider the explicit Euler discretization
\begin{equation*}
    g_{n + 1} =  g_n - \tau \nabla F_{\nu}(g_n), 
\end{equation*}
which is only available if $R^+_\nu = R^-_\nu =: R_\nu$ 
is continuous, so that
\begin{equation*}
    \nabla F_\nu(g)
    = 2 R_\nu \circ g - 2 \id_{(0,1)},
    \qquad g \in L_2(0, 1).   
\end{equation*}
The explicit Euler scheme has the advantage that we do \emph{not} have to solve an inclusion in each step.
However, this method comes with weaker convergence guarantees: neither the local uniform convergence nor the long-term convergence is clear in general as in the implicit case.
Moreover, it might not preserve $\C(0, 1)$, that is, the iterates might not be monotone
\footnote{In our experiments, monotonicity was not preserved for large step sizes, e.g., $\tau = 5$.}.
Nonetheless, in all our numerical examples we chose a step size $\tau = \tfrac{1}{100}$, which preserved monotonicity, and the differences between the implicit and explicit Euler schemes were negligible.

\subsection{Numerical Experiments}\label{sec:numerics_2}
Next, we compare the implicit and explicit Euler schemes for various
combinations of absolutely continuous initial and target measures\footnote{The python code recreating these plots can be found at \url{https://github.com/ViktorAJStein/MMD_Wasserstein_gradient_flow_on_the_line/tree/main}.}.
The following examples are covered by Theorem~\ref{thm:regularisation-Lip}, or Corollary~\ref{c:rangeDL}~i): since $L_{\text{low}}(Q_{\nu}) > 0$ in each figure, we have $L_{\text{low}}(Q_{\gamma_t}) > 0$ for all $t > 0$.

\paragraph{Flow between Gaussians, Laplacians, and uniform distributions.}
    In \figref{fig:Norm_To_Norm_Different_Means}, we plot the MMD flow, 
    where $\mu_0$ and $\nu$ are both Gaussians with different means, but equal variance.
    The behaviors of the implicit and explicit schemes are very similar, and visually, no difference is noticeable. 
    Further, the shape of the normal distribution is not at all preserved during the flow and instead, the densities first spread out, become more flat and then form a peak again, when they meet the mean of the target distribution.

    \noindent In \figref{fig:Laplace_to_Laplace},
    this behavior can also be observed when the initial and the target measure are both Laplacians. Here we can also see that the non-differentiability of the density at its peak is smoothed out and then re-created during the flow.
    
    \noindent In \figref{fig:Uniform-to-Uniform},
    we also see that the MMD flow between two uniform distributions does not stay a uniform distribution itself. Notice, how the mass outside the target's support $\supp \nu = [2,3]$ becomes arbitrarily slim, but never vanishes completely due to Theorem~\ref{p:rangeDc}. 
        
    \noindent In \figref{fig:Norm_To_Norm_Different_Scales}, we plot the MMD flow, where $\mu_0$ and $\nu$ are both Gaussians with different variance, but equal mean.
    Again, the behavior of both discretizations is very similar, and the measures do not stay Gaussian.
    In particular, as Theorem~\ref{p:rangeDc} suggests, the visible support of $\gamma_t$ is increasing, and matches the tails of the target distribution only exponentially slowly.

\begin{figure}
    \centering
    \includegraphics[width=.32\textwidth]{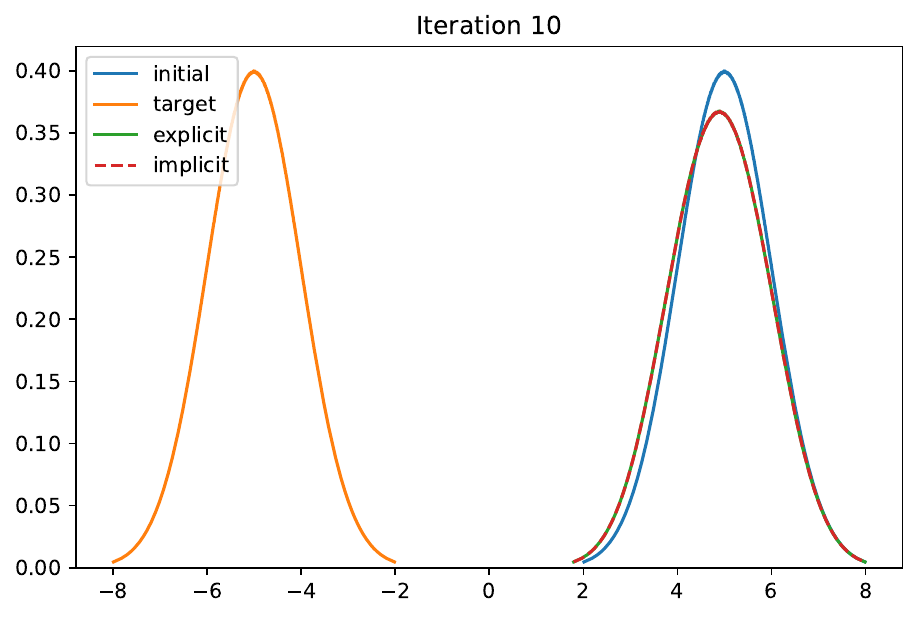}
    \includegraphics[width=.32\textwidth]{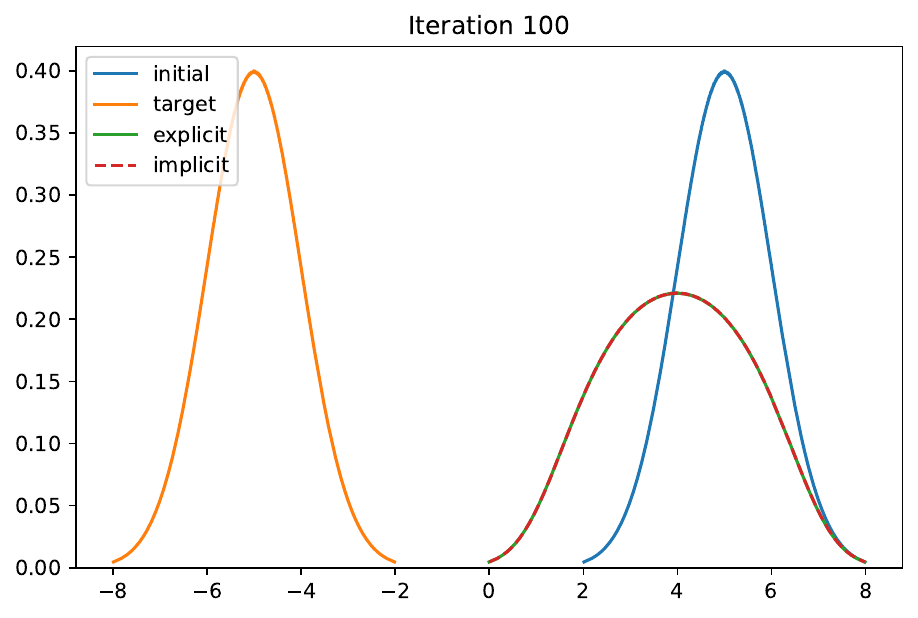}
    \includegraphics[width=.32\textwidth]{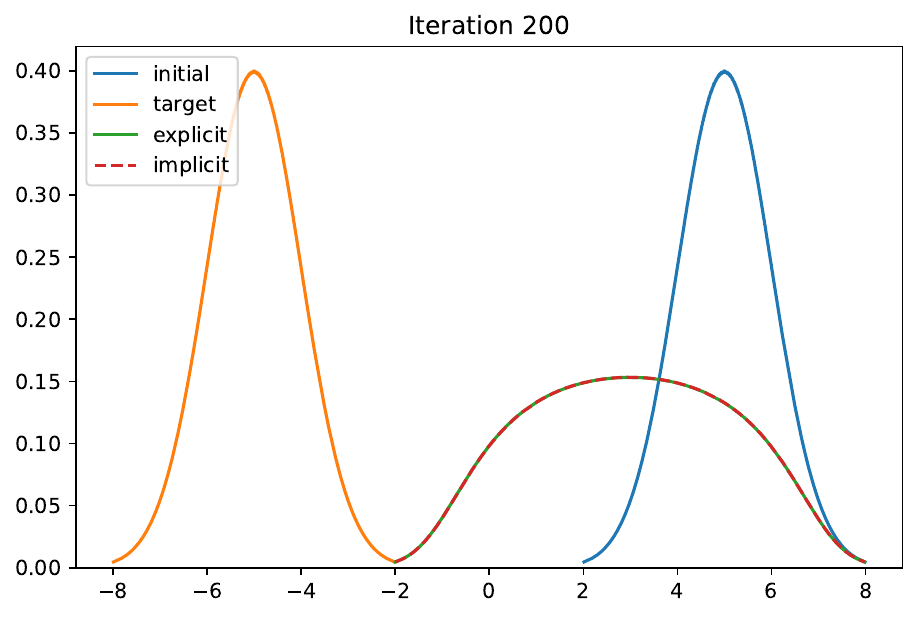}
    \\
    \includegraphics[width=.32\textwidth]{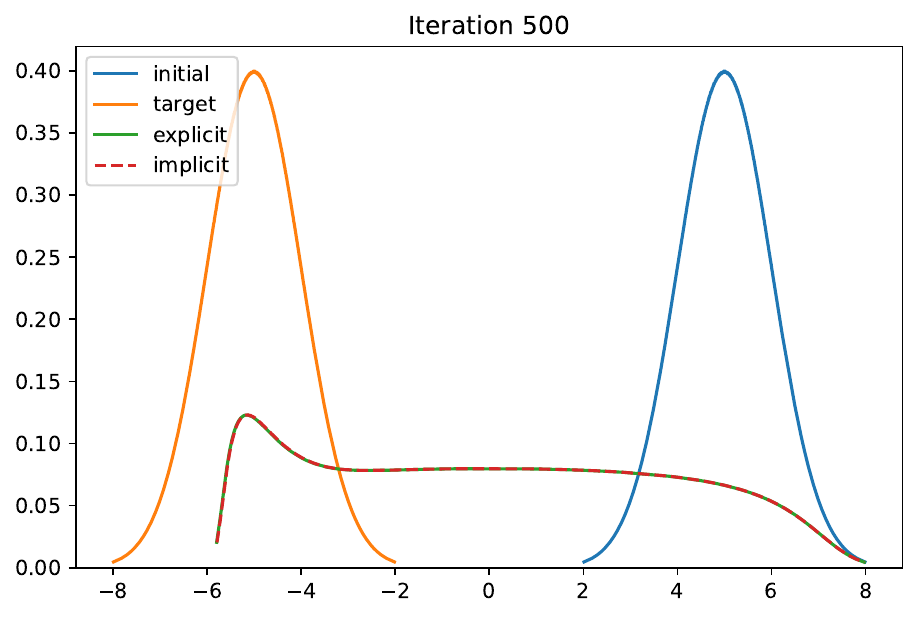}
    \includegraphics[width=.32\textwidth]{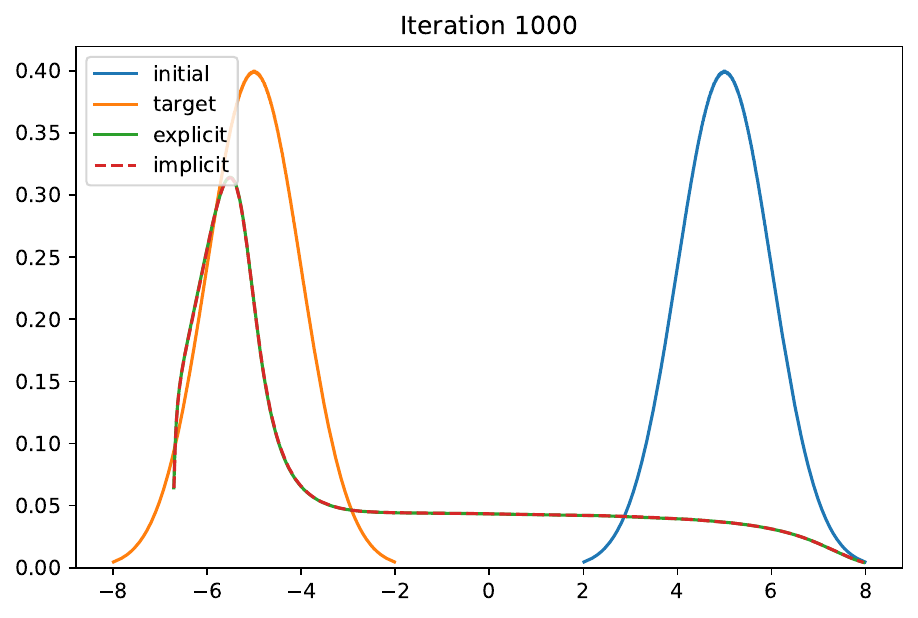}
    \includegraphics[width=.32\textwidth]{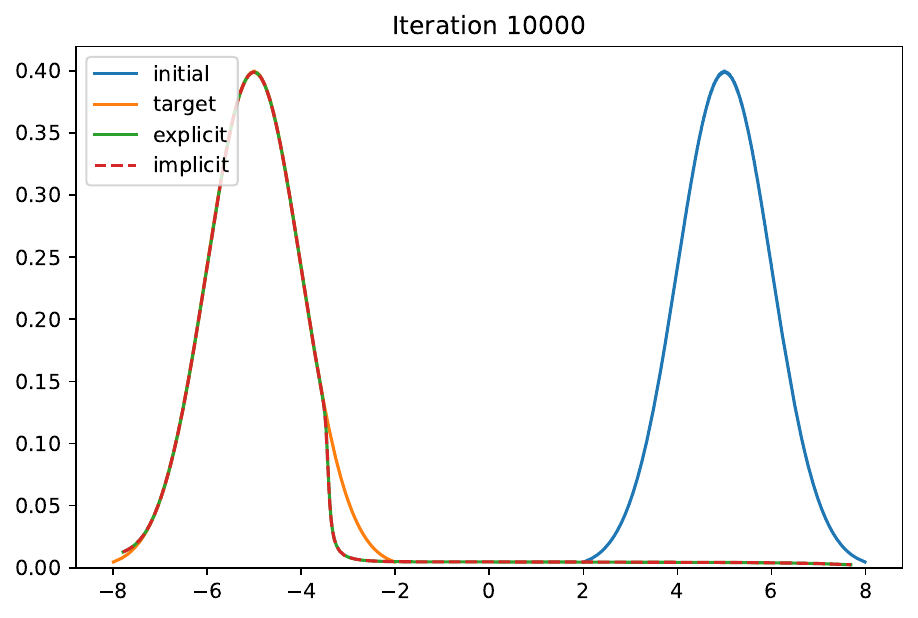}
    \caption{Comparison of implicit (red) and explicit (green) Euler schemes between two Gaussians $\mu_0 \sim \NN(5, 1)$ and $\nu \sim \NN(-5, 1)$ with $\tau = \tfrac{1}{100}$.
    For the corresponding quantile functions, see \figref{fig:Norm_To_Norm_Different_Means_q}.
    }
    \label{fig:Norm_To_Norm_Different_Means}
\end{figure}

\begin{figure}
    \centering
    \includegraphics[width=.32\textwidth]{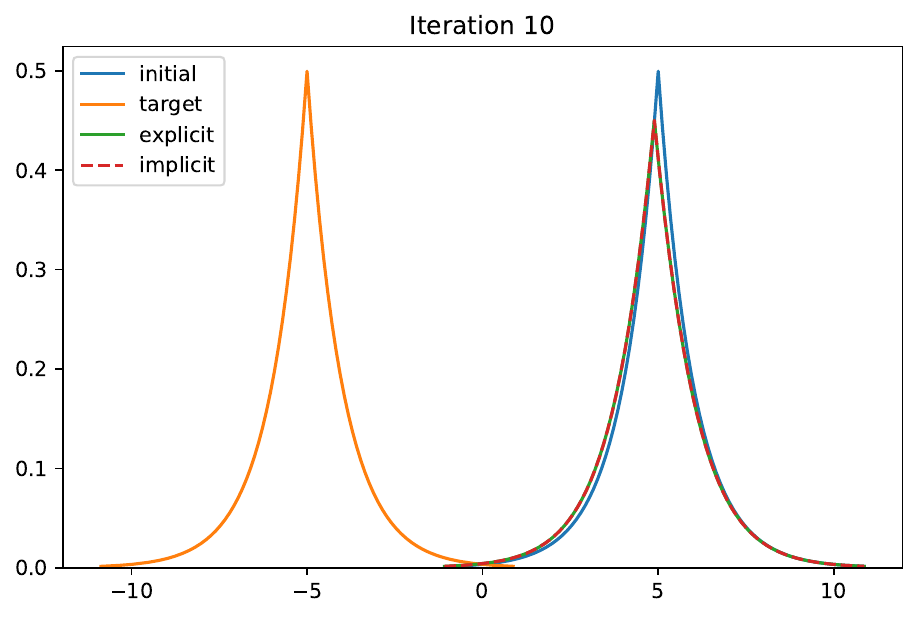}
    \includegraphics[width=.32\textwidth]{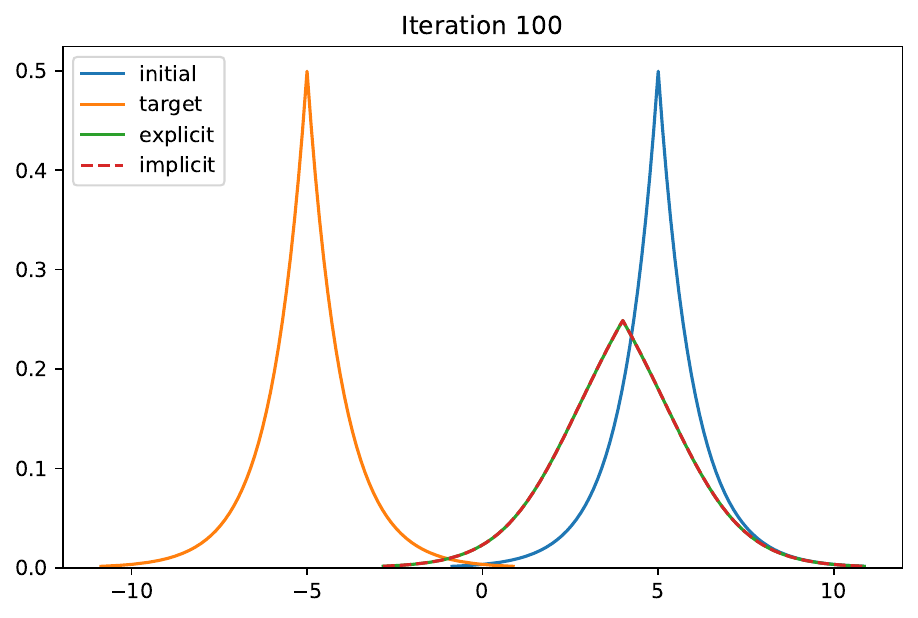}
    \includegraphics[width=.32\textwidth]{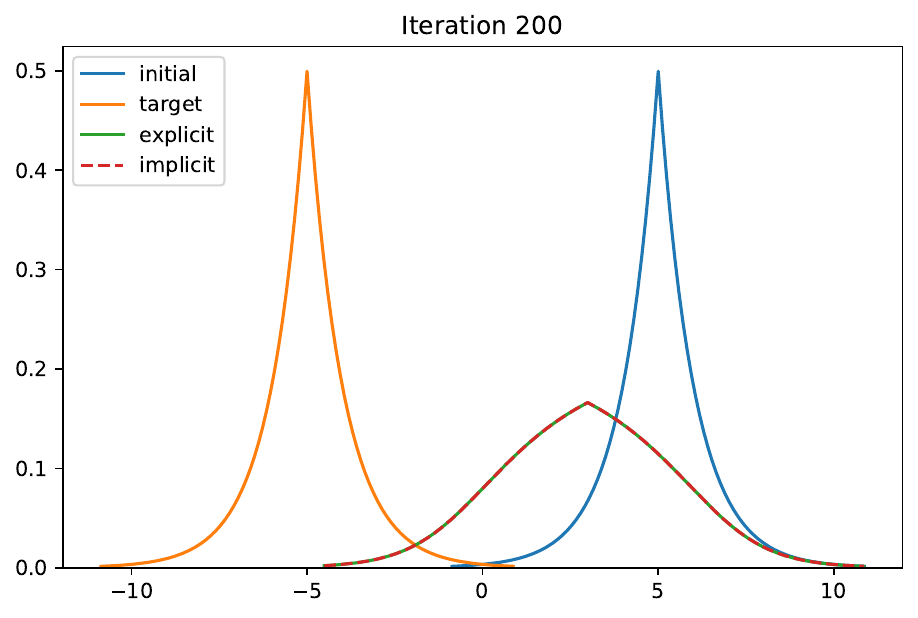}
    \includegraphics[width=.32\textwidth]{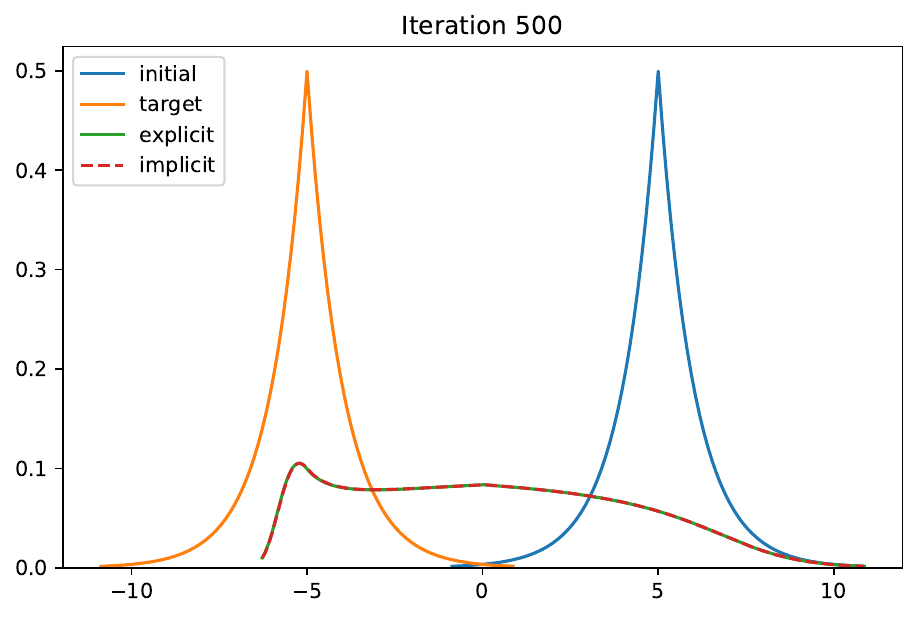}
    \includegraphics[width=.32\textwidth]{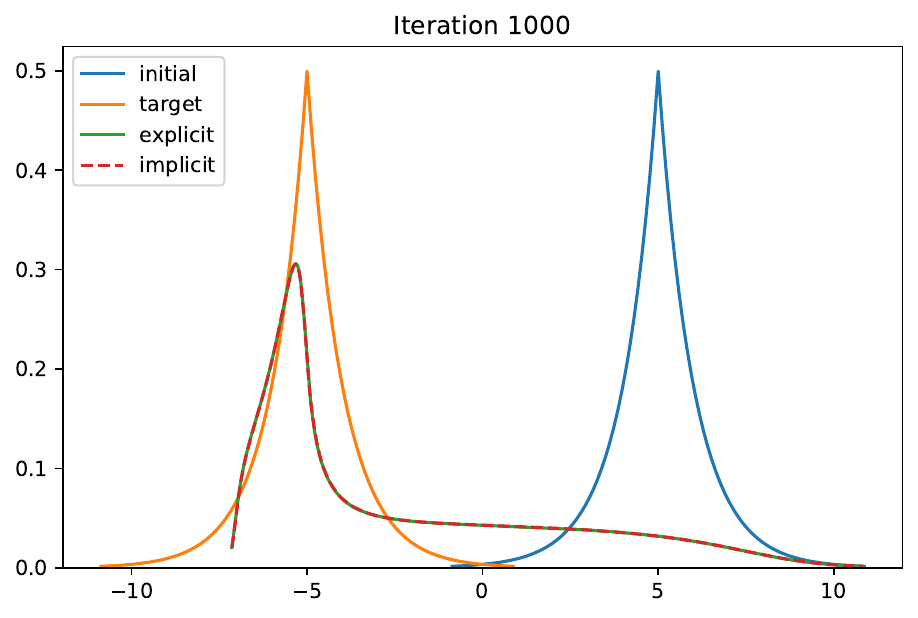}
    \includegraphics[width=.32\textwidth]{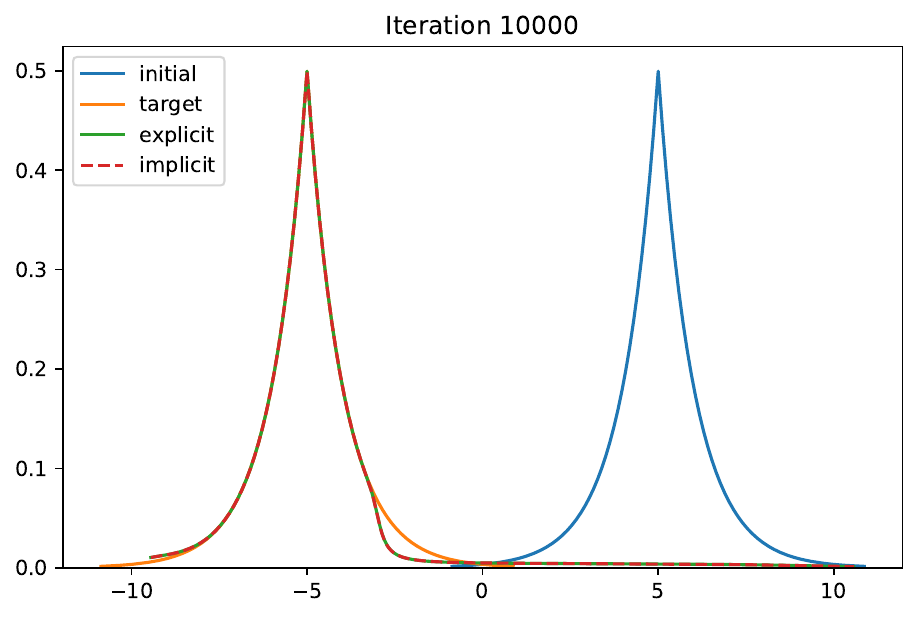}
    \caption{Comparison of implicit (red) and explicit (green) Euler schemes between two Laplacians $\mu_0 \sim \mathcal L(5, 1)$ and $\nu \sim \mathcal L(-5, 1)$ with $\tau = \tfrac{1}{100}$.
    For the corresponding quantile functions, see \figref{fig:Laplace_to_Laplace_q}.
    }
    \label{fig:Laplace_to_Laplace}
\end{figure}

\begin{figure}
    \centering
    \includegraphics[width=.32\textwidth]{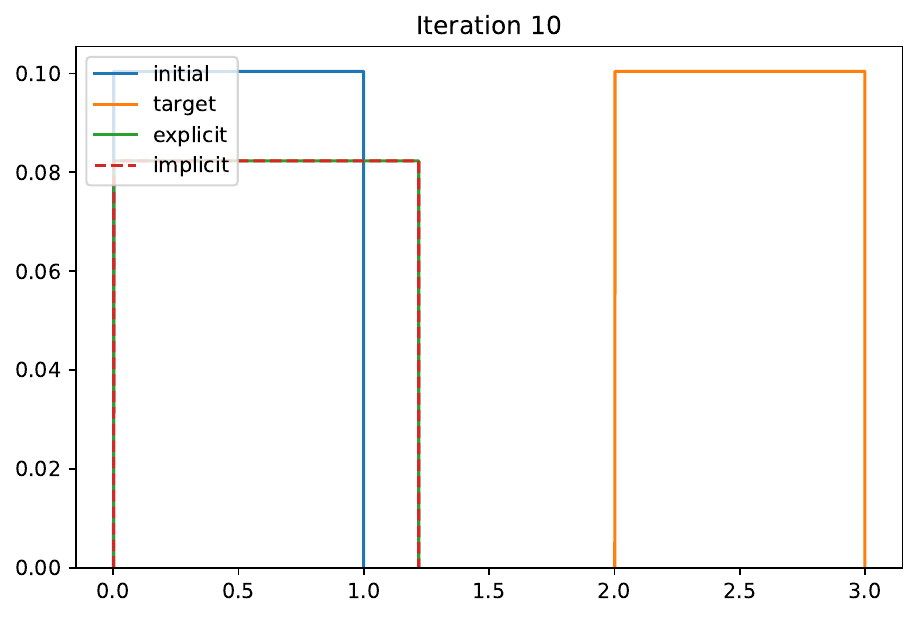}
    \includegraphics[width=.32\textwidth]{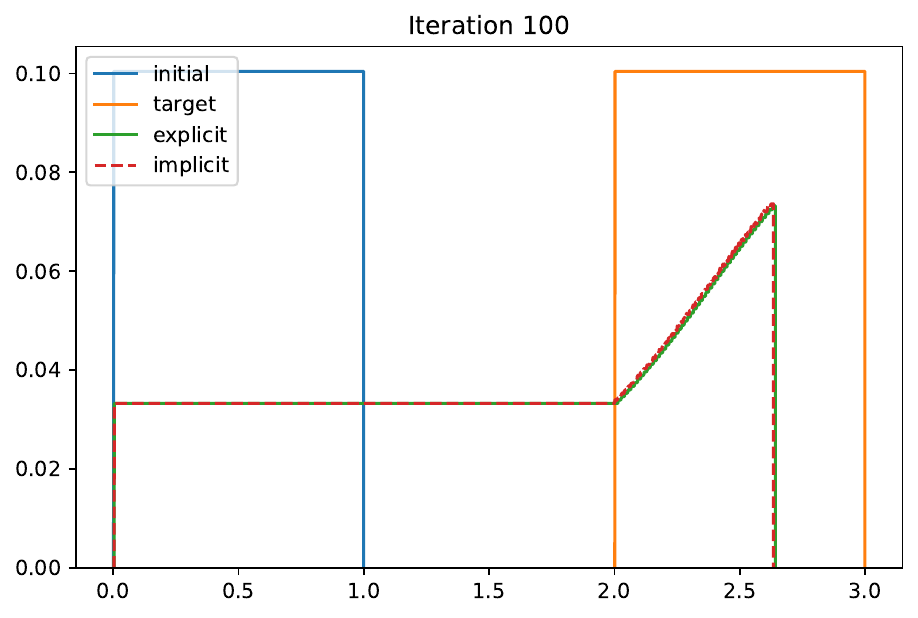}
    \includegraphics[width=.32\textwidth]{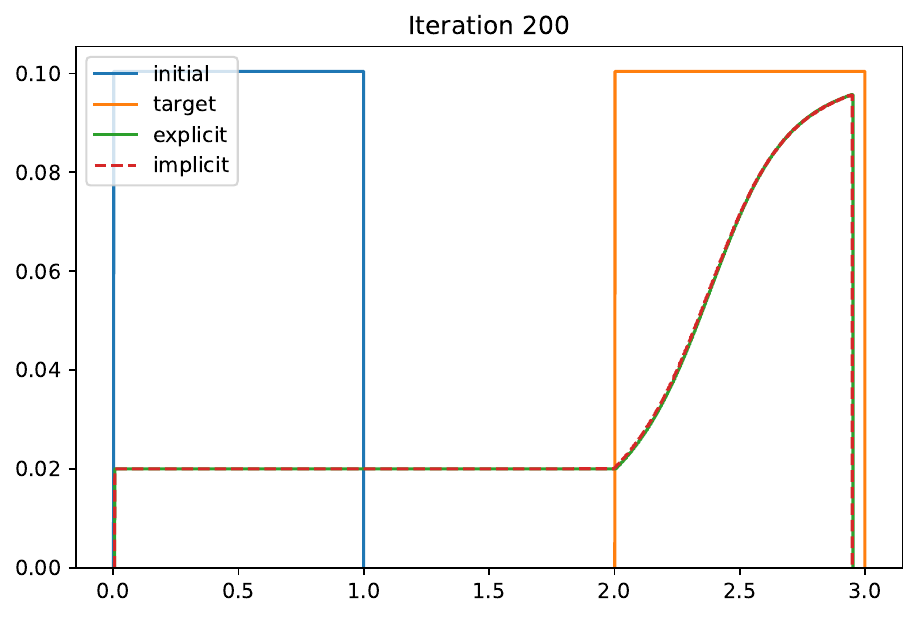}
    \includegraphics[width=.32\textwidth]{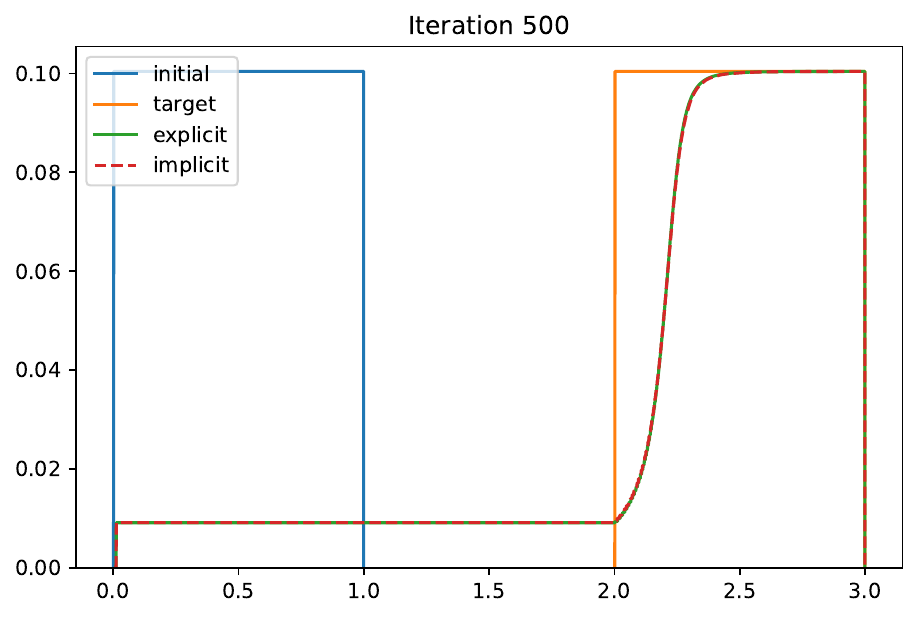}
    \includegraphics[width=.32\textwidth]{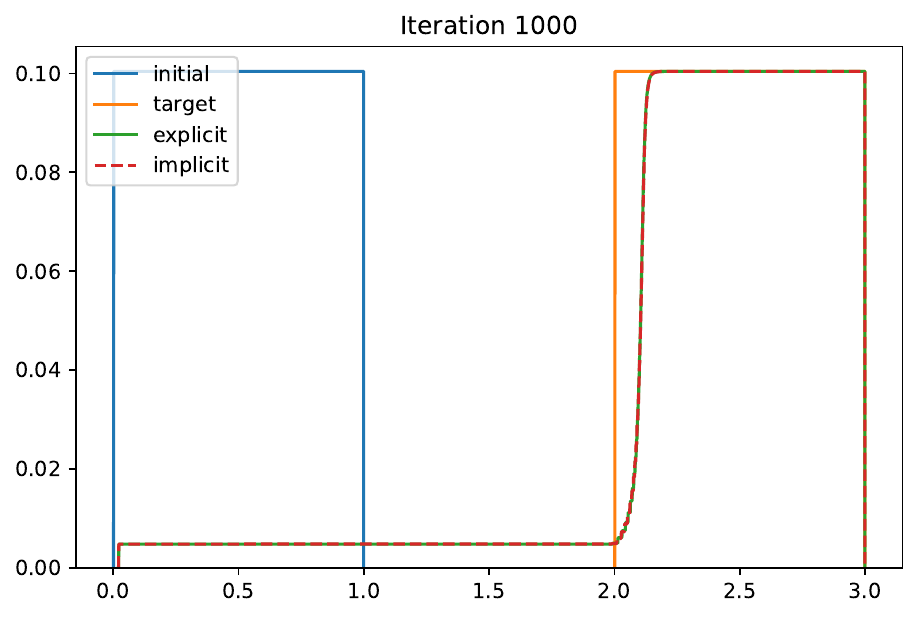}
    \includegraphics[width=.32\textwidth]{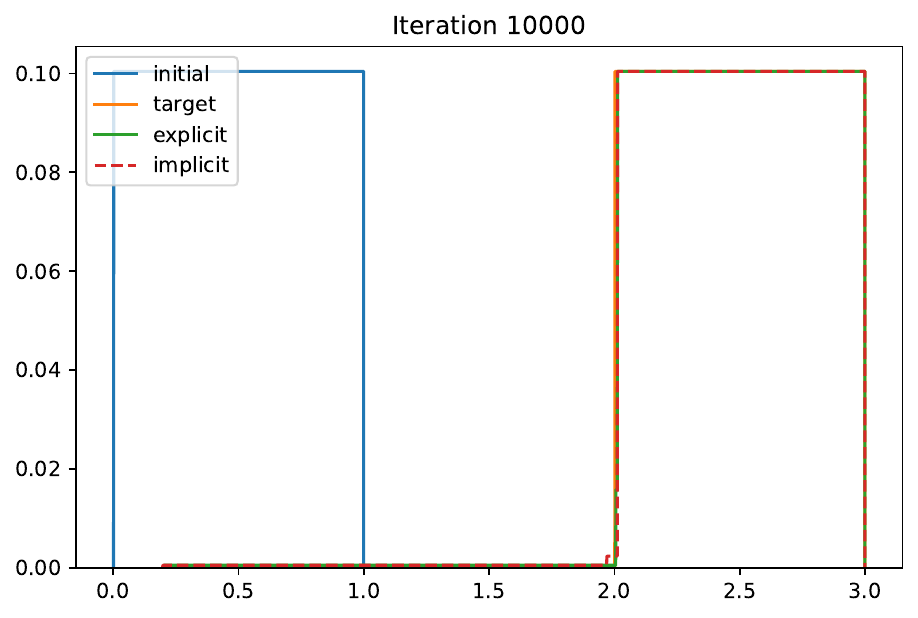}
    \caption{Comparison of implicit (red) and explicit (green) Euler schemes between two uniform distributions $\mu_0 \sim \mathcal U([0,1])$ and $\nu \sim \mathcal U([2,3])$ with $\tau = \tfrac{1}{100}$.
    For the corresponding quantile functions, see \figref{fig:Uniform_to_Uniform_q}.
    }
    \label{fig:Uniform-to-Uniform}
\end{figure}

\begin{figure}
    \centering
    \includegraphics[width=.32\textwidth]{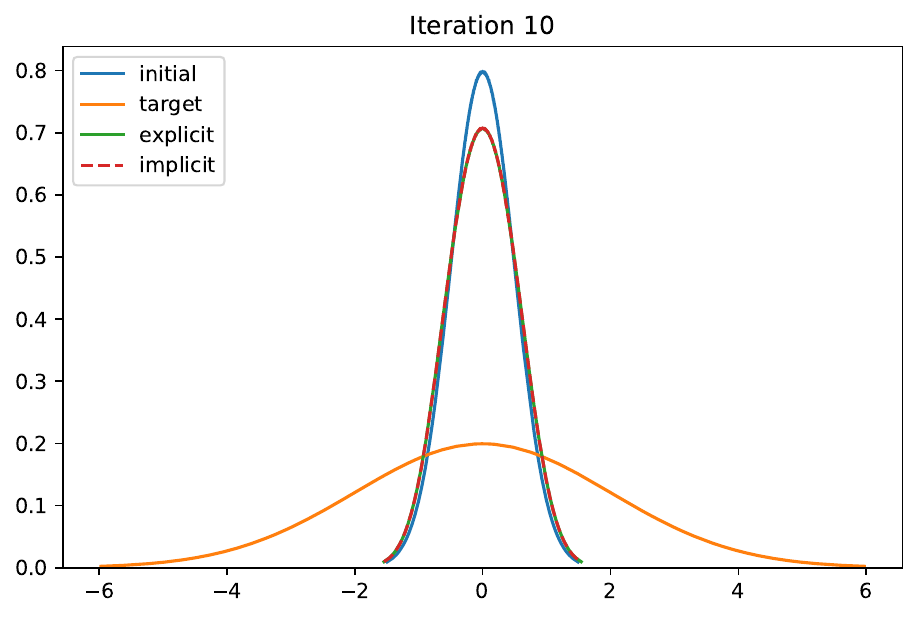}
    \includegraphics[width=.32\textwidth]{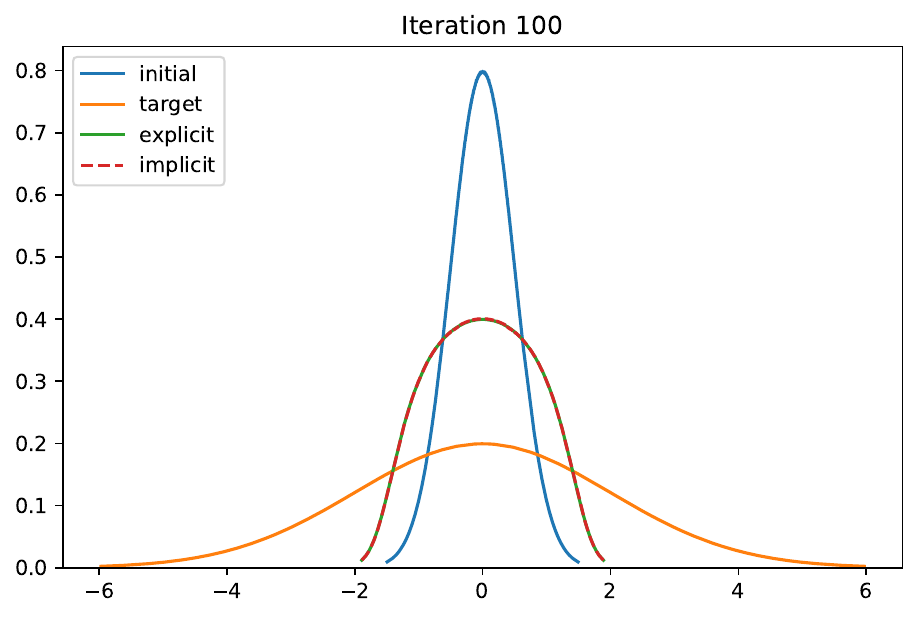}
    \includegraphics[width=.32\textwidth]{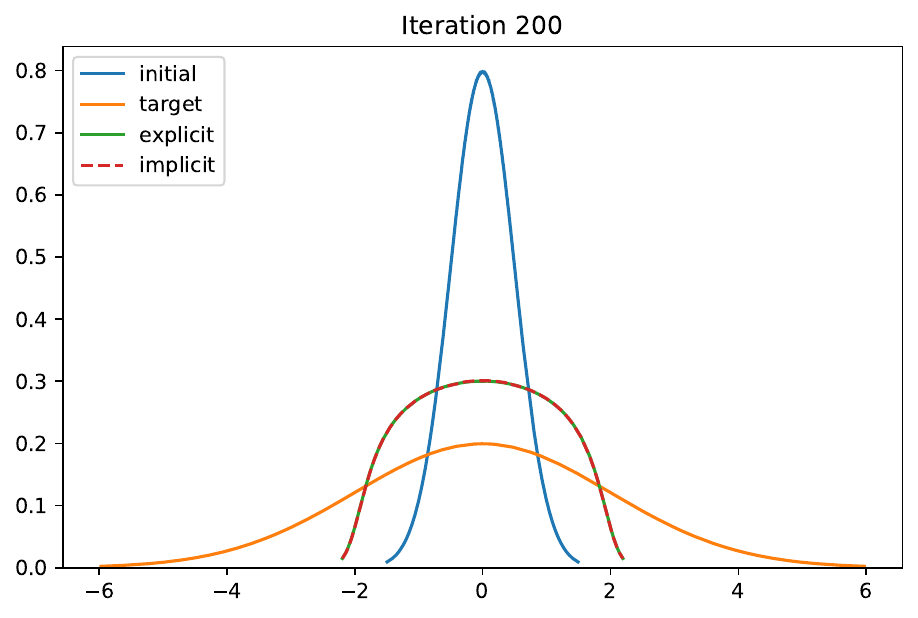}
    \includegraphics[width=.32\textwidth]{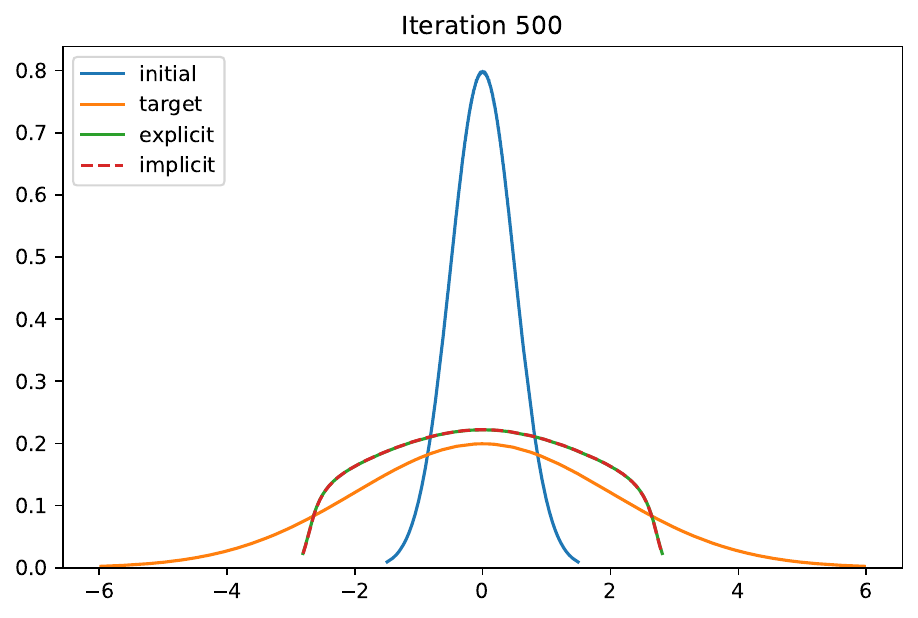}
    \includegraphics[width=.32\textwidth]{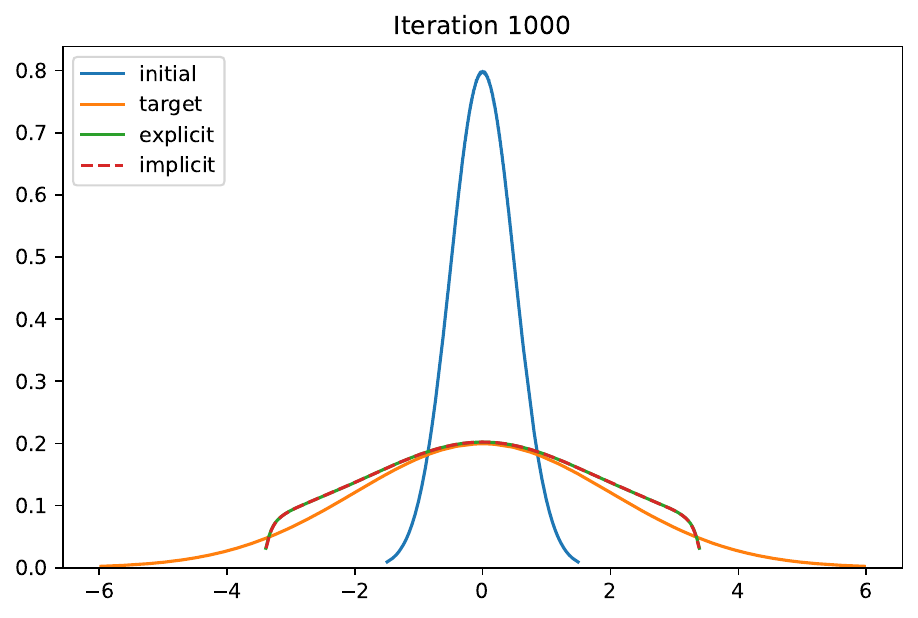}
    \includegraphics[width=.32\textwidth]{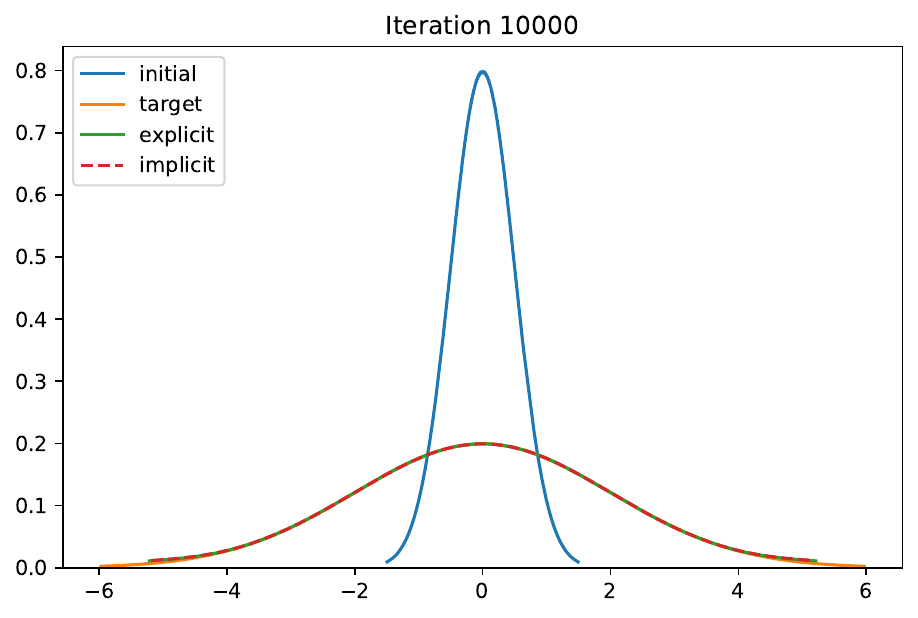}
    \caption{Comparison of implicit (red) and explicit (green) Euler schemes between two Gaussians $\mu_0 \sim \NN(0, \tfrac{1}{\sqrt{2}})$ and $\nu \sim \NN(0, \sqrt{2})$ with $\tau = \tfrac{1}{100}$.
    For the corresponding quantile functions, see \figref{fig:Norm_To_Norm_Different_Scales_q}.
    }
    \label{fig:Norm_To_Norm_Different_Scales}
\end{figure}

\paragraph{Flow between measures with different number of modes.}
    \figref{fig:Bimodal_Gauss_to_Gauss} shows the case
    where the initial measure $\mu_0$ is a Gaussian mixture model, which is symmetric with respect to the origin, and $\nu$ is a standard Gaussian.
    Interestingly, we see that when the two parts of the visible support of the flow collide at the origin, it takes a lot of time to recover the peak of the target measure, that is, the region of high density is recovered after regions of lower density.

    \noindent In \figref{fig:Gauss_to_Bimodal_Gauss} we switch initial and target measure.
    Here, we observe a similar behavior as in \figref{fig:Norm_To_Norm_Different_Means}: first the densities spread out and become more flat, and the two peaks of the target are developed not until the density meets the modes of the target.

\begin{figure}
    \centering
    \includegraphics[width=.32\textwidth]{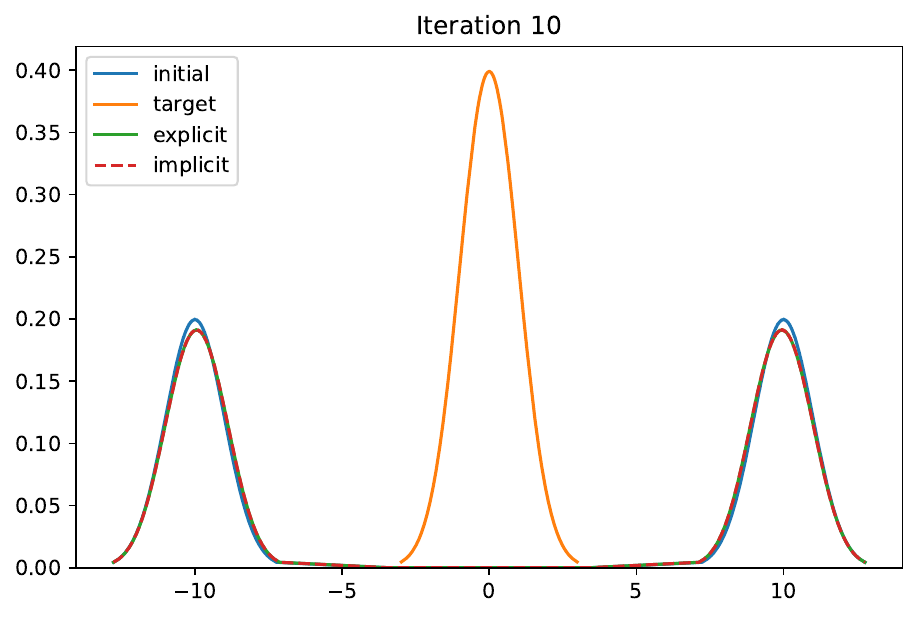}
    \includegraphics[width=.32\textwidth]{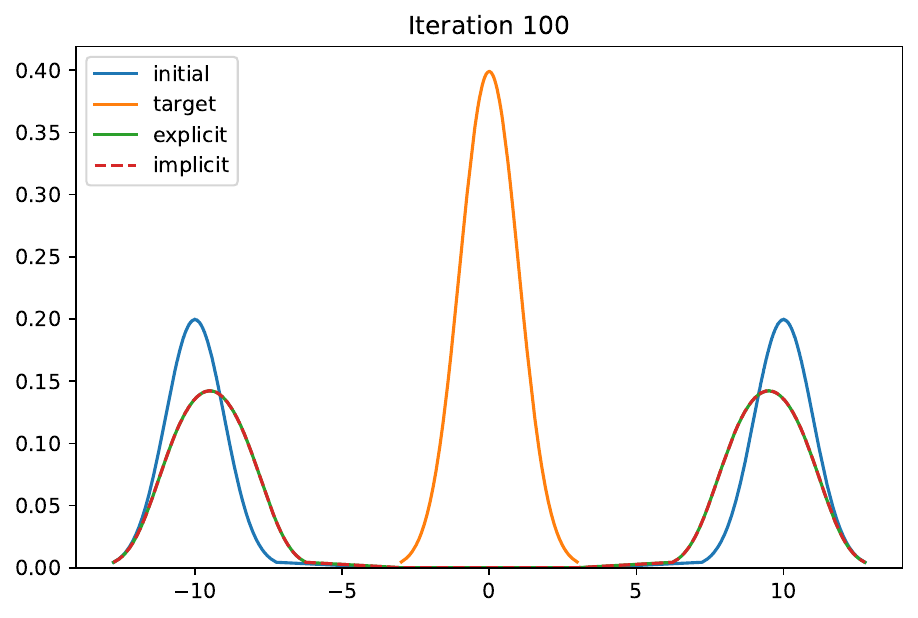}
    \includegraphics[width=.32\textwidth]{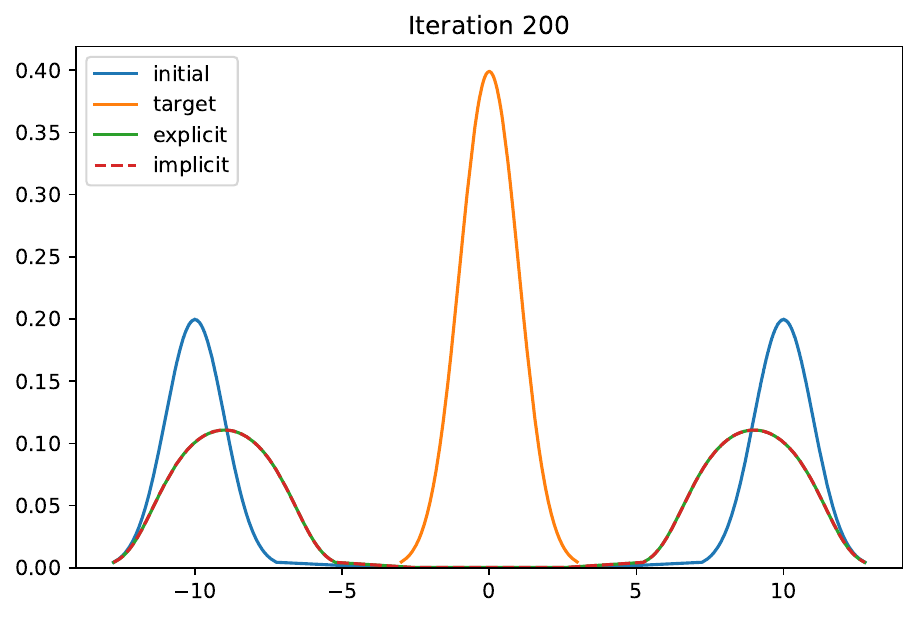}
    \includegraphics[width=.32\textwidth]{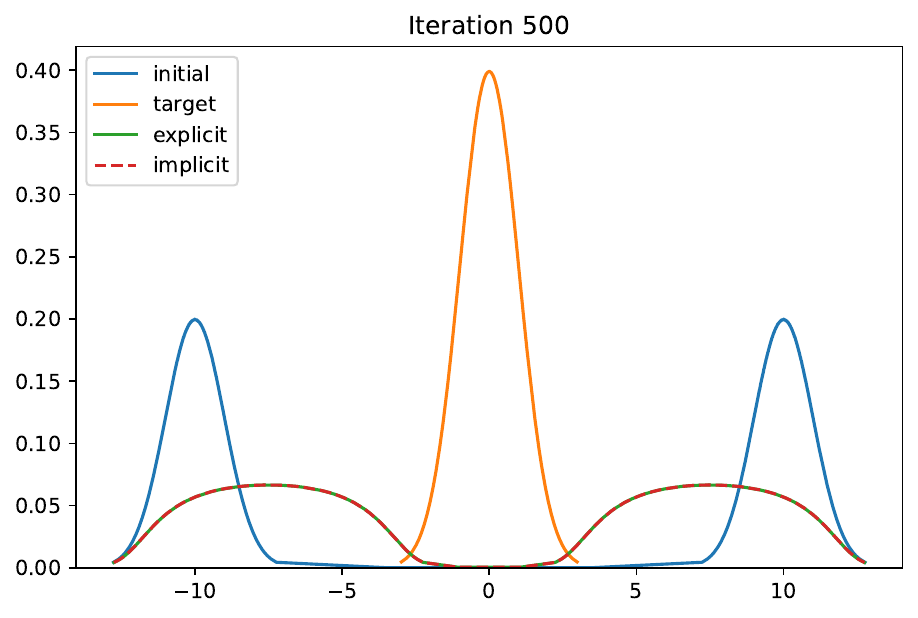}
    \includegraphics[width=.32\textwidth]{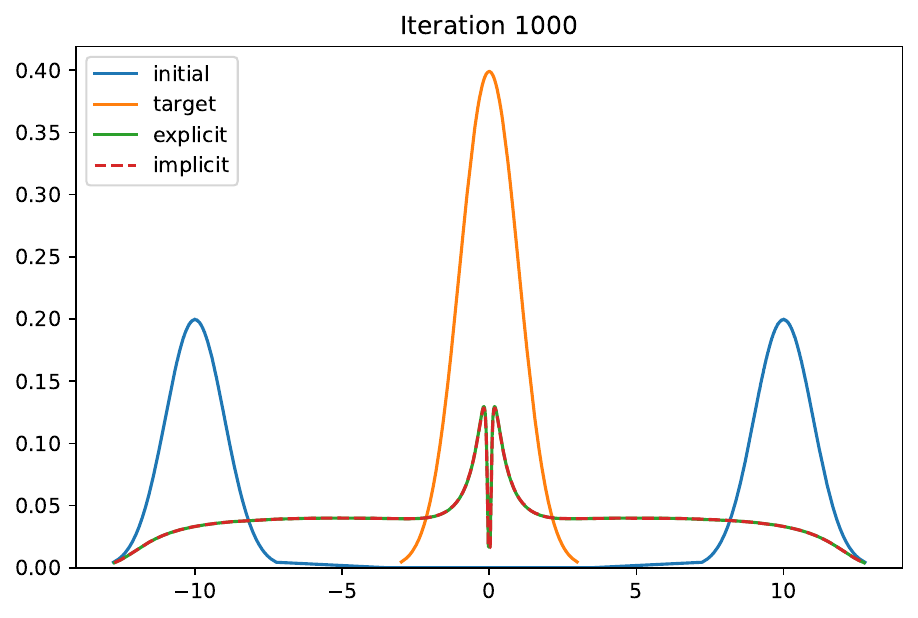}
    \includegraphics[width=.32\textwidth]{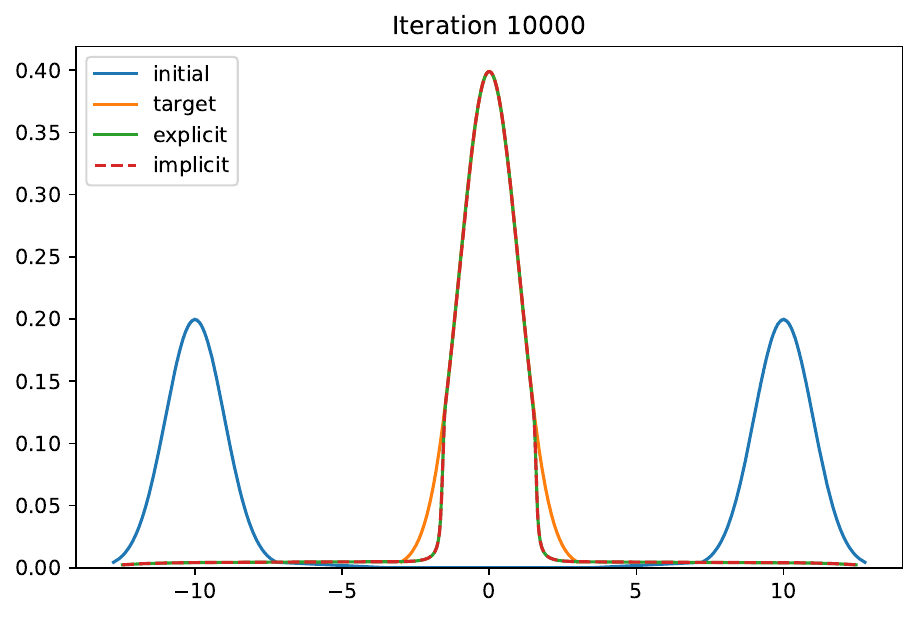}
    \caption{Comparison of implicit (red) and explicit (green) Euler schemes between $\mu_0 \sim \frac{1}{2} \NN(-10, 1) + \frac{1}{2} \NN(10, 1)$ and $\nu \sim \NN(0, 1)$ with $\tau = \tfrac{1}{100}$.
    For the corresponding quantile functions, see  \figref{fig:Bimodal_Gauss_to_Gauss_q}.
    }
    \label{fig:Bimodal_Gauss_to_Gauss}
\end{figure}

\begin{figure}
    \centering
    \includegraphics[width=.32\textwidth]{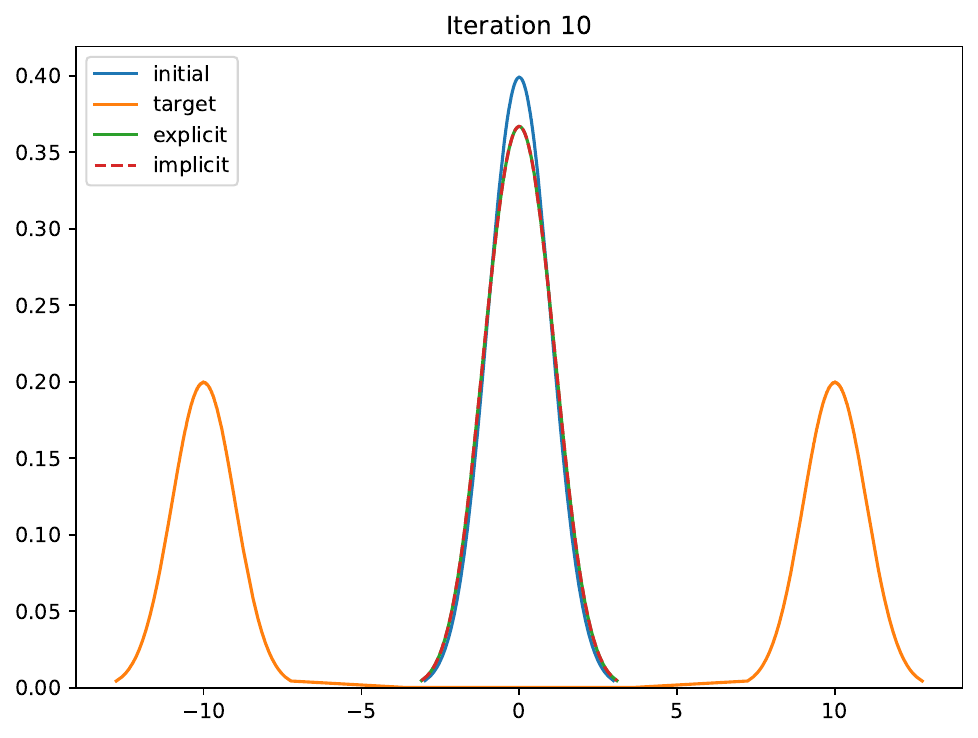}
    \includegraphics[width=.32\textwidth]{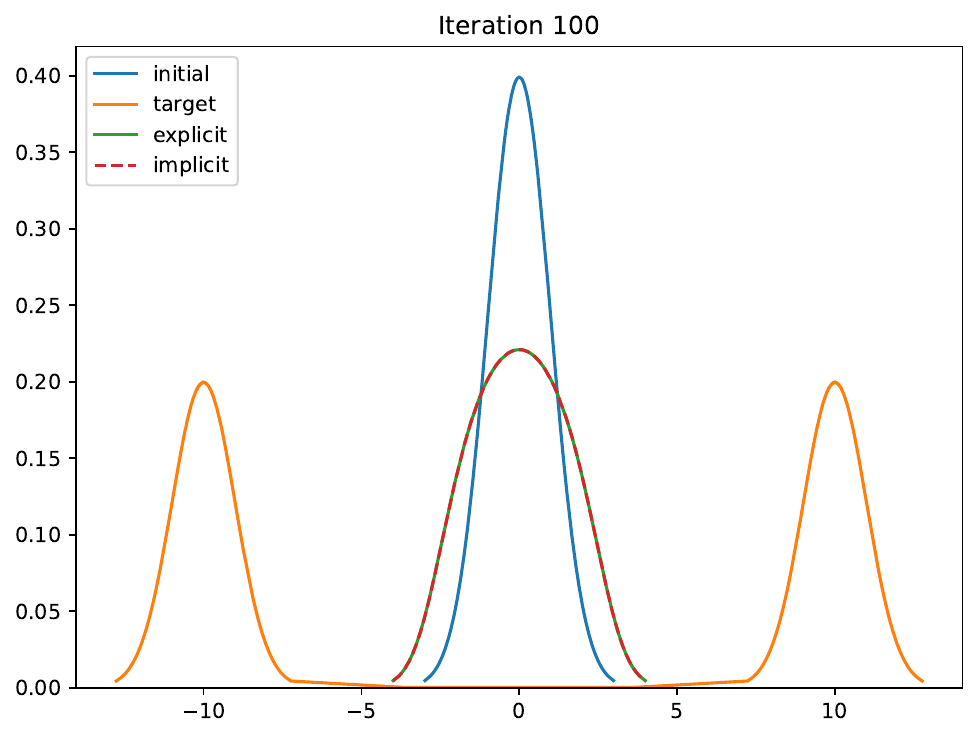}
    \includegraphics[width=.32\textwidth]{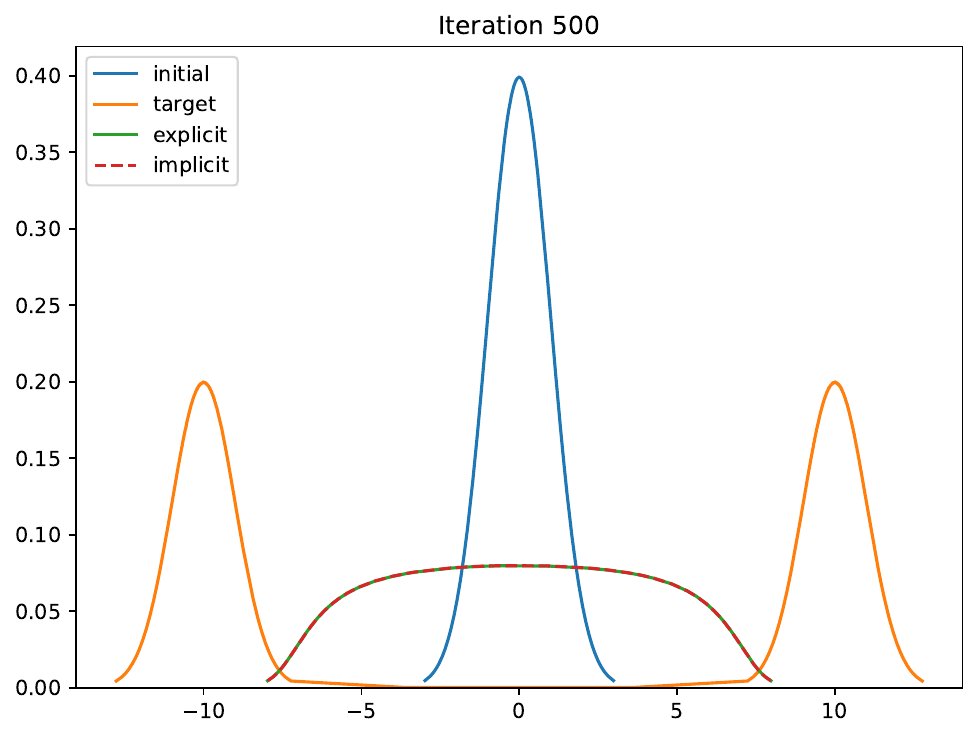}
    \includegraphics[width=.32\textwidth]{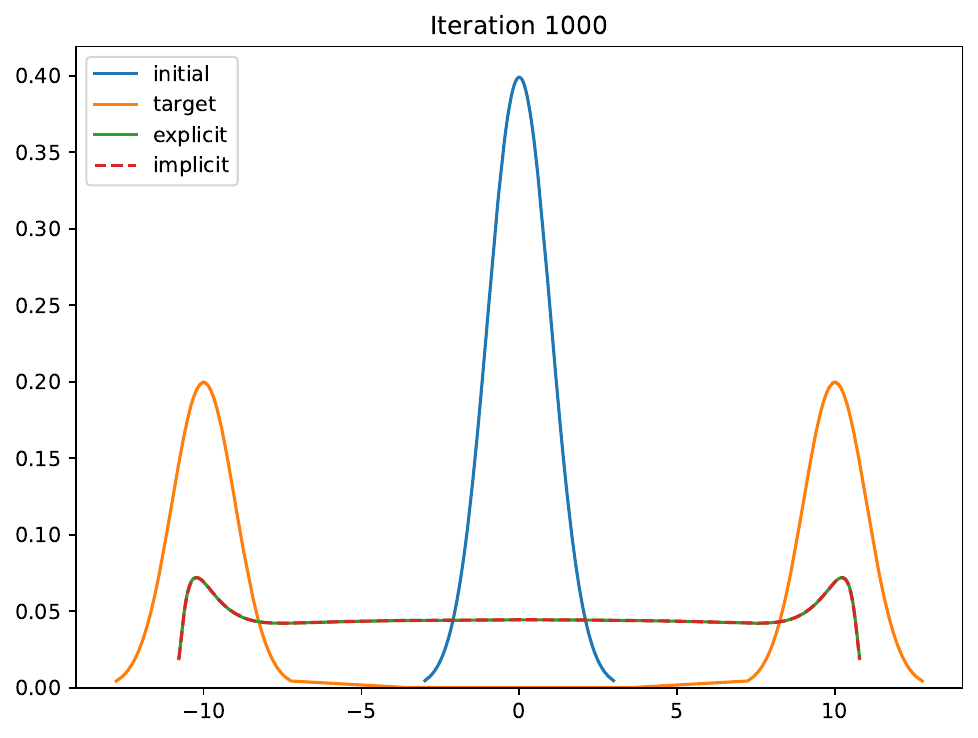}
    \includegraphics[width=.32\textwidth]{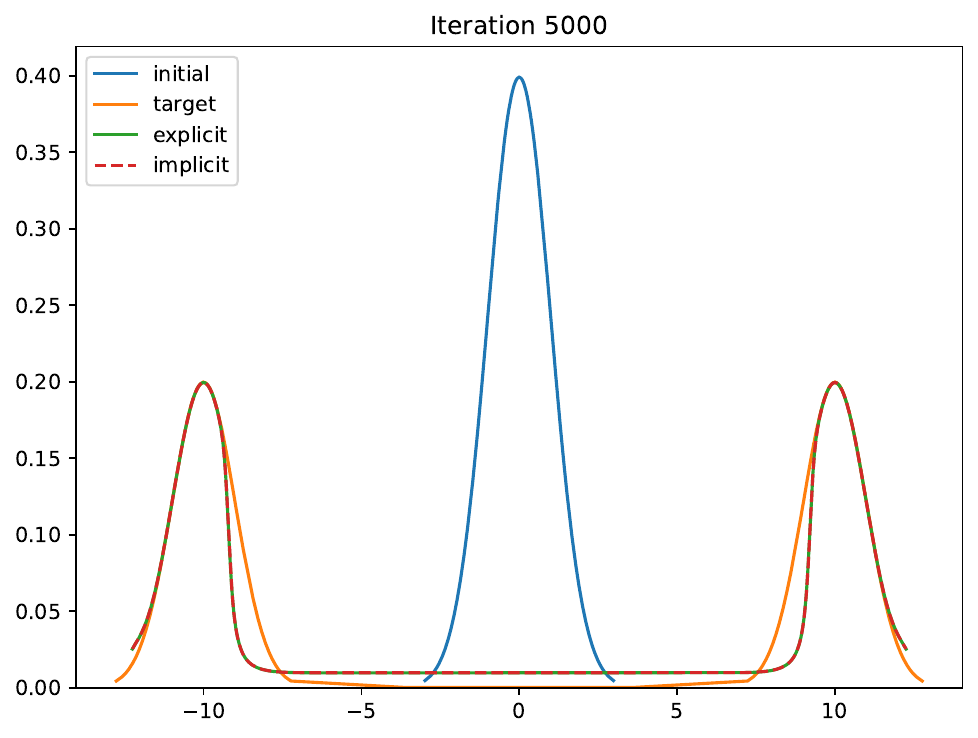}
    \includegraphics[width=.32\textwidth]{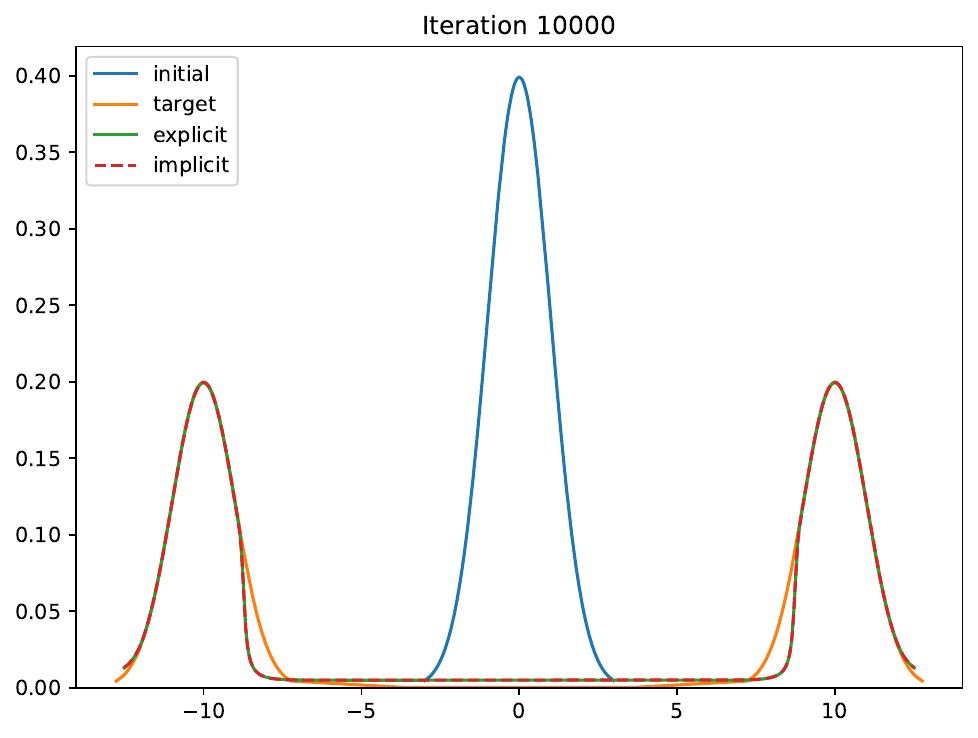}
    \caption{Comparison of implicit (red) and explicit (green) Euler schemes between $\mu_0 \sim \NN(0, 1)$ and $\nu \sim \frac{1}{2} \NN(-10, 1) + \frac{1}{2} \NN(10, 1)$ with $\tau = \tfrac{1}{100}$.
    For the corresponding quantile functions, see  \figref{fig:Gauss_to_Bimodal_Gauss_q}.
    }
    \label{fig:Gauss_to_Bimodal_Gauss}
\end{figure}

\paragraph{Flow between measures without full support}
    Our simulations can also handle initial measures and targets which are not everywhere supported. Consider for example the folded normal distribution $\mathcal {FN}(\mu)$ with "mean" $\mu$, which has the density 
    \begin{equation} \label{eq:folded_norm}
        \R \to \R, \qquad 
        x \mapsto \sqrt{\frac{2}{\pi}} \exp\left(- \frac{x^2 + \mu^2}{2} \right) \cosh(\mu x).
    \end{equation}
    We consider the case of the target and initial measure being folded Gaussians in \figref{fig:Folded_Norm_to_Folded_Norm}.
    We again observe that the structure of the distribution is not preserved by the flow, the support grows monotonically, and that the tail of the target distribution is matched exponentially slow.

\begin{figure}
    \centering
    \includegraphics[width=.32\textwidth]{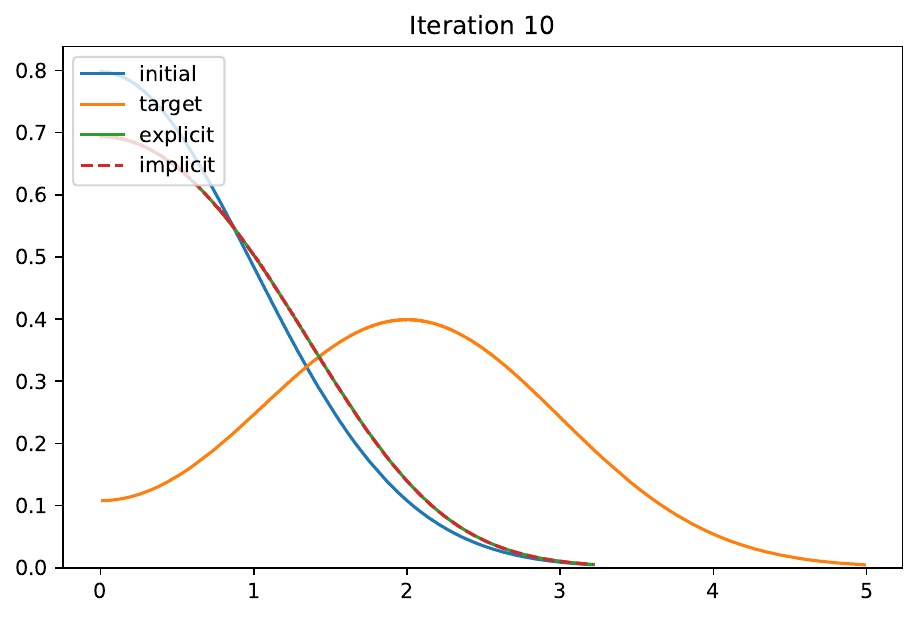}
    \includegraphics[width=.32\textwidth]{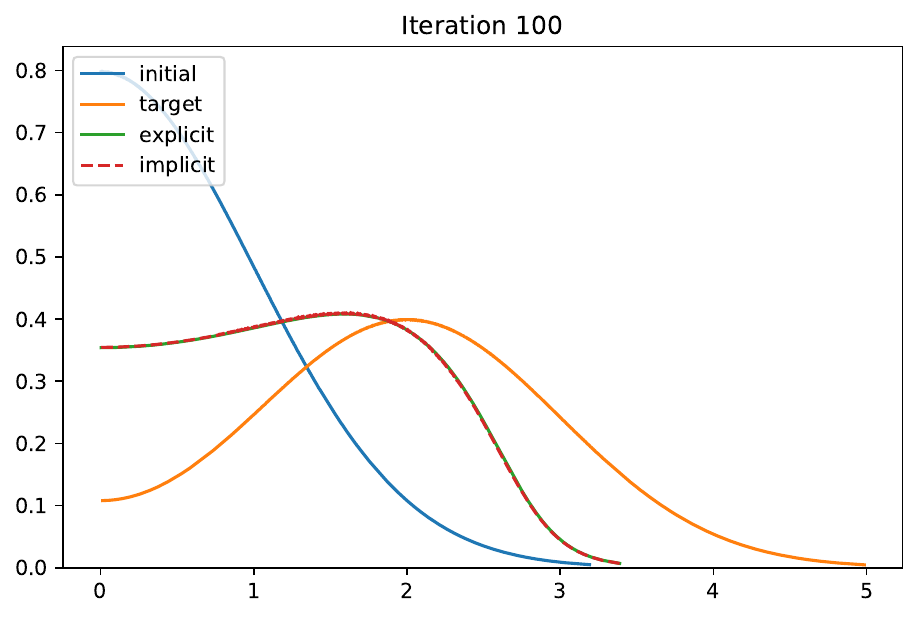}
    \includegraphics[width=.32\textwidth]{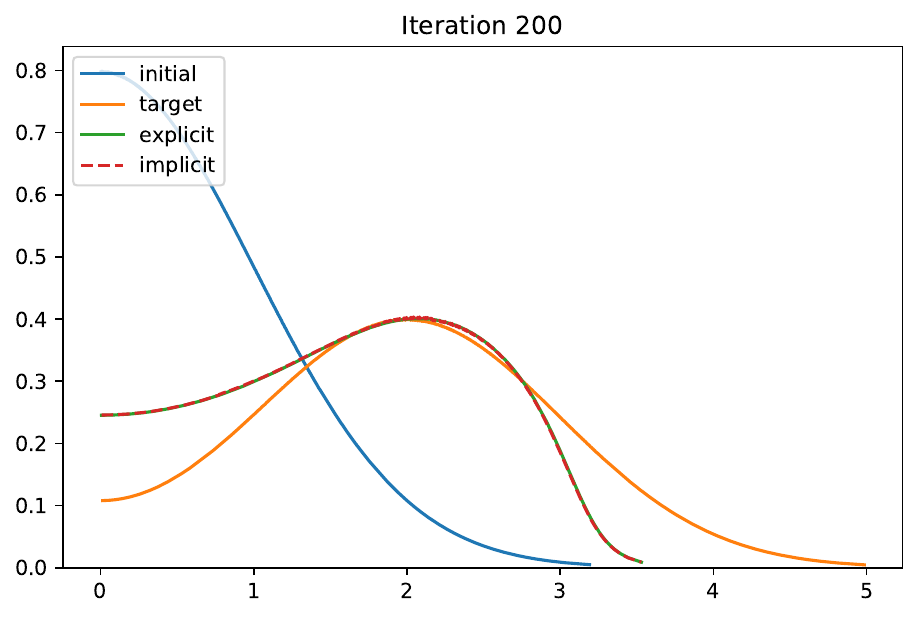}
    \includegraphics[width=.32\textwidth]{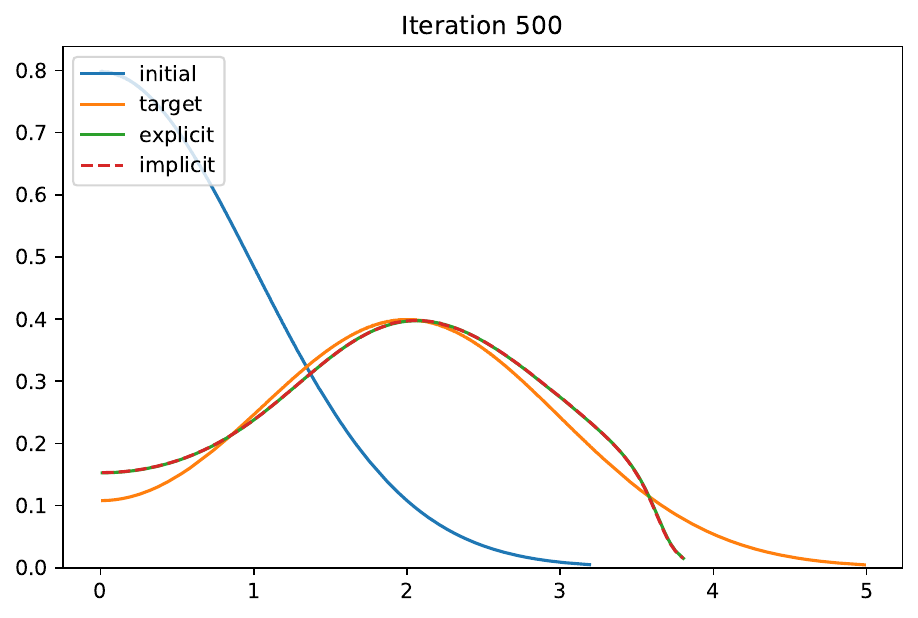}
    \includegraphics[width=.32\textwidth]{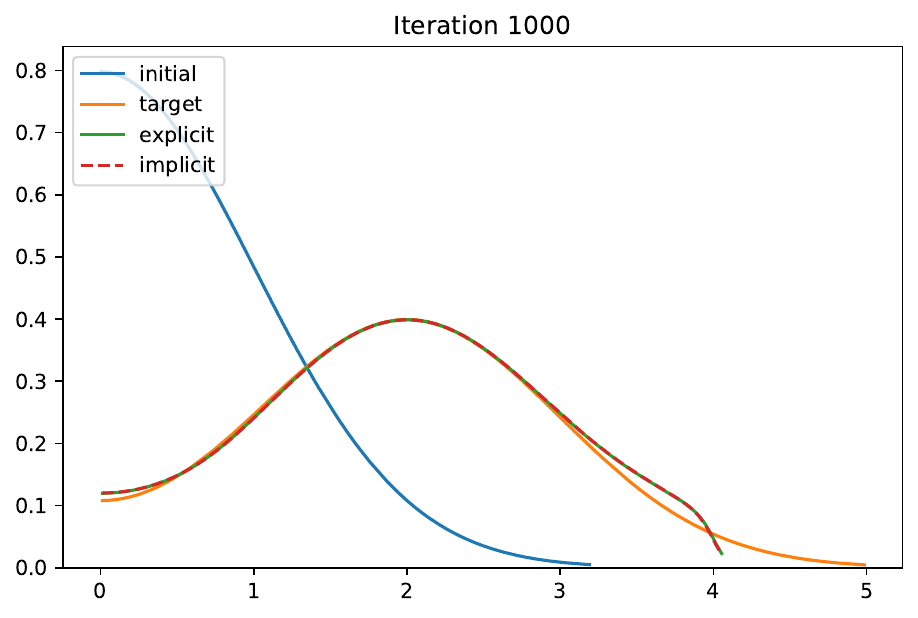}
    \includegraphics[width=.32\textwidth]{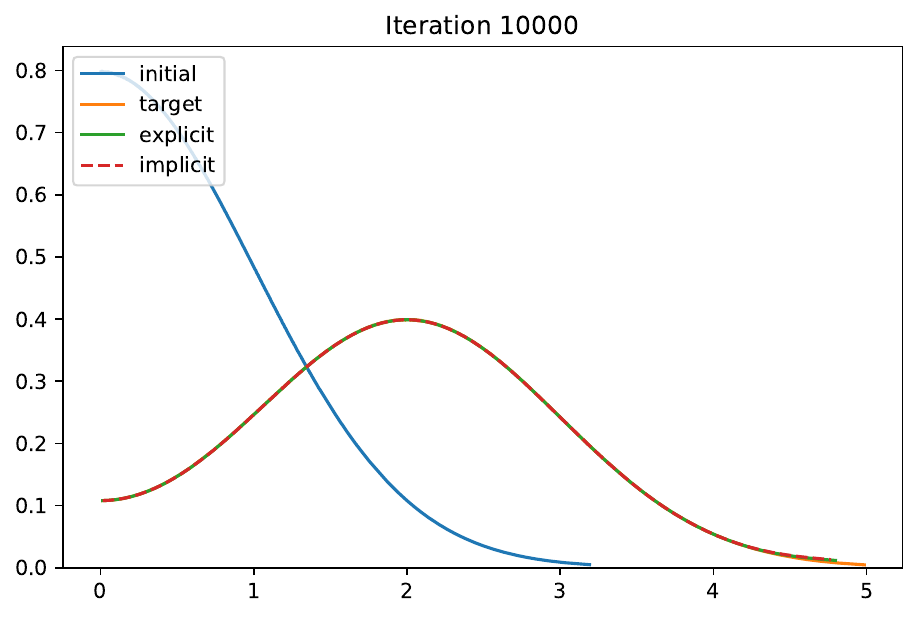}
    \caption{Comparison of implicit (red) and explicit (green) Euler schemes between $\mu_0 \sim \mathcal{FN}(0, 1)$ and $\nu \sim \mathcal{FN}(2, 1)$ with $\tau = \tfrac{1}{100}$.
    For the corresponding quantile functions, see \figref{fig:Folded_Norm_to_Folded_Norm_q}.
    }
    \label{fig:Folded_Norm_to_Folded_Norm}
\end{figure}

\paragraph*{Acknowledgments.}
GS acknowledges funding by the 
 German Research Foundation (DFG) within the project STE 571/16-1.
JH acknowledges funding by the EPSRC programme grant \textit{The Mathematics of Deep Learning} with reference EP/V026259/1 and by the DFG within project 530824055.

\bibliographystyle{abbrv}
\bibliography{Bibliography}

@article{NS2022,
  title={Weak topology and {O}pial property in {W}asserstein spaces with applications to gradient flows and proximal point algorithms of geodesically
convex functionals},
         author = {E. Naldi and G. Savar\'e},
journal = {Rendiconti Lincei
Matematica e Applicazioni},
volume = {32}, 
number = {4},
pages = {725--750},
year = {2022}
}

@article{NS2009,
title={A {W}asserstein Approach to the One-Dimensional Sticky Particle System},
author = {L. Natile and G. Savar\'e},
journal = {SIAM Journal on Mathematical Analysis},
volume = {41},
number = {4},
year = {2009},
pages = {1340--1365}
}

@article{CMSVW2024,
title={Riesz Energy with a Radial External Field: When is the Equilibrium Support a Sphere?},
author = {D. Chafaï and R.W. Matzke and E.B. Saff and M.Q.H. Vu and  R.S. Womersley},
journal ={arXiv preprint arXiv:2405.00120},
year = {2024}
}

@article{BS2024,
  title={ 
Quasi-alpha ﬁrmly nonexpansive mappings in {W}asserstein spaces},
author = {A. Berdellima and G. Steidl},
year = {2024},
journal = {Fixed Point Theory, 
 arXiv:2203.04851},
}

@article{LBADDP2022,
author = {R. Laumont and V.D. Bortoli and A. Almansa and J. Delon and A. Durmus and M. Pereyra},
title = {Bayesian imaging using plug \& play priors: when {L}angevin meets {T}weedie},
journal = {SIAM Journal on Imaging Sciences}, 
volume = {15},
number = {2}, 
pages = {701--737},
year = {2022}
}

@inproceedings{AKSG2019,
	author = {Arbel, Michael and Korba, Anna and Salim, Adil and Gretton, Arthur},
	booktitle = {Advances in Neural Information Processing Systems},
	pages = {},
	title = {Maximum Mean Discrepancy Gradient Flow},
	volume = {32},
	year = {2019}
}

@inproceedings{HHABCS2023,
  title={Posterior Sampling Based on Gradient Flows of the {MMD} with Negative Distance Kernel},
  author={Hagemann, Paul and Hertrich, Johannes and Altekr{\"u}ger, Fabian and Beinert, Robert and Chemseddine, Jannis and Steidl, Gabriele},
  booktitle={International Conference on Learning Representations},
  year={2024}
}

@article{DRSS2025,
    title={Wasserstein gradient ﬂows of {MMD} functionals with distance kernels under {Sobolev} regularization}, 
    author={R. Duong and R. Rux and V. Stein and G. Steidl},
    journal = {Philosophical Transactions of the Royal Society A: Mathematical, Physical and Engineering Sciences},
    volume = {383},
    number = {2298},
    pages = {20240243},
    year = {2025},
    doi = {10.1098/rsta.2024.0243}, 
}

@article{DCFS2025,
  title={Telegrapher's generative model via {Kac} flows},
  author={Duong, Richard and Chemseddine, Jannis and Friz, Peter K. and Steidl, Gabriele},
  journal={arXiv preprint arXiv:2506.20641},
  year={2025}
}

@book{nathanson1955,
author={P. I. Natanson},
title={Theory of Functions of a Real Variable}, 
publisher = {Ungar}, 
address = {New York}, 
year = {1955}
}

@book{BC2009,
author={J. Benedetto and W. Czaja}, 
title={Integration and Modern Analysis}, 
publisher = {Birkh\"auser}, 
address = {Boston},
year = {2009}
}

@inproceedings{HBGS2023,
  title={Wasserstein gradient flows of the discrepancy with distance kernel on the line},
  author={Hertrich, Johannes and Beinert, Robert and Gr{\"a}f, Manuel and Steidl, Gabriele},
  booktitle={International Conference on Scale Space and Variational Methods in Computer Vision},
  pages={431--443},
  year={2023},
  organization={Springer}
}

@article{R1971,
  title={Integrals which are convex functionals. {II}},
  author={Rockafellar, Ralph T.},
  journal={Pacific Journal of Mathematics},
  volume={39},
  number={2},
  pages={439--469},
  year={1971},
  publisher={Mathematical Sciences Publishers}
}

@book{SGGHL2009,
  title={Variational Methods in Imaging},
  author={Scherzer, Otmar and Grasmair, Markus and Grossauer, Harald and Haltmeier, Markus and Lenzen, Frank},
  volume={167},
  year={2009},
  publisher={Springer}
}

@inproceedings{HWAH2023,
      title={Generative sliced {MMD} flows with {Riesz} kernels}, 
      author={J. Hertrich and C. Wald and F. Altekrüger and P. Hagemann},
      year={2024},
    booktitle={International Conference on Learning Representations },
}

@article{HGBS2022,
  title={Wasserstein steepest descent flows of discrepancies with {R}iesz kernels},
  author={Hertrich, Johannes and Gr{\"a}f, Manuel and Beinert, Robert and Steidl, Gabriele},
  journal={Journal of Mathematical Analysis and Applications},
  volume={531},
  number={1},
  pages={127829},
  year={2024},
  publisher={Elsevier}
}

@article{FHS2013,
author = {M. Fornasier and J. Haskovec and G. Steidl}, 
title = {Consistency of variational continuous-domain quantization via kinetic theory}, 
journal = {Applicable Analysis}, 
volume =  {92}, 
number = {6}, 
pages = {1283--1298}, 
year = {2013}
}

@article{JKO1998,
    author = {Jordan, Richard and Kinderlehrer, David and Otto, Felix},
    title = {The Variational Formulation of the {F}okker--{P}lanck Equation},
    journal = {SIAM Journal on Mathematical Analysis},
    volume = {29},
    number = {1},
    pages = {1--17},
    year = {1998},
   }

@article{RoRo14,
    author = {Rockafellar, R. T. and Royset, J. O.},
    title = {Random variables, monotone relations, and convex analysis},
    journal = {Mathematical Programming},
    year = {2014},
    pages = {297--331},
    volume = {148},
    publisher={Springer}
   }

@book{BookAmGiSa05,
    author = {L. Ambrosio and N. Gigli and G. Savare},
    title = {Gradient Flows},
    subtitle = {in Metric Spaces and in the Space of Probability Measures},
    publisher = {Birkh\"auser, Basel},
    year = {2008},
    series = {Lectures in Mathematics ETH Zürich},
    edition = {2nd},
    doi = {10.1007/978-3-7643-8722-8}
   }

@book{BookVi09,
    author = {C. Villani},
    title = {Optimal Transport},
    subtitle = {Old and New},
    publisher = {Springer, Berlin},
    year = {2009}
}

@book{Vil03,
    AUTHOR = {Villani, C\'{e}dric},
     TITLE = {Topics in Optimal Transportation},
    SERIES = {Graduate Studies in Mathematics},
    number = {58},
 PUBLISHER = {American Mathematical Society},
 address = {Providence},
      YEAR = {2003}
}

@book{BC2011,
  title={Convex Analysis and Monotone Operator Theory in Hilbert Spaces},
  author={Bauschke, H.H. and Combettes, P.L.},
  isbn={9781441994677},
  lccn={2011926587},
  series={CMS Books in Mathematics},
  year={2011},
  publisher={Springer New York}
}

@article{CL1971,
 ISSN = {00029327, 10806377},
 author = {M. G. Crandall and T. M. Liggett},
 journal = {American Journal of Mathematics},
 number = {2},
 pages = {265--298},
 publisher = {Johns Hopkins University Press},
 title = {Generation of Semi-Groups of Nonlinear Transformations on General {B}anach Spaces},
 volume = {93},
 year = {1971},
 doi = {10.2307/2373376}
}

@book{M2023,
place={Cambridge},
series={Cambridge Studies in Advanced Mathematics},
title={Optimal Mass Transport on Euclidean Spaces},
publisher={Cambridge University Press},
author={Maggi, Francesco},
year={2023},
collection={Cambridge Studies in Advanced Mathematics},
ISBN = {9781009179713},
doi = {10.1017/9781009179713}
}

@book{B1973,
  title={Operateurs Maximaux Monotones},
  subtitle = {et semi-groups de contractions dans les espaces de Hilbert},
  author={Brezis, Haim},
  year={1973},
  publisher={North-Holland Mathematics Studies},
  language = {French}
}

@article{BCDP15,
  title={Equivalence of gradient flows and entropy solutions for singular nonlocal interaction equations in {1D}},
  author={Bonaschi, Giovanni A and Carrillo, Jos{\'e} A and Di Francesco, Marco and Peletier, Mark A},
  journal={ESAIM: Control, Optimisation and Calculus of Variations},
  volume={21},
  number={2},
  pages={414--441},
  year={2015}
}

@mastersthesis{B2013,
    author={Bonaschi, Giovanni A.},
    title={Gradient flows driven by a non-smooth repulsive interaction potential},
    school = {University of Pavia},
    year = {2013},
    type = {Master’s thesis},
    url = {https://arxiv.org/abs/1310.3677}
}

@article{LT2004,
  title={Long-time asymptotics of kinetic models of granular flows},
  author={Li, Hailiang and Toscani, Giuseppe},
  journal={Archive for Rational Mechanics and Analysis},
  volume={172},
  pages={407--428},
  year={2004},
  publisher={Springer}
}

@article{PS2010,
  title={Evolution Equations for Maximal Monotone Operators: Asymptotic Analysis in Continuous and Discrete Time},
  author={Peypouquet, Juan and Sorin, Sylvain},
  journal={Journal of Convex Analysis},
  volume={17},
  number={3\&4},
  pages={1113--1163},
  year={2010}
}

@article{szekely2002,
author={Sz\'ekely, Gabor J.},
title = {E-statistics: The Energy of Statistical Samples},
journal = {Technical Report, Bowling Green University},
year = {2002},
}

@article{szekely_energy,
title = {Energy statistics: A class of statistics based on distances},
journal = {Journal of Statistical Planning and Inference},
volume = {143},
number = {8},
pages = {1249-1272},
year = {2013},
author = {Gábor J. Székely and Maria L. Rizzo}
}

@article{mmd_energy_eq,
author = {Dino Sejdinovic and Bharath Sriperumbudur and Arthur Gretton and Kenji Fukumizu},
title = {{Equivalence of distance-based and RKHS-based statistics in hypothesis testing}},
volume = {41},
journal = {The Annals of Statistics},
number = {5},
publisher = {Institute of Mathematical Statistics},
pages = {2263 -- 2291},
year = {2013}
}

@article{CLW2016,
author = {José A. Carrillo and Dejan  Slepčev and Lijiang  Wu},
title = {Nonlocal-interaction equations on uniformly prox-regular sets},
journal = {Discrete and Continuous Dynamical Systems},
volume = {36},
number = {3},
pages = {1209-1247},
year = {2016},
issn = {1078-0947},
doi = {10.3934/dcds.2016.36.1209},
url = {https://www.aimsciences.org/article/id/7d22068c-dd57-4a13-b37b-ad56f3c2cf3f}
}

@article{FR2011,
title = {Stability of stationary states of non-local equations with singular interaction potentials},
journal = {Mathematical and Computer Modelling},
volume = {53},
number = {7},
pages = {1436-1450},
year = {2011},
issn = {0895-7177},
doi = {https://doi.org/10.1016/j.mcm.2010.03.021},
url = {https://www.sciencedirect.com/science/article/pii/S0895717710001391},
author = {Klemens Fellner and Gaël Raoul},
}

@incollection{CCH2014,
author="Carrillo, Jos{\'e} Antonio
and Choi, Young-Pil
and Hauray, Maxime",
editor="Muntean, Adrian
and Toschi, Federico",
title="The derivation of swarming models: Mean-field limit and {W}asserstein distances",
bookTitle="Collective Dynamics from Bacteria to Crowds: An Excursion Through Modeling, Analysis and Simulation",
year="2014",
publisher="Springer Vienna",
address="Vienna",
pages="1--46",
isbn="978-3-7091-1785-9",
doi="10.1007/978-3-7091-1785-9_1",
}

@article{BD2008,
  title={Large time behavior of nonlocal aggregation models with nonlinear diffusion},
  author={Burger, Martin and Di Francesco, Marco},
  journal={Networks and Heterogeneous Media},
  volume={3},
  number={4},
  pages={749--785},
  year={2008},
  publisher={American Institute of Mathematical Sciences},
  doi = {10.3934/nhm.2008.3.749}
}

@article{CDEFS2020,
  title={Measure solutions to a system of continuity equations driven by {N}ewtonian nonlocal interactions},
  author={Carrillo, Jos{\'e} Antonio and Di Francesco, Marco and Esposito, Antonio and Fagioli, Simone and Schmidtchen, Markus},
  journal={Discrete and Continuous Dynamical Systems},
  volume={40},
  number={2},
  pages={1191--1231},
  year={2020},
  publisher={Discrete and Continuous Dynamical Systems}
}

@article{DB2022,
  title={Polynomial-time Sparse Measure Recovery: From Mean Field Theory to Algorithm Design},
  author={Daneshmand, Hadi and Bach, Francis},
  journal={arXiv preprint arXiv:2204.07879},
  year={2022}
}

@inproceedings{DLJ2023,
  title={Efficient displacement convex optimization with particle gradient descent},
  author={Daneshmand, Hadi and Lee, Jason D and Jin, Chi},
  booktitle={International Conference on Machine Learning},
  pages={6836--6854},
  year={2023},
  organization={PMLR}
}

@article{CHHRS2023,
  title={Sampling via gradient flows in the space of probability measures},
  author={Chen, Yifan and Huang, Daniel Zhengyu and Huang, Jiaoyang and Reich, Sebastian and Stuart, Andrew M},
  journal={arXiv preprint arXiv:2310.03597},
  year={2023}
}

@article{FFR2024,
  title={Measure solutions, smoothing effect, and deterministic particle approximation for a conservation law with nonlocal flux},
  author={Di Francesco, M and Fagioli, S and Radici, E},
  journal={arXiv preprint arXiv:2406.03837},
  year={2024}
}

@article{HMVV2024,
  title={Nonlinear aggregation-diffusion equations with {R}iesz potentials},
  author={Huang, Yanghong and Mainini, Edoardo and V{\'a}zquez, Juan Luis and Volzone, Bruno},
  journal={Journal of Functional Analysis},
  volume={287},
  number={2},
  pages={110465},
  year={2024},
  publisher={Elsevier}
}

@article{F2016,
  title={Scalar conservation laws seen as gradient flows: known results and new perspectives},
  author={Di Francesco, Marco},
  journal={ESAIM: Proceedings and Surveys},
  volume={54},
  pages={18--44},
  year={2016},
  publisher={EDP Sciences}
}

@article{BV2023,
  title={On the global convergence of {W}asserstein gradient flow of the {C}oulomb discrepancy},
  author={Boufad{\`e}ne, Siwan and Vialard, Fran{\c{c}}ois-Xavier},
  journal={arXiv preprint arXiv:2312.00800},
  year={2023}
}

@article{dFFHM2014,
  title={Asymptotic Behavior of Gradient Flows Driven by Nonlocal Power Repulsion and Attraction Potentials in One Dimension},
  author={Di Francesco, Marco and Fornasier, Massimo and Hütter, Jan-Christian and Matthes, Daniel},
  journal={SIAM Journal on Mathematical Analysis},
  volume={46},
  number={6},
  pages={3814--3837},
  year={2014},
  publisher={SIAM}
}

\appendix
\section{Additional Material}\label{sec:appendix}
\subsection{Supplement to Section~\ref{sec:prelim}}

\begin{remark} \label{rem:cone}
To be completely accurate, the set of quantile functions is the set of increasing and left-continuous functions (not equivalence classes) in $\mathcal L_2(0, 1)$.
However, each almost everywhere increasing function $f \in L_2(0, 1)$ has a unique left-continuous representative:
consider a function $\tilde{f} \in \mathcal L_2(0, 1)$ in the equivalence class $f$, which is everywhere increasing.
Then $\tilde{f}$ has at most countably many jump discontinuities $(s_k)_{k \in \N}$.
Define $f_0 \colon (0, 1) \to \R$ via $s \mapsto \lim_{t \uparrow s} \tilde{f}(t)$.
Then $f_0$ differs from $\tilde{f}$ only at possibly $(s_k)_{k \in \N}$, and $f_0$ is increasing everywhere and left-continuous.

 Indeed, $\C(0, 1)$ is closed, since for any sequence $(f_n)_{n \in \N} \subset \C(0, 1)$ converging to $f$ in $L_2(0, 1)$ as $n \to \infty$,
we have $f_n(s) \to f(s)$ along a subsequence for almost all $s \in (0, 1)$. Since the pointwise limit of increasing functions is increasing, we see that $f \in \C(0, 1)$.
\end{remark}

\noindent
\textbf{Proof of Proposition~\ref{prop:conv-geo}.}
    Let $\lambda \in \R$ and $F\colon L_2(0,1)\to\R$ be $\lambda$-convex, that is,
    \begin{equation*}
        F\left(t u + (1 - t) v\right)
        \le t F(u) + (1 - t) F(v) - \frac{1}{2} \lambda t (1 - t) \| u - v \|^2
    \end{equation*}
    for all $u, v \in \dom(F) \cap \C(0,1)$ and all $t \in (0, 1)$.
    Let $\gamma \colon [0,1] \to \P_2(\R)$ be any geodesic with $\gamma_0 = \mu$
    and $\gamma_1 = \nu$.
    Since $\mu\mapsto Q_\mu$ is an isometry by Theorem~\ref{prop:Q}, 
    the curve $t\mapsto Q_{\gamma_t}$ is a geodesic in $L_2(0,1)$. 
    Since $L_2(0,1)$ is a linear space, the only geodesics are straight
    line segments, so we obtain 
    \begin{equation*}
        Q_{\gamma_t}
        = (1 - t) Q_{\mu} + t Q_{\nu}.
    \end{equation*}
    Finally, we conclude
    \begin{align*}
        \F(\gamma_t)
        & = F(Q_{\gamma_t})
        = F\big((1 - t) Q_{\mu} + t Q_{\nu}\big) \\
        & \le t F(Q_{\mu}) + (1 - t) F(Q_{\nu}) - \frac{1}{2} \lambda t (1 - t) \| Q_{\mu} - Q_{\nu} \|_{L_2(0, 1)}^2 \\
        & = t \F(\mu) + (1 - t) \F(\nu) - \frac{1}{2} \lambda t (1 - t) W_2^2(\mu, \nu).\qedhere
    \end{align*}
The lower semicontinuity of $\mathcal F$ follows directly from the lower semicontinuity of $F$ using the isometric embedding, see Theorem~\ref{prop:Q}. \hfill $\Box$
\medskip

In the proof of Theorem \ref{thm:L2_representation} we assumed the existence of some $\tilde \tau > 0$ such that $I + \tilde\tau v_t$ is monotonically increasing, in order to employ the isometry to $L_2(0,1)$, see Theorem \ref{prop:Q}. Unfortunately, 
$I + \tilde\tau v_t$ does not need to be monotone in the general case. Hence, we include a rigorous treatment.

\noindent
\textbf{Appendix to the proof of Theorem~\ref{thm:L2_representation}.}
We only need to justify the use of the isometry to $L_2(0,1)$, see Theorem \ref{prop:Q}. Fix $t \ge 0$.
Since $v_t \in \text{T}_{\gamma_t} \mathcal P_2(\R)$, there exist sequences of positive numbers $\lambda_n > 0$ and mappings $T_n : \R \to \R$, which are optimal mappings from $\gamma_t$ to $(T_n)_\sharp \gamma_t$, such that 
\begin{equation}
    \zeta_n \coloneqq  v_t - \lambda_n(T_n - I) \to 0 \quad \text{in } L_2(\R; \R; \gamma_t).
\end{equation}
Setting $\xi_n \coloneqq (I + h_n v_t) - T_n$, where $h_n \coloneqq \lambda_n^{-1}$, we obtain
\begin{equation}\label{o(h)}
    \frac{\xi_n}{h_n} = \zeta_n \to 0 \quad \text{in } L_2(\R; \R; \gamma_t).
\end{equation}
Now, we consider two cases.

\underline{Case 1:} $(\lambda_n)_n$ is unbounded. Here, we can extract a subsequence (still labeled $(\lambda_n)_n$) such that $\lambda_n \uparrow \infty$, i.e. $h_n \downarrow 0$.
By the reverse triangle inequality, we have
\begin{align}
    &|W_2(\gamma_{t +h_n}, (I+h_n v_t)_\sharp \gamma_t) - W_2(\gamma_{t +h_n}, (T_n)_\sharp \gamma_t)|
    \le W_2((I+h_n v_t)_\sharp \gamma_t, (T_n)_\sharp \gamma_t)\\
    \le &\left( \int_{\R} |x + h_n v_t(x) - T_n(x)|^2 \d \gamma_t(x)
    \right)^\frac{1}{2}
    = \left( \int_{\R} |\xi_n(x)|^2 \d \gamma_t(x)
    \right)^\frac{1}{2}
    = \|\xi_n\|_{L_2(\R;\R;\gamma_t)}.
\end{align}
Here, the second estimate uses the plan $\tilde T_\sharp \gamma_t$, where $\tilde T : x \mapsto (x+h_n v_t(x), T_n(x))$. 
Dividing by $h_n$ and using \eqref{o(h)}, we obtain together with \cite[Proposition 8.4.6]{BookAmGiSa05} that
\begin{equation}
     0 =\lim_{n \to \infty}\frac{W_2(\gamma_{t+h_n},(I+h_n \, v_t)_\#\gamma_t)}{h_n} 
     = \lim_{n \to \infty}\frac{W_2(\gamma_{t+h_n},(T_n)_\#\gamma_t)}{h_n}.
\end{equation}
Since the $T_n$ are optimal mappings in $\R$, they are increasing. In particular, $T_n \circ g(t)$ is increasing. Hence, the isometry to $L_2(0,1)$ from Theorem \ref{prop:Q} yields
\begin{align}
   0 &=  \lim_{n \to \infty}\frac{W_2(g(t+h_n)_\sharp \Lambda_{(0,1)}, (T_n \circ g(t))_\sharp \Lambda_{(0,1)})}{h_n}\\
   &= \lim_{n \to \infty}\frac{\|g(t+h_n) - T_n \circ g(t) \|_{L_2(0,1)}}{h_n}\\
   &= \lim_{n \to \infty}\frac{\|g(t+h_n) - (g(t) +h_n v_t \circ g(t) - \xi_n \circ g(t)) \|_{L_2(0,1)}}{h_n}\\
   &= \lim_{n \to \infty}\|\frac{g(t+h_n) - g(t)}{h_n} - v_t \circ g(t) + \frac{\xi_n \circ g(t)}{h_n} \|_{L_2(0,1)}.
\end{align}
The reverse triangle inequality leads to 
\begin{align}
    \limsup_{n \to \infty}\|\frac{g(t+h_n) - g(t)}{h_n} - v_t \circ g(t) \|_{L_2(0,1)}
    &\le 
    \lim_{n \to \infty} \|\frac{\xi_n \circ g(t)}{h_n} \|_{L_2(0,1)}\\
    &= \lim_{n \to \infty} \|\frac{\xi_n}{h_n} \|_{L_2(\R;\R;\gamma_t)} = 0.
\end{align}
Note that the second step does \emph{not} need $\xi_n$ to be monotone, and the third step is exactly \eqref{o(h)}.
Since we have already argued that $g:[0,\infty) \to L_2(0,1)$ is Lipschitz continuous, the derivative $g'(t) = \lim_{\tau \to 0} \frac{g(t+\tau)-g(t)}{\tau}$ must exist for a.e. $t$, and by above, it coincides with $v_t \circ g(t)$. 
Now, the remaining part follows as in the given proof of Theorem \ref{thm:L2_representation}.

\underline{Case 2:} $(\lambda_n)_n$ is bounded. Here, we can extract a subsequence $(\lambda_n)_n$ which converges to some $\lambda \ge 0$. Again, consider two subcases:

\underline{Case 2i:} $\lambda = 0$. Here, we have $h_n \to \infty$. 
We know that $T_n = I + h_n v_t - \xi_n$ is increasing. For any $x,y\in \R$ with $x < y$, it then holds $x + h_n v_t(x) - \xi_n(x) \le y + h_n v_t(y) - \xi_n(y)$. Dividing by $h_n$ yields $\lambda_n x + v_t(x) - \zeta_n(x) \le \lambda_n y + v_t(y) - \zeta_n(y)$. For $n \to \infty$ (along a subsequence), this becomes $v_t(x) \le v_t(y)$ for $\gamma_t$-a.e.\ $x,y$. Hence, $v_t$ admits a non-decreasing representative $\tilde{v}_t$.
Thus, the map $I + \tau \tilde v_t$ is increasing as well. One quickly verifies that also $\tilde v_t \in \text{T}_{\gamma_t} \mathcal P_2(\R)$ and $\tilde v_t$ solves \eqref{eq:CE}, and we can proceed as in the given proof of Theorem \ref{thm:L2_representation}.

\underline{Case 2ii:} $\lambda > 0$. Here, we have the convergence (along a subsequence)
\begin{equation}
    T_n \to I + \lambda^{-1} v_t
\end{equation}
for $\gamma_t$-a.e.\ $x$. Since the pointwise limit of increasing functions is increasing, the map $I + \lambda^{-1} v_t$ is increasing outside a $\gamma_t$-null set. By a suitable modification $\tilde v_t$ of $v_t$, we again obtain that $I + \lambda^{-1} \tilde v_t$ is increasing everywhere, and the arguments from the subcase 2i) conclude the proof.
\hfill $\Box$

\begin{theorem}[Existence and regularity of strong solutions to \eqref{eq:cauchy} {\cite[Thm.~3.1, p.~54]{B1973}, \cite{CL1971}}] \label{theorem:BrezisRegularity}
    Let $H$ be a Hilbert space and $A \colon H \to 2^H$ be a maximal monotone operator.
    For all $g_0 \in \dom(A)$, there exists a unique function $g \colon [0, \infty) \to H$ such that
    \begin{enumerate}
        \item 
        $g(t) \in \dom(A)$ for all $t > 0$,
        \item 
        $g$ is Lipschitz continuous on $[0, \infty)$, that is, $\frac{\d g}{\d t} \in L_{\infty}((0, \infty), H)$ (in the sense of distributions and in the strong sense a.e.) and
        \begin{equation*}
            \left\| \frac{\d g}{\d t}  \right\|_{L_{\infty}((0, \infty), H)}
            \le \| A^{\circ} g_0 \|,
        \end{equation*}
where $A^{\circ} z \coloneqq \argmin \{ \| y \| : y \in A z \}$ denotes the minimal norm selection,
        \item 
        $\frac{\d g}{\d t}(t) \in - A g(t)$ for almost all $t > 0$,
        \item 
        $g(0) = g_0$.
    \end{enumerate}
    Furthermore, $g$ satisfies the following properties:
    \begin{enumerate}
    \setcounter{enumi}{4}
        \item 
        $g$ admits a right derivative for all $t > 0$ and $\frac{\d^+ g}{\d t}(t) + A^{\circ} g(t) = 0$ for all $t > 0$,

        \item 
        $t \mapsto A^{\circ} g(t)$ is right-continuous and $t \mapsto \| A^{\circ} g(t) \|$ is decreasing.

        \item
        $g$ is given by the exponential formula
        \begin{equation} \label{eq:exponentialFormula}
            g(t)
            = \lim_{n \to \infty} \left( I + \frac{t}{n} A\right)^{-n} g_0,
        \end{equation}
        uniformly on compact time intervals.
    \end{enumerate}
\end{theorem}

\subsection{Supplement to Section~\ref{sec:mmd} }
\noindent
\textbf{Proof of Proposition~\ref{lemma:FnuGamma0}.}
    The functional $F_{\nu}$ is convex, since for $u, v \in L_2(0, 1)$ and $t \in (0, 1)$ we have
    \begin{align*}
        F_{\nu}\big(t u + (1 - t) v\big)
        & = \int_{0}^{1} \bigg( (1 - 2 s) \big( t u(s) + (1 - t) v(s) \big) \\
        & + \int_{0}^{1} | t u(s) + (1 - t) v(s) - Q_{\nu}(t) | \d{t} \bigg) \d{s} \\
        & \le t \int_{0}^{1} \left( (1 - 2 s)  u(s)  + \int_{0}^{1} | u(s) - Q_{\nu}(t) | \d{t} \right) \d{s} \\
        & + (1 - t) \int_{0}^{1} \left( (1 - 2 s)  v(s)  + \int_{0}^{1} | v(s) - Q_{\nu}(t) | \d{t} \right) \d{s} \\
        & = t F_{\nu}(u) + (1-t) F_{\nu}(v).
    \end{align*}
    To show that $F_{\nu}$ is finite everywhere, notice that for all $u \in L_2(0,1)$,
    \begin{align*}
        |F_{\nu}(u)|
        & \le \int_{0}^{1} | 1 - 2 s | | u(s) | + |u(s)| + \|Q_\nu\|_{L_1(0,1)}\d{s} \\
        &\le 2\|u\|_{L_2(0,1)} + \|Q_\nu\|_{L_2(0,1)} < \infty.
        \end{align*}
    Lastly, we will show the continuity of $F_{\nu}$.
    Suppose that $(u_n)_{n \in \N} \subset L_2(0, 1)$ converges to $u \in L_2(0, 1)$.
    Then there exists a subsequence $(u_{n_k})_{k \in \N}$ such that $u_{n_k}(s) \to u(s)$ for $k \to \infty$ for a.e. $s \in (0, 1)$, and there exists a $m \in L_1(0, 1)$ such that $| u_{n_k}(s) | \le m(s)$ for a.e. $s \in (0, 1)$.\\
    Hence, applying Lebesgue's dominated convergence theorem twice yields 
    \begin{align*}
        \lim_{k \to \infty} F_{\nu}(u_{n_k})
        & = \lim_{k \to \infty} \int_{0}^{1} \left( (1 - 2 s) u_{n_k}(s)  + \int_{0}^{1} | u_{n_k}(s) - Q_{\nu}(t) | \d{t} \right) \d{s}
         = F_{\nu}(u).
    \end{align*}
    We can apply the dominated convergence theorem because in both cases the integrands are bounded above by an integrable function as follows: for almost all $s \in (0, 1)$, we have by applying the triangle inequality many times, that
    \begin{align*}
        & \quad \left| (1 - 2 s)  u_{n_k}(s)  + \int_{0}^{1} | u_{n_k}(s) - Q_{\nu}(t) | \d{t} \right| \\
        & \le \underbrace{| 1 - 2 s |}_{\le 1} | u_{n_k}(s)  | + \int_{0}^{1} | u_{n_k}(s) - Q_{\nu}(t) | \d{t} \\
        & \le | u_{n_k}(s) |  + | u_{n_k}(s) | + \int_{0}^{1} | Q_{\nu}(t) | \d{t} \\
        & \le 2 m(s)  + \| Q_{\nu} \|_{L_1(0, 1)}.
    \end{align*}
    This function is integrable because $m \in L_1(0, 1)$ and $Q_{\nu} \in L_2(0, 1) \subset L_1(0, 1)$.
    Analogously, for any $s \in (0, 1)$, the integrand $(0, 1) \ni t \mapsto | u_{n_k}(s) - Q_{\nu}(t) |$ is bounded by the integrable function $t \mapsto | u_{n_k}(s) | + | Q_{\nu}(t) |$. 
\hfill $\Box$
\subsection{Supplement to Section~\ref{sec:inv} }
 \textbf{Proof of Lemma~\ref{lem:helper}.}
    Let $s_0 \in (0,1)$ be arbitrary and take a sequence $(s_n) \subset (0,1)$ of values $s$ where \eqref{eq:RC-start} holds, and such that $s_n \uparrow s_0$. Then, it holds
    \begin{equation}
      h(s_0) + \vareps s_0 = \lim_{n \to \infty} h(s_n) + \vareps s_n \ge \lim_{n \to \infty} u(s_n) + \vareps R_\nu^-(u(s_n)) = u(s_0) + \vareps R_\nu^-(u(s_0))
    \end{equation}
    by the left-continuity of $h$, $u$ and $R_\nu^-$, and since $u$ is increasing. Also, one has
    \begin{equation}
      h(s_0) + \vareps s_0 = \lim_{n \to \infty} h(s_n) + \vareps s_n \le \lim_{n \to \infty} u(s_n) + \vareps R_\nu^+(u(s_n)) \le u(s_0) + \vareps R_\nu^+(u(s_0)),
    \end{equation}
    since $R_\nu^+$ is upper semicontinuous. Since $s_0 \in (0,1)$ was arbitrary, this shows the claim.
    \hfill $\Box$ 
\\[2ex]
\textbf{Proof of Proposition~\ref{l:DLOW-eq}.}
  Let us write $R_\mu \coloneqq R_\mu^+$.
  \\
 i) \emph{'Only if'}:
  Let $x, y \in R_\mu^{-1}((0,1))$ with $x > y$. W.l.o.g. suppose $R_\mu(x) > R_\mu(y)$ and let $\varepsilon > 0$ such that $R_\mu(x) > R_\mu(y) + \varepsilon$. Then, it holds $Q_\mu(R_\mu(x)) \le x$ and $Q_\mu(R_\mu(y) + \varepsilon) > y$ by the Galois inequalities \eqref{eq:galois}. Now, using that $Q_\mu \in D_{L, \text{low}}$, it follows that
  \begin{align*}
      L|R_\mu(x) - (R_\mu(y)+\varepsilon)| &\le |Q_\mu(R_\mu(x)) - Q_\mu(R_\mu(y)+\varepsilon)|\\
      &= Q_\mu(R_\mu(x)) - Q_\mu(R_\mu(y)+\varepsilon)\\
      &\le x - y = |x - y|.
  \end{align*}
  Letting $\varepsilon \downarrow 0$ shows
  \begin{equation}
   |R_\mu(x) - R_\mu(y)| \le \frac{1}{L} |x-y| \quad \text{for all } x,y \in R_\mu^{-1}((0,1)).  
  \end{equation}
  Since $R_\mu \equiv 0$ on $R_\mu^{-1}((-\infty,0])$, $R_\mu \equiv 1$ on $R_\mu^{-1}([1,\infty))$, and since $R_\mu$ is continuous (on $\R$) by Remark~\ref{l:DLOW-eq-lem}, $R_\mu$ is Lipschitz continuous on $\R$ with constant $\le L^{-1}$, which proves the 'only-if' part.
  \\[1ex]
  \emph{'If'}:
  Since $R_\mu$ is (Lipschitz) continuous, one has $R_\mu (Q_\mu (s)) = s$ for all $s \in (0,1)$. It immediately follows
  \begin{equation}
      |Q_\mu(s_1) - Q_\mu(s_2)| \ge L | R_\mu(Q_\mu(s_1)) - R_\mu(Q_\mu(s_2))| = L|s_1 - s_2|
  \end{equation}
  for all $s_1,s_2 \in (0,1)$, which concludes the proof of part i).
  \\[1ex]
  ii) \emph{'Only if'}:
  Since $Q_\mu$ is (Lipschitz) continuous, $R_\mu$ is strictly increasing on $R_\mu^{-1}((0,1))$ by Remark~\ref{l:DLOW-eq-lem}, and hence, $Q_\mu (R_\mu (x)) = x$ for all $x \in R_\mu^{-1}((0,1))$. It immediately follows
  \begin{equation}
      |R_\mu(x) - R_\mu(y)| \ge \frac{1}{L} | Q_\mu(R_\mu(x)) - Q_\mu(R_\mu(y))| = \frac{1}{L}|x - y|
  \end{equation}
  for all $x, y \in R_\mu^{-1}((0,1))$. This inequality directly extends to all $x,y \in \overline{R_\mu^{-1}((0,1))}$, since $R_\mu$ is increasing.
  \\[1ex]
  \emph{'If'}:
  Let $s_1, s_2 \in (0,1)$ with $s_1 > s_2$. W.l.o.g. suppose $Q_\mu(s_1) > Q_\mu(s_2)$ and let $\varepsilon > 0$ such that $Q_\mu(s_1) - \vareps > Q_\mu(s_2)$. Then, it holds $R_\mu(Q_\mu(s_2)) \ge s_2$ and $R_\mu(Q_\mu(s_1) - \varepsilon) < s_1$ by the Galois inequalities \eqref{eq:galois}. Now, using the lower Lipschitz bound of $R_\mu$, it follows
  \begin{align*}
      \frac{1}{L}|(Q_\mu(s_1)-\varepsilon) - Q_\mu(s_2)| 
      &\le |R_\mu(Q_\mu(s_1)-\varepsilon) - R_\mu(Q_\mu(s_2))|\\
      &= R_\mu(Q_\mu(s_1)-\varepsilon) - R_\mu(Q_\mu(s_2))\\
      &\le s_1 - s_2 = |s_1 - s_2|.
  \end{align*}
  Letting $\varepsilon \downarrow 0$ shows $Q_\mu \in D_L$, which completes the proof of ii).
  \\[1ex]
    Finally, assume that $\mu$ is absolutely continuous.
    Then the bounds for its density follow by applying the fundamental theorem of calculus for absolutely continuous $R_\mu$, i.e., for all  $x \ge y$,
    \begin{equation}
        |R_\mu(x) - R_\mu(y)| = R_\mu(x) - R_\mu(y) = \int_y^x f_\mu(\xi) \d\xi.
        \qquad \Box
    \end{equation}

\subsection{Smoothness Properties: Different Approach}\label{app:diff}

In the rest of this section, we highlight the statements from Section~\ref{sec:inv} from a different point of view. Instead of working with the exponential formula \eqref{eq:form}, we deal with the pointwise differential inclusion \eqref{eq:pw-ode}, allowing for a more direct approach. \footnote{But note that the exponential approach of Section~\ref{sec:inv} also yields information about the implicit Euler steps.}

The following theorem is analogous to Theorem~\ref{thm:regularisation-Lip_1}. Let us remark that both theorems also hold true in a local sense, i.e., on subintervals $I \subseteq (0,1)$.

\begin{theorem}\label{thm:regularisation-Lip_1_Robert}
        Let $I \coloneqq (a,b)$ be such that $Q_{\mu_0}(I) \subset \supp \nu$.
        If $Q_{\mu_0}$ is locally upper Lipschitz on $I$ 
        with constant $L_{\mu_0} \ge 0$,
        and if $R^\pm_\nu$ is locally lower Lipschitz
        on $\conv(Q_{\mu_0}(I) \cup Q_\nu(I))$
        with constant $L_{\nu}^{-1} > 0$,
        then $g(t)$ is locally upper Lipschitz on $I$ with 
        \begin{equation}\label{eq:invariance1_Robert}
            {\rm Lip}( g(t) ) 
            \le 
            L_{\mu_0} \, e^{-2t L_\nu^{-1}}
            + L_\nu \, (1 - e^{-2t L_\nu^{-1}}).
        \end{equation}
\end{theorem}

\begin{proof}
    W.l.o.g.
    we consider $s_1<s_2$ in $I$
    with $Q_{\mu_0} (s) < Q_\nu (s)$ for all $s_1 \le s \le s_2$.
    Due to the local lower Lipschitz property of $R_{\nu}^+$,
    the derivatives of $g_{s_1}$ are bounded by
    \begin{equation*}
        \dot g_{s_1}(t) = 2 s_1 - 2 R_\nu^+ (g_{s_1}(t))
        \ge
        2 L_\nu^{-1} (Q_\nu(s_1) - g_{s_1}(t))
    \end{equation*}
    for almost every $t \ge 0$.
    Solving the differential equation 
    with the lower bound 
    and the initial value $Q_{\mu_0}(s_1)$,
    we obtain
    \begin{equation}
        g_{s_1}(t) 
        \ge
        Q_\nu(s_1) + (Q_{\mu_0}(s_1) - Q_\nu(s_1)) \, 
        e^{-2t L_\nu^{-1}}.
    \end{equation}
    A similar procedure for $g_{s_2}$ yields
    \begin{equation}
        \dot g_{s_2}(t) 
        = 
        2 s_2 - 2 R^+_\nu(g_{s_2}(t))
        \le
        2 s_2 - 2 R^+_\nu(Q_{\mu_0}(s_2) 
        + 2 L_\nu^{-1}( Q_{\mu_0}(s_2) - g_{s_2}(t))
    \end{equation}
    and
    \begin{equation}
        g_{s_2}(t)
        \le
        L_\nu( s_2 - R^{+}_\nu(Q_{\mu_0}(s_2)))
        + Q_{\mu_0}(s_2)
        - L_\nu( s_2 - R^{+}_\nu(Q_{\mu_0}(s_2)))
        \, e^{-2 t L_\nu^{-1}}.
    \end{equation}
    Taking the difference 
    and exploiting that
    \begin{align}
        Q_{\mu_0}(s_2) - Q_\nu(s_1) 
        &\le 
        L_\nu (R^+_\nu(Q_{\mu_0}(s_2)) - R^+_\nu(Q_\nu(s_1))) 
        \\
        &\le
        L_\nu (R^+_\nu(Q_{\mu_0}(s_2)) - s_1)
    \end{align}
    by the lower Lipschitz property of $R^+_\nu$,
    we obtain
    \begin{align}
        \lvert g_{s_2}(t) - g_{s_1}(t) \rvert
        &\le
        (Q_{\mu_0}(s_2) - Q_{\mu_0}(s_1)) \, e^{-2t L_\nu^{-1}}
        \\
        &+
        (Q_{\mu_0}(s_2) - Q_\nu(s_1)
        + L_\nu s_2 - L_\nu R^+_\nu(Q_{\mu_0}(s_2))
        \, (1 - e^{-2t L_\nu^{-1}} )
        \\
        &\le
        ( L_{\mu_0} \, e^{-2t L_\nu^{-1}} 
        +
        L_\nu \, (1 - e^{-2t L_\nu^{-1}}  ))
        \, (s_2 - s_1).
    \end{align}
\end{proof}


We can analogously argue for the lower Lipschitz property.

The following theorem is a counterpart to Theorem~\ref{p:rangeDc}, but with a subtle difference in its nature: it states an invariance of the continuity in a \textit{single point} $s \in (0,1)$, assuming in addition that $s$ is a continuity point of the target $Q_\nu$. On the other hand, Theorem~\ref{p:rangeDc} states an invariance of continuity on \textit{neighborhoods} of $s$, without any further assumptions on the target.

\begin{theorem}\label{p:rangeDc-Robert}
    Let $s \in (0,1)$ be a continuity point of $Q_\nu$ and $Q_{\mu_0}$,
    then $g(t)$ is continuous at $s$ for all $t \ge 0$.
\end{theorem}

\begin{proof}
    Since $Q_\nu$ is continuous in $s$,
    for all $x < Q_\nu(s)$,
    there exists $\delta_x >0$ such that
    $\Phi_{\nu,s'}(x) < \infty$
    for all $s' \in (s - \delta_x, s + \delta_x)$.
    The integrand of $\Phi_{\nu,s}$ is monotone in $s$
    such that $\Phi_{\nu,s}(x)$ is continuous in $s$
    due to Lebesgue's dominated convergence theorem.
    To show that $g(t)$ is continuous in $s$,
    we use the $\epsilon$--$\delta$ criterion.
    First, 
    we discuss the case $2t = \Phi_{\nu,s}(x)$ with $x < Q_\nu(s)$. 
    For arbitrary $\epsilon > 0$,
    we find $x_1 \in (x - \epsilon, x)$ and $\delta_1 > 0$
    such that
    $\Phi_{\nu,s'}(x_1) < 2t$
    for all $s' \in (s - \delta_1, s)$.
    Due to the monotonicity of $\Phi_{\nu,s}(x)$ in $s$,
    we have $\Phi_{\nu,s}^{-1}(2t) \in (x-\epsilon,x)$
    for all $s' \in (s - \delta_1, s)$.
    Analogously,
    we find $x_2 \in (x, x+ \epsilon)$ and $\delta_2 > 0$
    such that
    $\Phi_{\nu,s'}(x_2)$ is finite 
    and $\Phi_{\nu,s'}(x_2) > 2t$
    for all $s' \in (s, s + \delta_2)$.
    Again due to the monotonicity,
    we have 
    $\Phi_{\nu,s}^{-1}(2t) \in (x, x + \epsilon)$
    for all $s' \in (s, s + \delta_2)$.
    Taking $\delta \coloneqq \min\{\delta_1, \delta_2\}$,
    we finally have 
    \begin{equation}
        [g(t)](s') \in ([g(t)](s) - \epsilon, [g(t)](s) + \epsilon)   
        \quad\text{for all}\quad
        s' \in (s-\delta, s+ \delta).
    \end{equation}
    Second,
    for the case $2t \ge \Phi_{\nu,s}(Q_\nu(s))$,
    we have to argue slightly differently.
    Here, 
    we find $x_1 \in (Q_\nu(s) - \epsilon, Q_\nu(s))$
    and $\delta_1 > 0$
    such that
    $\Phi_{\nu,s'}(x_1) < 2t$
    and thus
    $\Phi_{\nu,s'}^{-1} (2t) > Q_\nu(s) - \epsilon$
    for all $s' \in (s - \delta_1, s + \delta_1)$.
    Due to the continuity of the target $Q_\nu$ in $s$,
    we find $\delta_2 > 0$
    such that
    $Q_\nu(s') < Q_\nu(s) + \epsilon)$
    for all $s' \in (s-\delta_2,s+\delta_2)$.
    Taking $\delta \coloneqq \min\{\delta_1, \delta_2\}$,
    we again have 
    \begin{equation}
        [g(t)](s') \in ([g(t)](s) - \epsilon, [g(t)](s) + \epsilon)   
        \quad\text{for all}\quad
        s' \in (s-\delta, s+ \delta).\vspace{-3mm}
    \end{equation}
\end{proof}


  \subsection{Quantile functions plots for the experiments}

In this supplementary section, we plot the quantile functions belonging to the densities in the Figures \ref{fig:pt-target} and \ref{fig:Norm_To_Norm_Different_Means} -- \ref{fig:Folded_Norm_to_Folded_Norm} from the main text.
Here, one can nicely verify the results from the previous subsection and Section \ref{sec:inv}. 

Generally speaking, we observe that the quantile functions differ mostly from the target at the boundary where their slopes rise sharply, i.e., at values of $s$ close to zero or one. This is the case, e.g., when the supports (or mass) of the initial and target measure are disjoint, and the flow's mass gets arbitrarily slim due to Theorem~\ref{p:rangeDc}~ii).

In the fully discrete case, see \figref{fig:pt-target_q}, we observe the following:
\begin{enumerate}
    \item The initial quantile $Q_{\mu_0}$ is discontinuous, i.e., $Q_{\mu_0} \not\in D_c$, and so we can not apply Theorem~\ref{p:rangeDc}~i). In fact, $Q_{\gamma_t} \not\in D_c$ for all $t \ge 0$, i.e., the support of $\gamma_t$ stays disconnected.
    \item The target quantile fulfills no lower Lipschitz condition, i.e., $Q_{\nu} \notin D_L^-$, so Theorem \ref{thm:regularisation-Lip} cannot be applied. Indeed, it holds $Q_{\gamma_t} \notin D_L^-$ and $\gamma_t$ is \emph{not} absolutely continuous for all sufficiently late times $t \ge 0$.
\end{enumerate}

\begin{figure}
    \includegraphics[width=.32\textwidth]{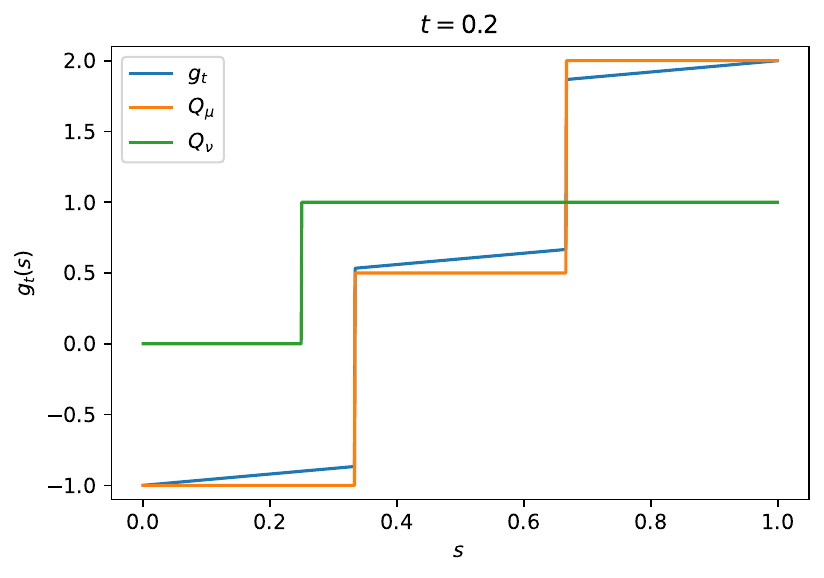}
    \includegraphics[width=.32\textwidth]{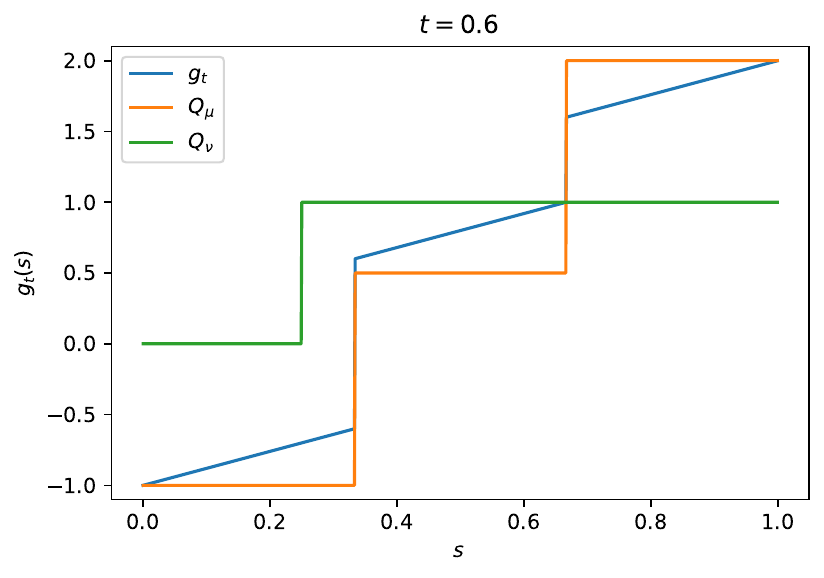}
    \includegraphics[width=.32\textwidth]{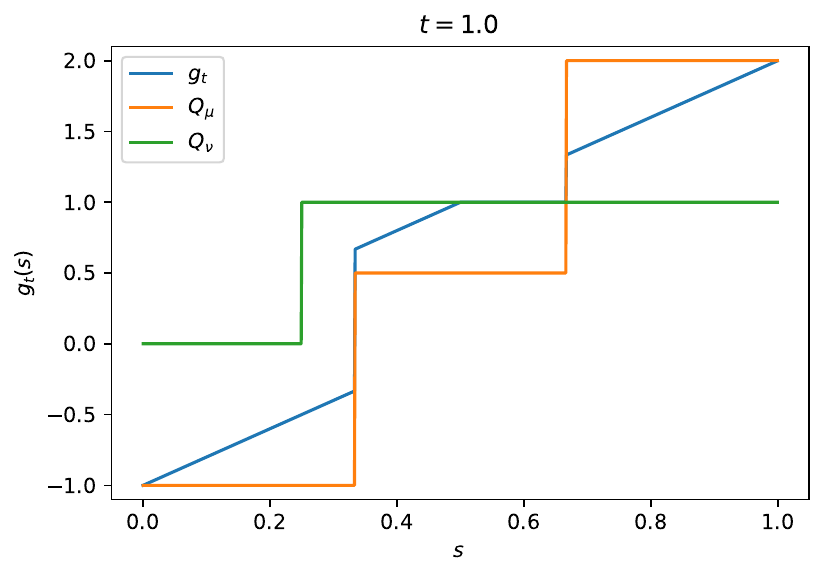} \\
    \includegraphics[width=.32\textwidth]{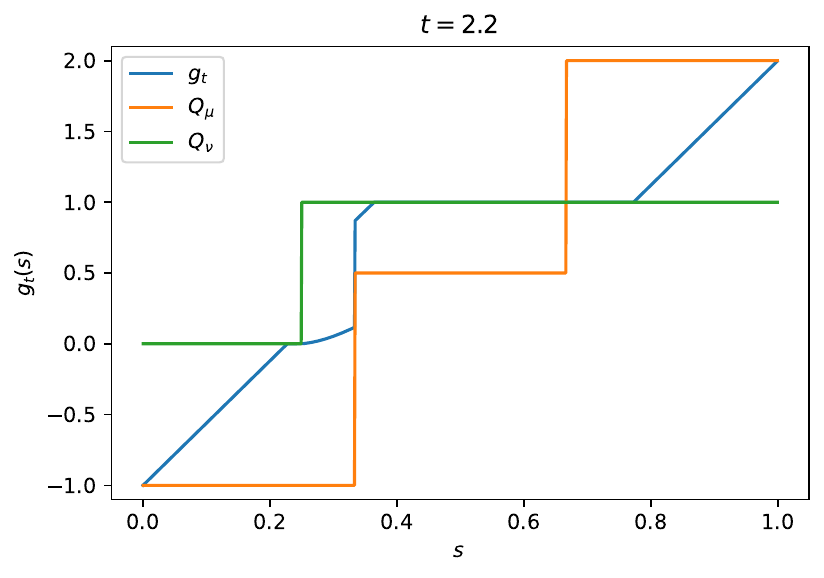}
    \includegraphics[width=.32\textwidth]{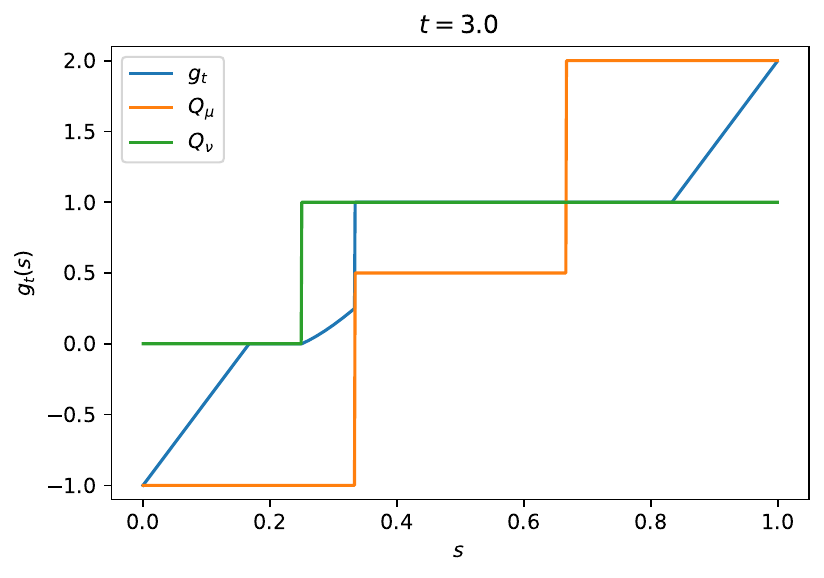}
    \includegraphics[width=.32\textwidth]{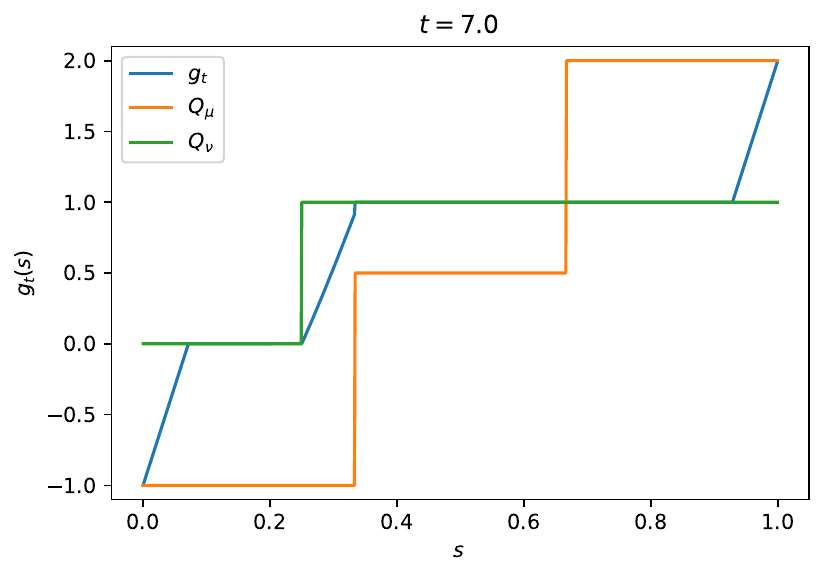}
    \caption{The quantile functions of the Wasserstein gradient flow between the 3-point and 2-point measures given
    in Example~\ref{example:TwoToThree}.
    For the corresponding densities, see \figref{fig:pt-target}.}
    \label{fig:pt-target_q}
\end{figure}

Concerning the examples with absolutely continuous targets $\nu$, plotted in the Figures \ref{fig:Norm_To_Norm_Different_Means_q} -- \ref{fig:Folded_Norm_to_Folded_Norm_q}, we can verify the following facts:

\begin{enumerate}
    \item The range of the quantile functions $Q_{\gamma_t}$ does not leave the convex hull of the ranges of $Q_{\mu_0}$ and $Q_{\nu}$ by Proposition~\ref{p:rangeDab}. That is, the support of $\gamma_t$ stays enclosed by the convex hull of the initial and target supports.
    \item The quantiles from Figures \ref{fig:Norm_To_Norm_Different_Means_q} -- \ref{fig:Folded_Norm_to_Folded_Norm_q} all satisfy a lower Lipschitz condition because their respective targets do, as described by Theorem \ref{thm:regularisation-Lip}. Hence, the flow $\gamma_t$ stays an absolutely continuous measure for all times.
    \item The quantiles from Figures \ref{fig:Norm_To_Norm_Different_Means_q} -- \ref{fig:Folded_Norm_to_Folded_Norm_q} all stay continuous because their priors are continuous. Furthermore, their ranges grow monotonically over time, see Theorem \ref{p:rangeDc}. This translates to the flow's support staying convex and growing monotonically.
    \item Lastly, we have ${\rm Lip}(Q_{\mu_0}) = {\rm Lip}(Q_{\nu}) = \infty$ in all of these examples, so Theorem~\ref{thm:regularisation-Lip_1} does not apply. Indeed, $Q_{\gamma_t} \not\in D_{L}^+$ for all $L > 0$. In other words, in our examples, the flow's mass gets arbitrarily slim in certain regions.
\end{enumerate}

\begin{figure}
    \centering
    \includegraphics[width=.32\textwidth]{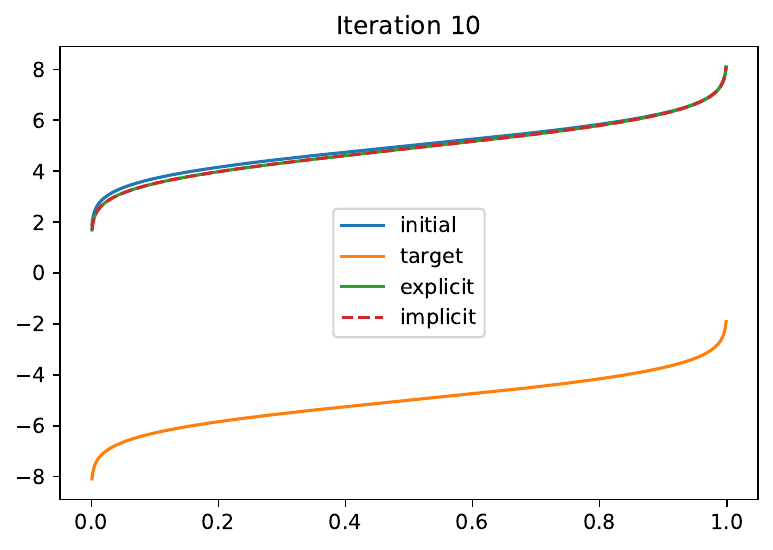}
    \includegraphics[width=.32\textwidth]{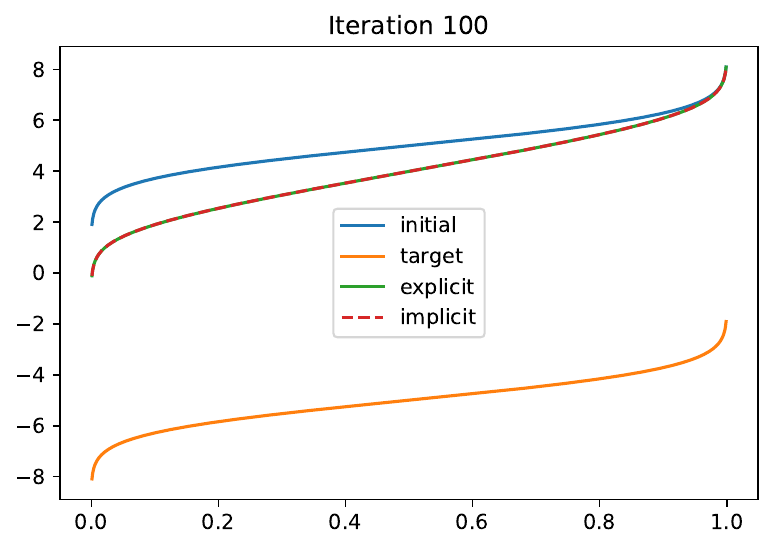}
    \includegraphics[width=.32\textwidth]{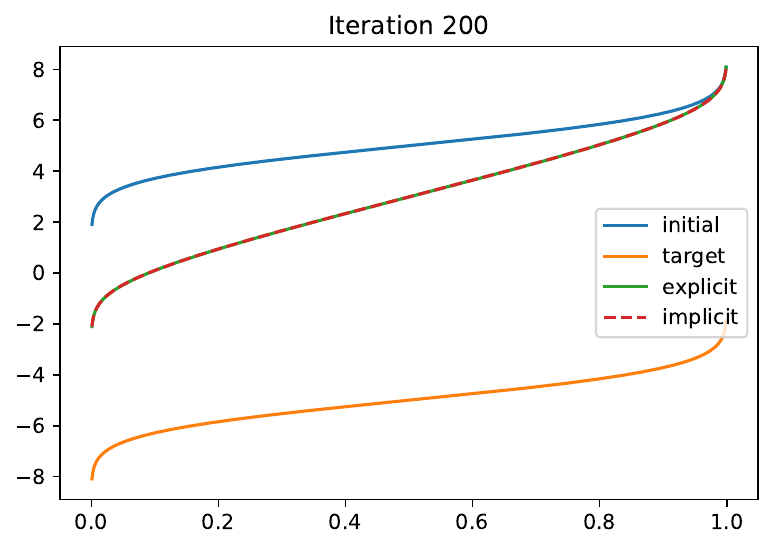}
    \\
    \includegraphics[width=.32\textwidth]{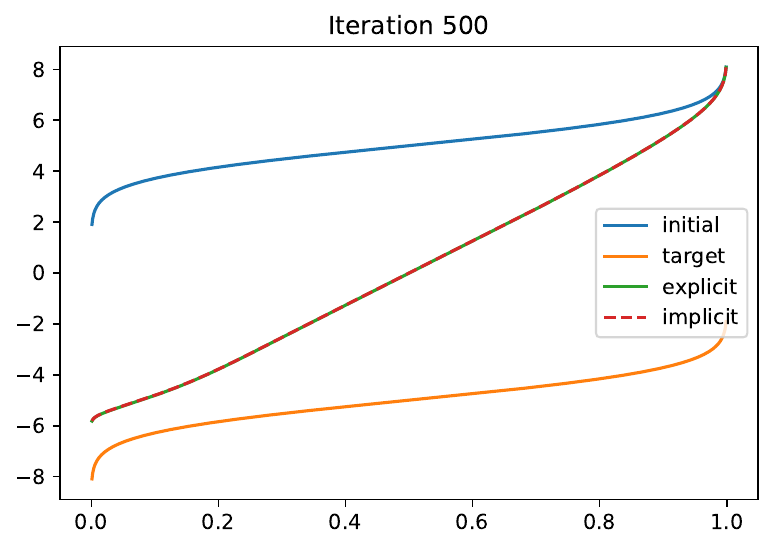}
    \includegraphics[width=.32\textwidth]{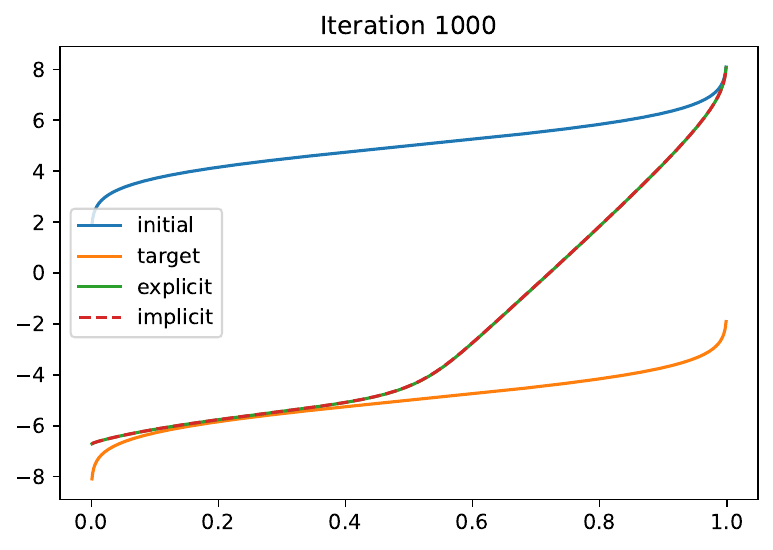}
    \includegraphics[width=.32\textwidth]{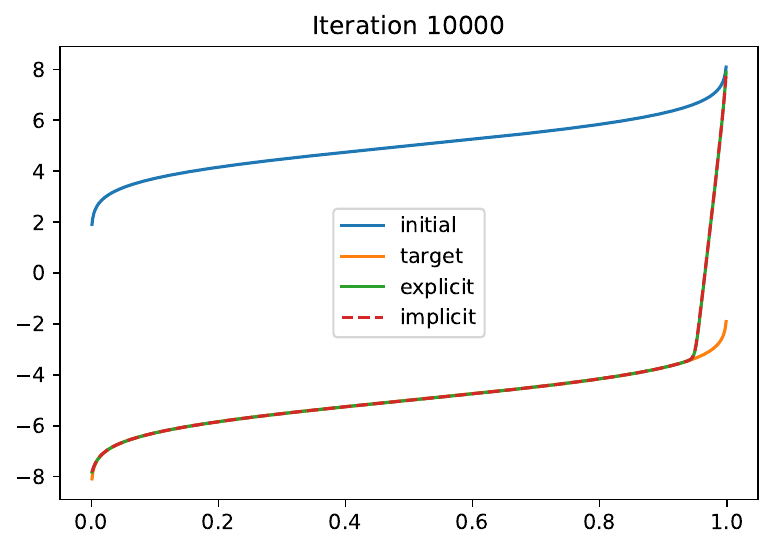}
    \caption{Comparison of the quantile functions belonging to the implicit (red) and explicit (green) Euler schemes between two Gaussians $\mu_0 \sim \NN(5, 1)$ and $\nu \sim \NN(-5, 1)$ with $\tau = \tfrac{1}{100}$. For the corresponding densities, see \figref{fig:Norm_To_Norm_Different_Means}.
    }
    \label{fig:Norm_To_Norm_Different_Means_q}
\end{figure}

\begin{figure}
    \centering
    \includegraphics[width=.32\textwidth]{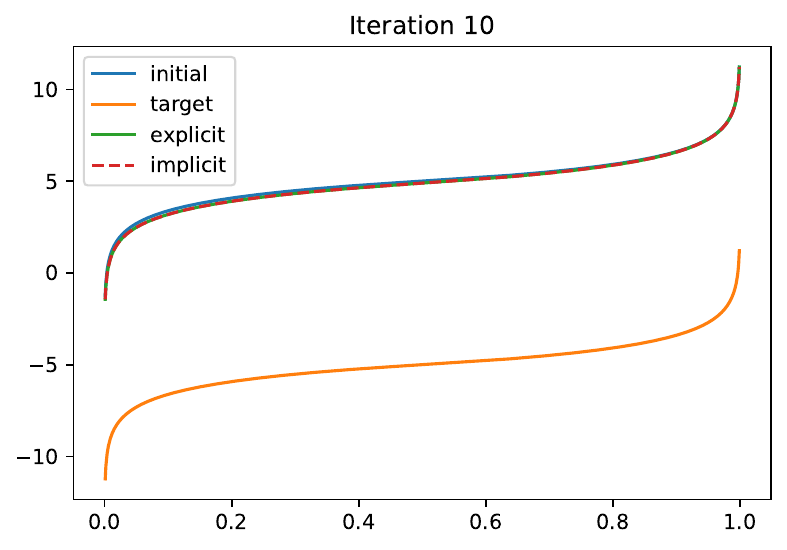}
    \includegraphics[width=.32\textwidth]{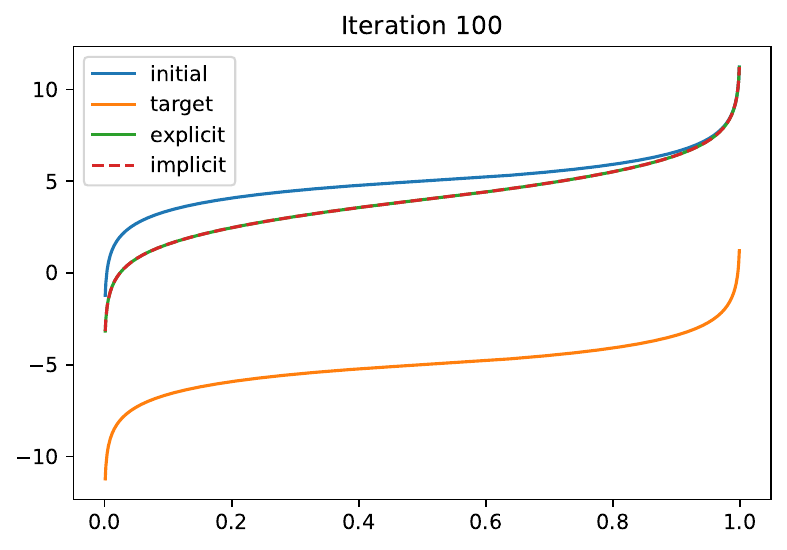}
    \includegraphics[width=.32\textwidth]{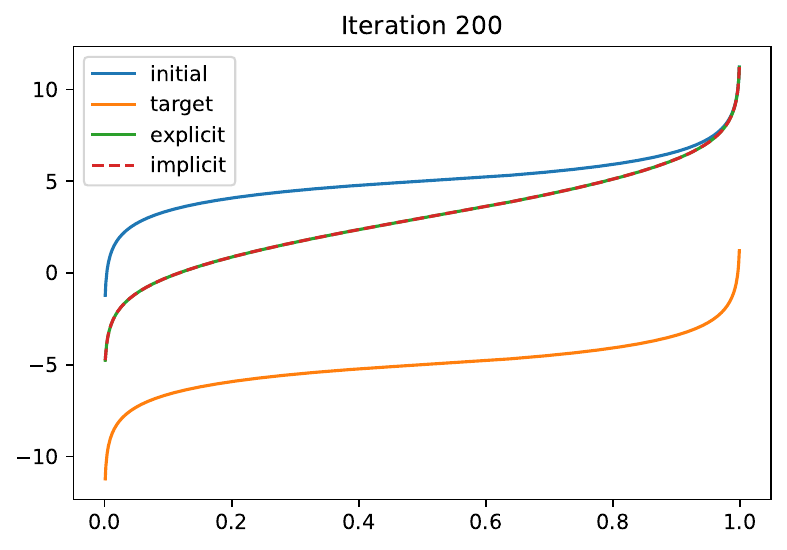}
    \includegraphics[width=.32\textwidth]{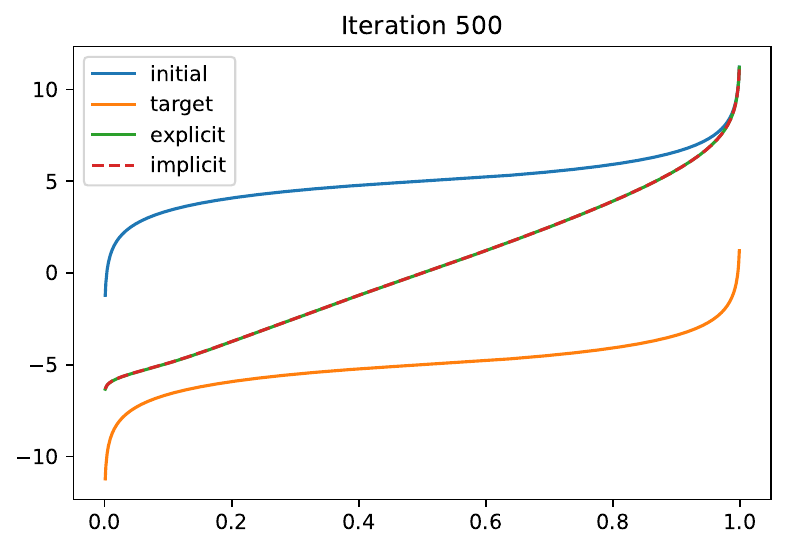}
    \includegraphics[width=.32\textwidth]{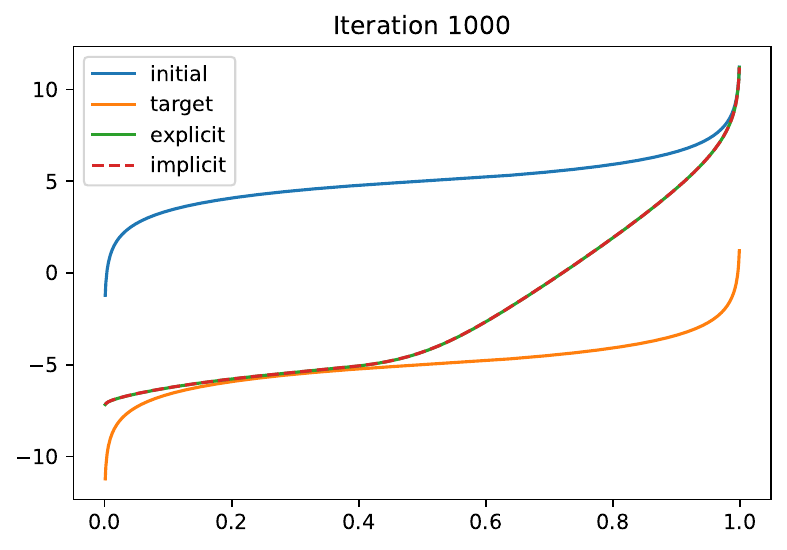}
    \includegraphics[width=.32\textwidth]{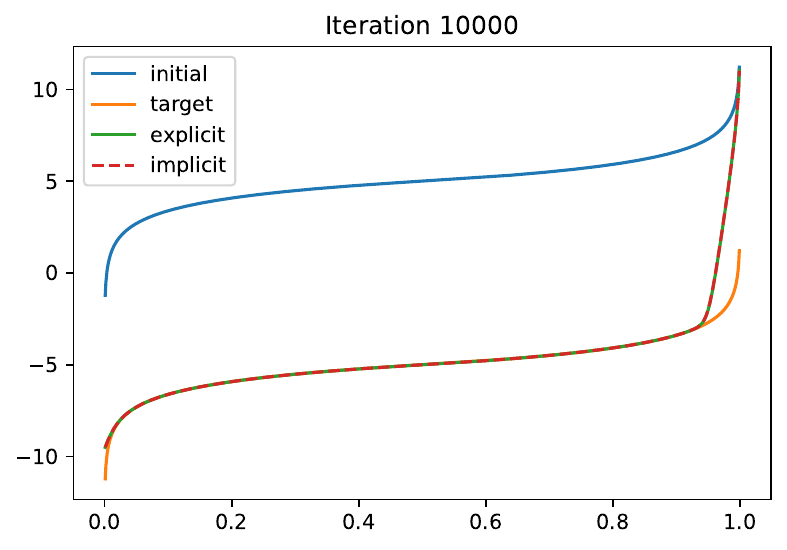}
    \caption{Comparison of the quantile functions belonging to the implicit (red) and explicit (green) Euler schemes between two Laplacians $\mu_0 \sim \mathcal L(5, 1)$ and $\nu \sim \mathcal L(-5, 1)$ with $\tau = \tfrac{1}{100}$. For the corresponding densities, see \figref{fig:Laplace_to_Laplace}.
    }
    \label{fig:Laplace_to_Laplace_q}
\end{figure}

\begin{figure}
    \centering
    \includegraphics[width=.32\textwidth]{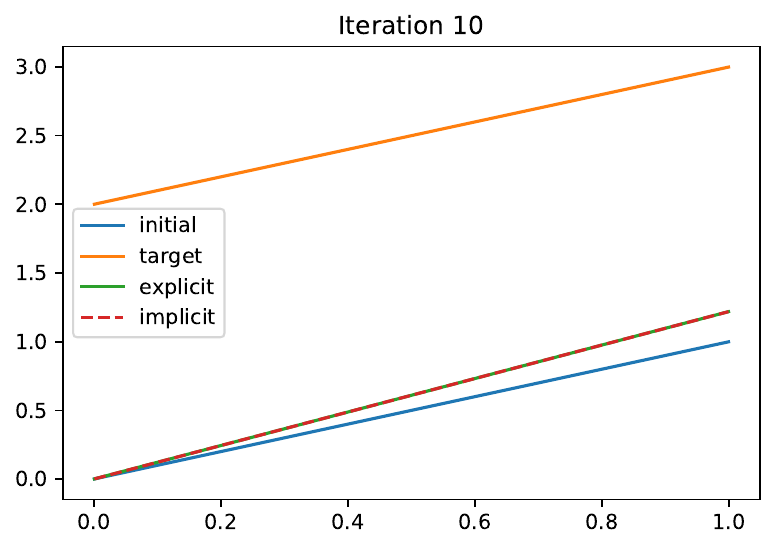}
    \includegraphics[width=.32\textwidth]{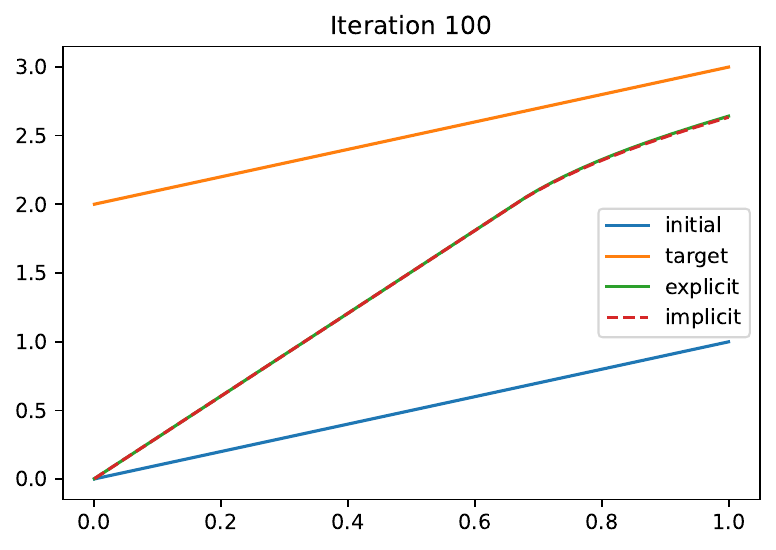}
    \includegraphics[width=.32\textwidth]{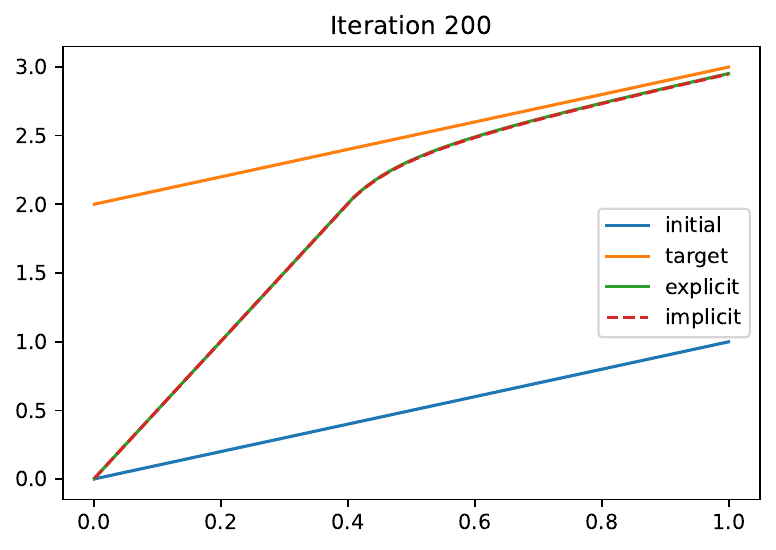}
    \includegraphics[width=.32\textwidth]{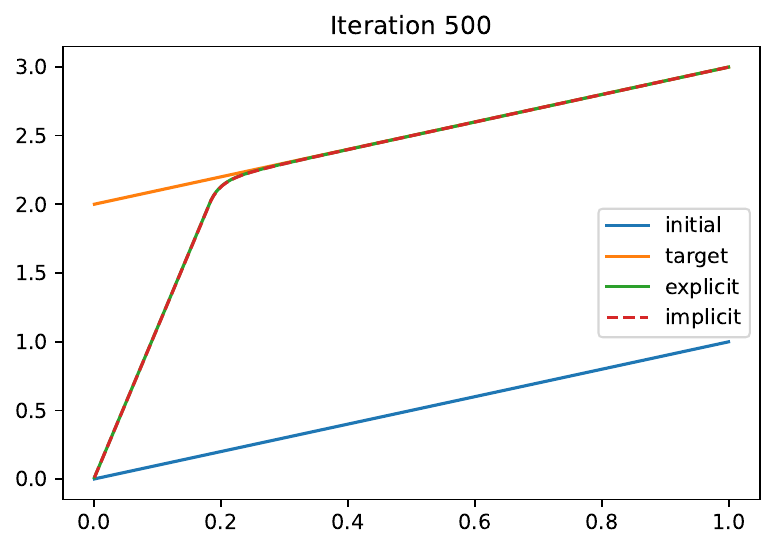}
    \includegraphics[width=.32\textwidth]{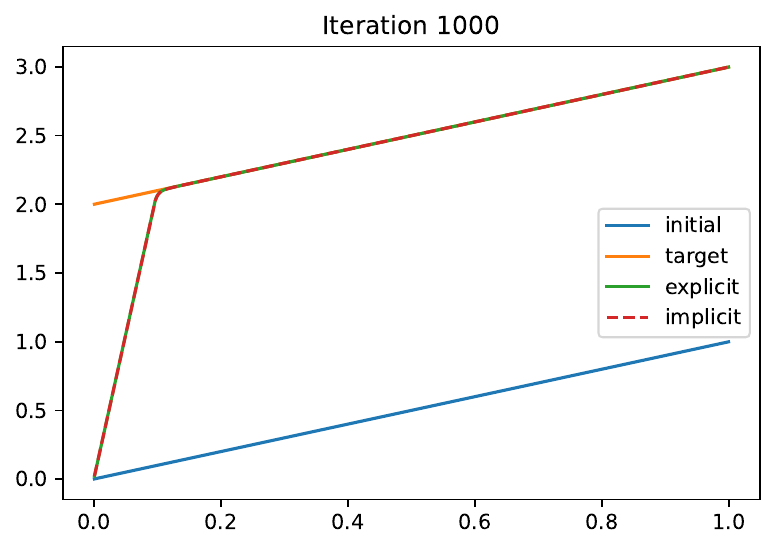}
    \includegraphics[width=.32\textwidth]{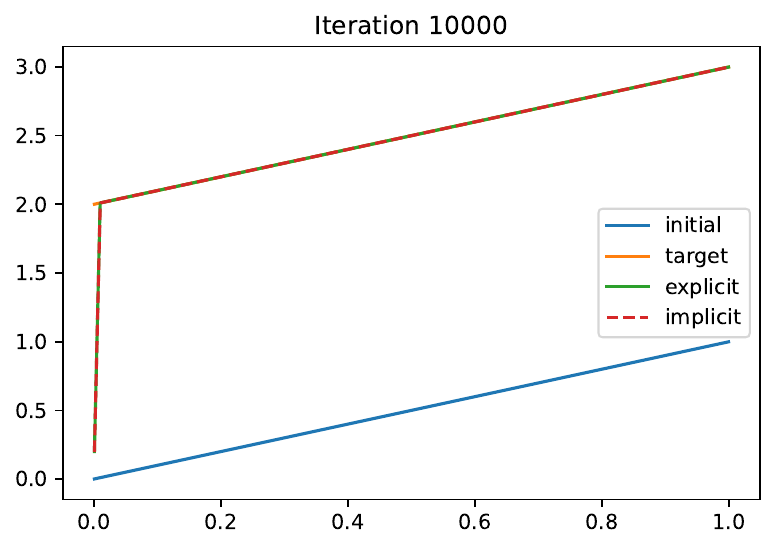}
    \caption{Comparison of the quantile functions belonging to the implicit (red) and explicit (green) Euler schemes between two uniform distributions $\mu_0 \sim \mathcal U([0,1])$ and $\nu \sim \mathcal U([2,3])$ with $\tau = \tfrac{1}{100}$.
    For the corresponding densities, see \figref{fig:Uniform-to-Uniform}.
    }
    \label{fig:Uniform_to_Uniform_q}
\end{figure}

\begin{figure}
    \centering
    \includegraphics[width=.32\textwidth]{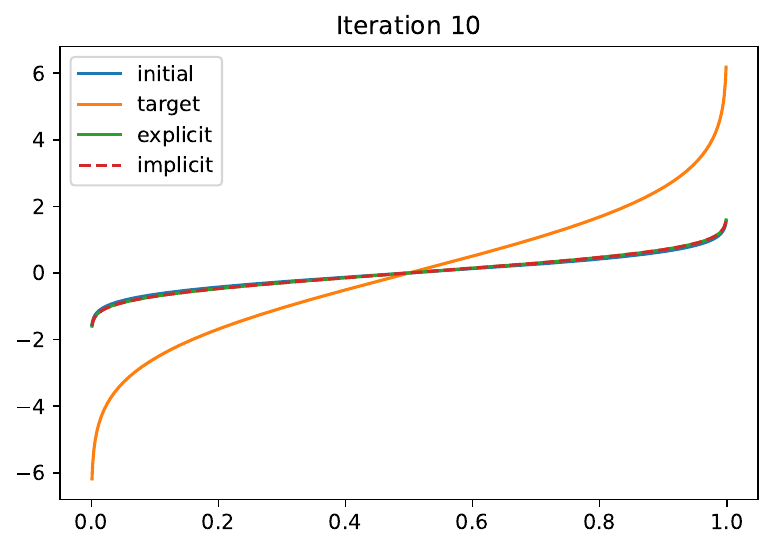}
    \includegraphics[width=.32\textwidth]{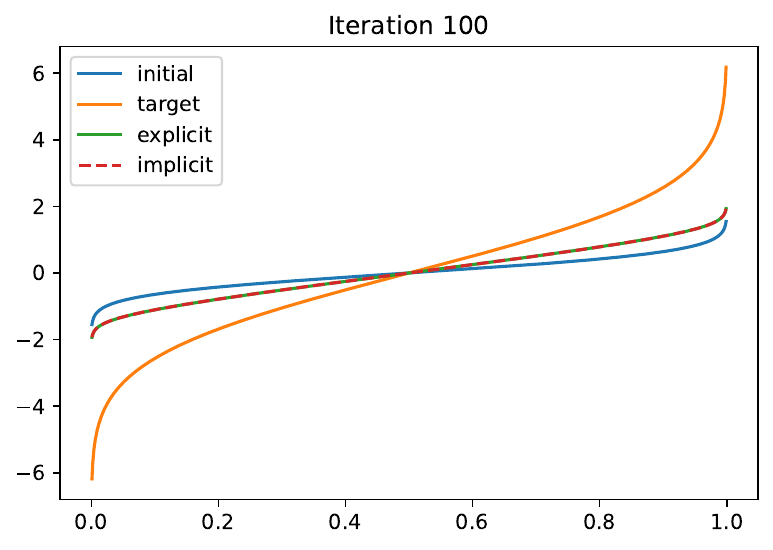}
    \includegraphics[width=.32\textwidth]{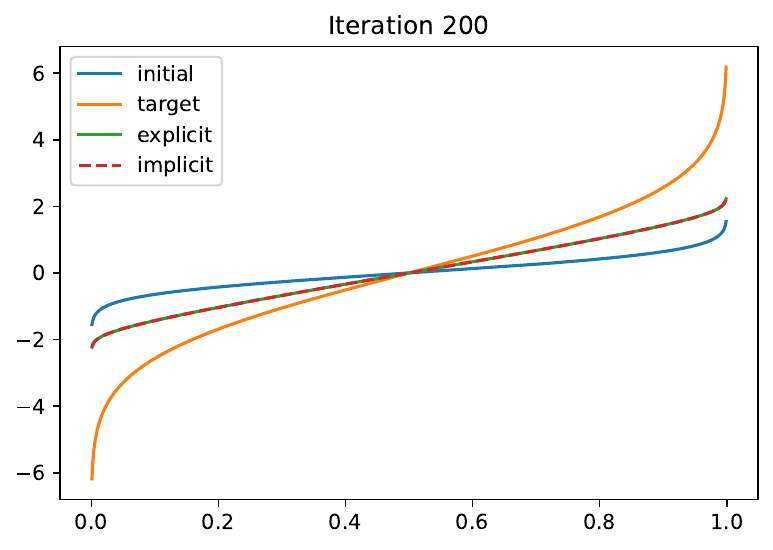}
    \includegraphics[width=.32\textwidth]{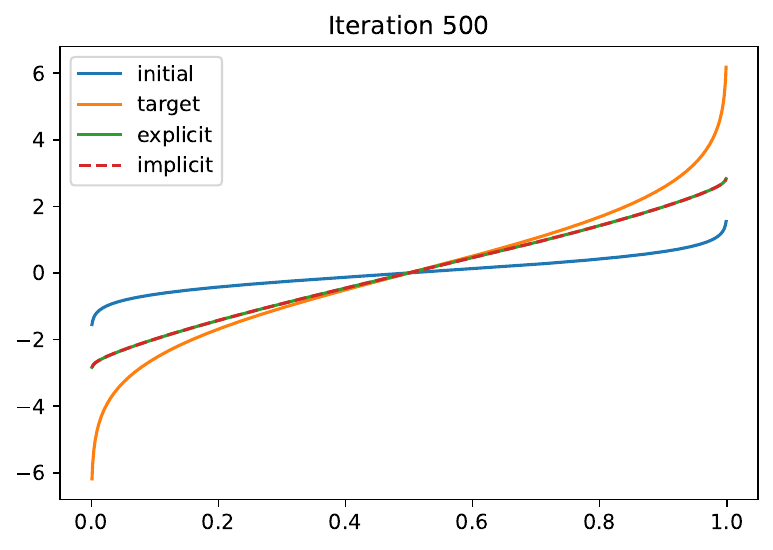}
    \includegraphics[width=.32\textwidth]{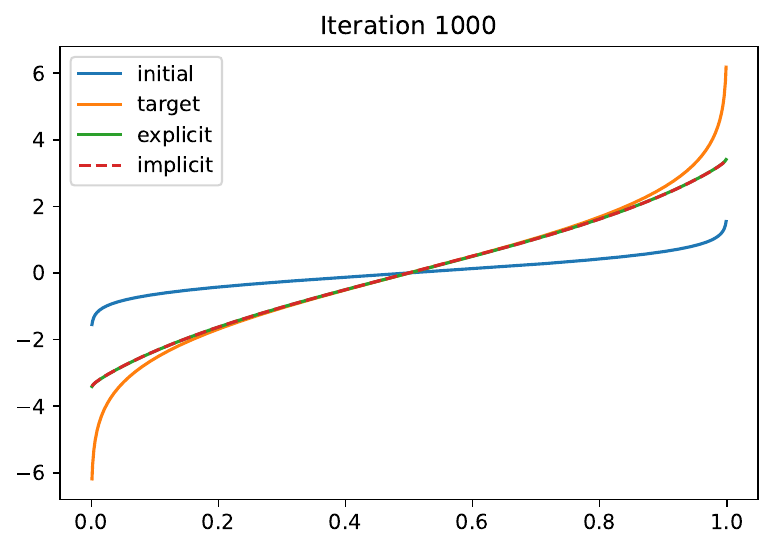}
    \includegraphics[width=.32\textwidth]{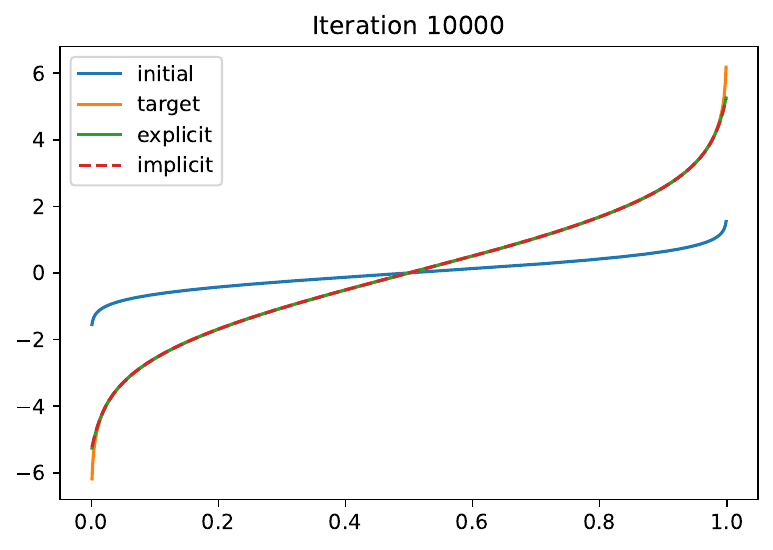}
    \caption{Comparison of the quantile functions belonging to the implicit (red) and explicit (green) Euler schemes between two Gaussians $\mu_0 \sim \NN(0, \tfrac{1}{\sqrt{2}})$ and $\nu \sim \NN(0, \sqrt{2})$ with $\tau = \tfrac{1}{100}$.
    For the corresponding densities, see \figref{fig:Norm_To_Norm_Different_Scales}.
    }
    \label{fig:Norm_To_Norm_Different_Scales_q}
\end{figure}

\begin{figure}
    \centering
    \includegraphics[width=.32\textwidth]{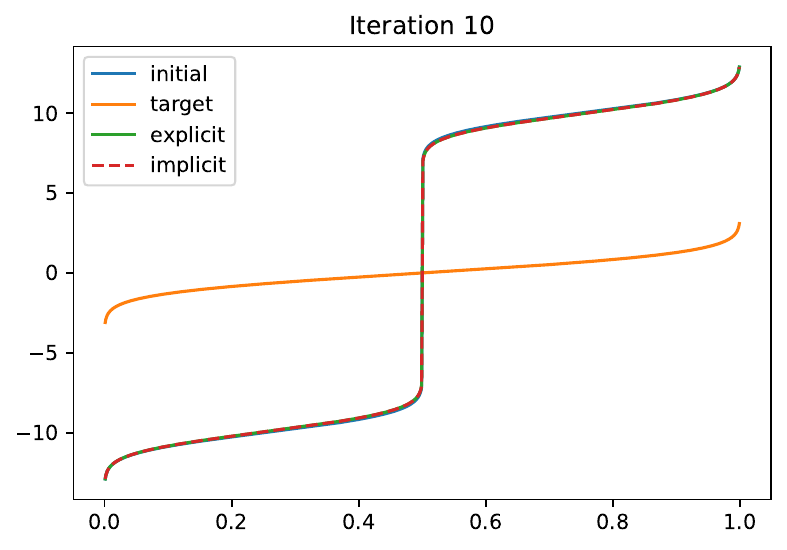}
    \includegraphics[width=.32\textwidth]{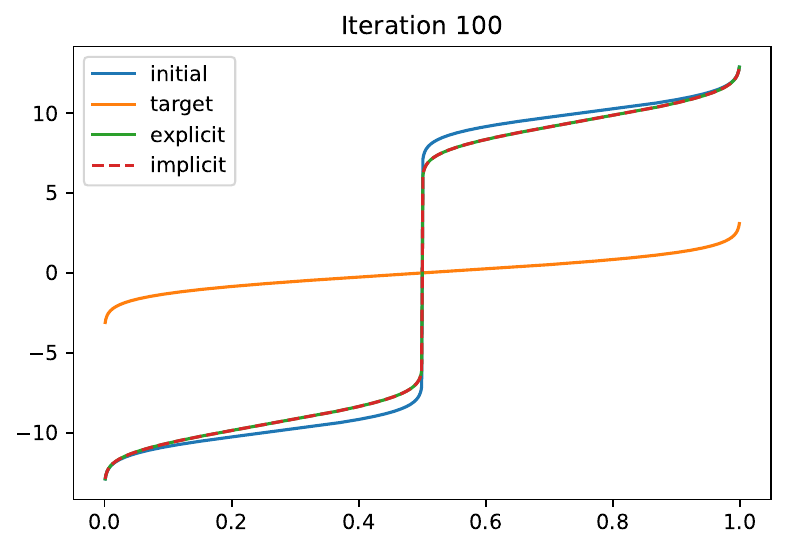}
    \includegraphics[width=.32\textwidth]{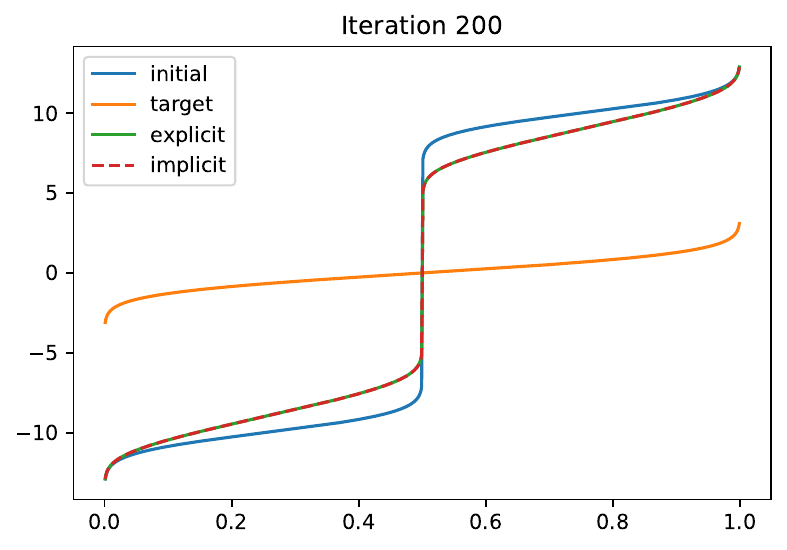}
    \includegraphics[width=.32\textwidth]{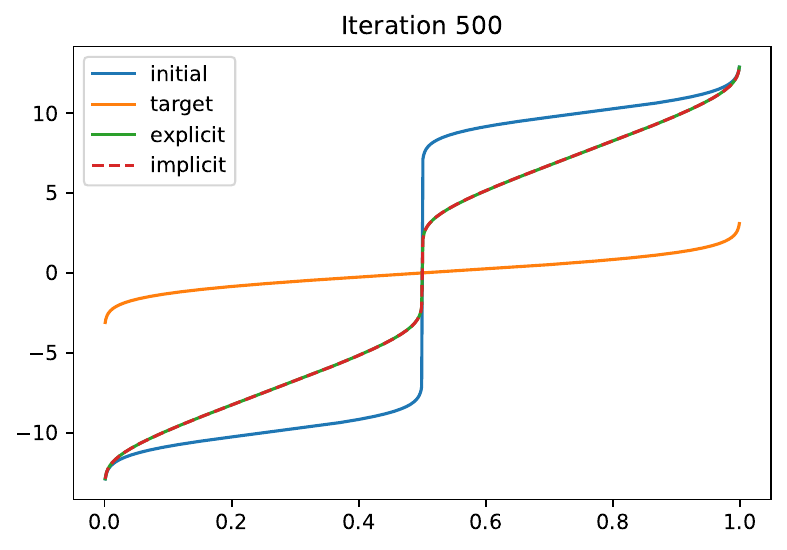}
    \includegraphics[width=.32\textwidth]{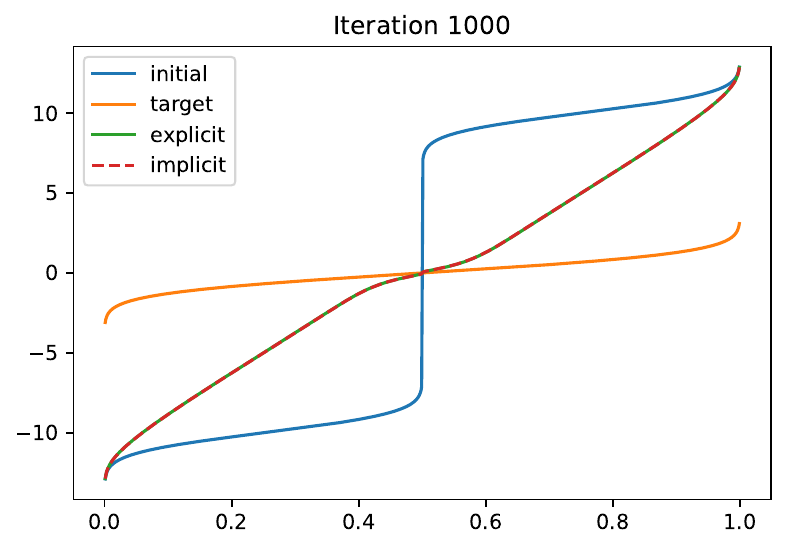}
    \includegraphics[width=.32\textwidth]{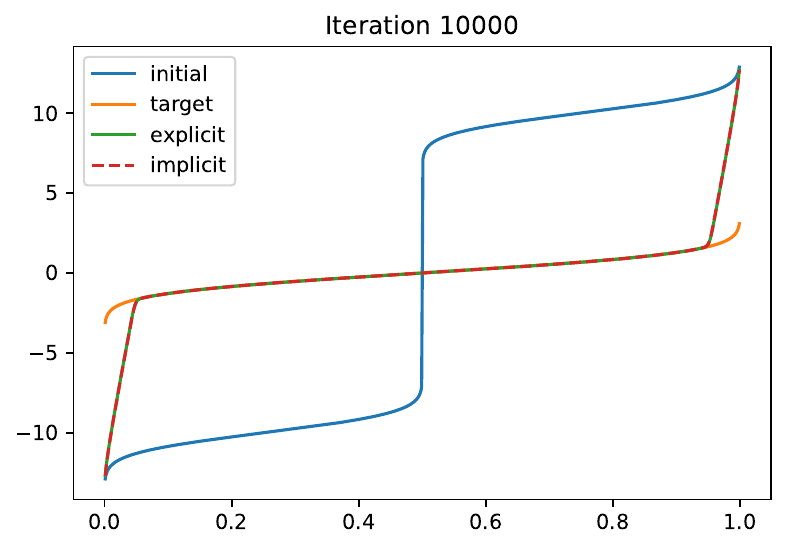}
    \caption{Comparison of of the quantile functions belonging to the implicit (red) and  explicit (green) Euler schemes between $\mu_0 \sim \frac{1}{2} \NN(-10, 1) + \frac{1}{2} \NN(10, 1)$ and $\nu \sim \NN(0, 1)$ with $\tau = \tfrac{1}{100}$. For the corresponding densities, see \figref{fig:Bimodal_Gauss_to_Gauss}.}
    \label{fig:Bimodal_Gauss_to_Gauss_q}
\end{figure}

\begin{figure}
    \centering
    \includegraphics[width=.32\textwidth]{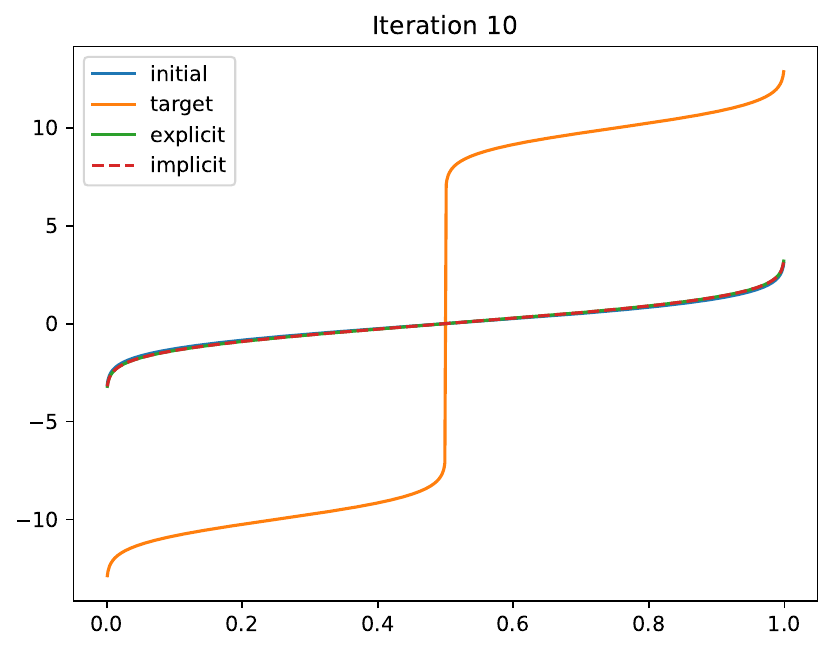}
    \includegraphics[width=.32\textwidth]{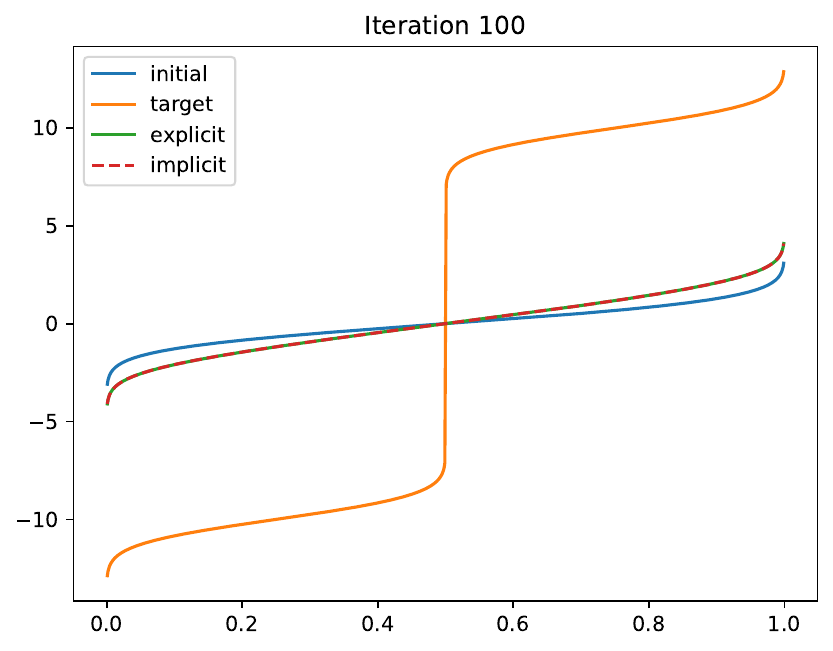}
    \includegraphics[width=.32\textwidth]{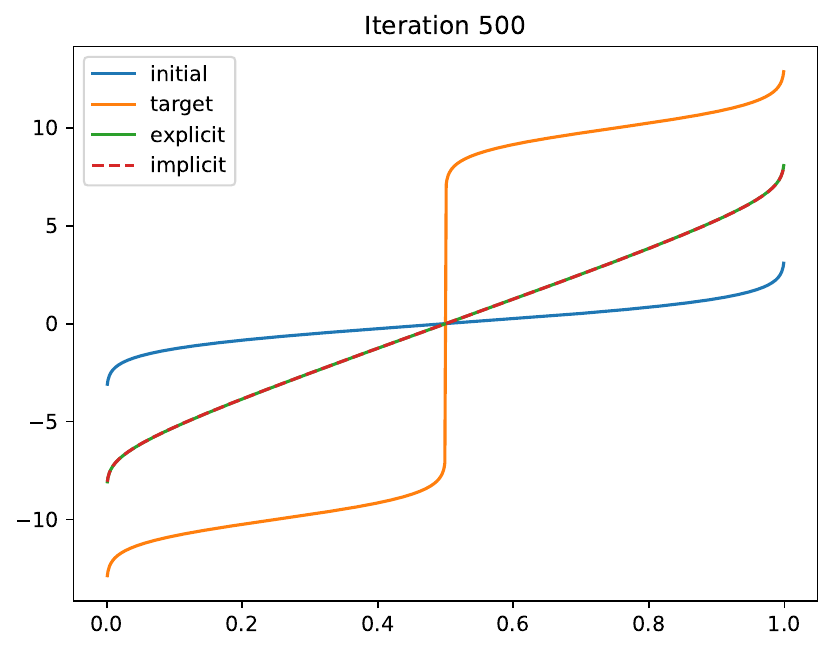}
    \includegraphics[width=.32\textwidth]{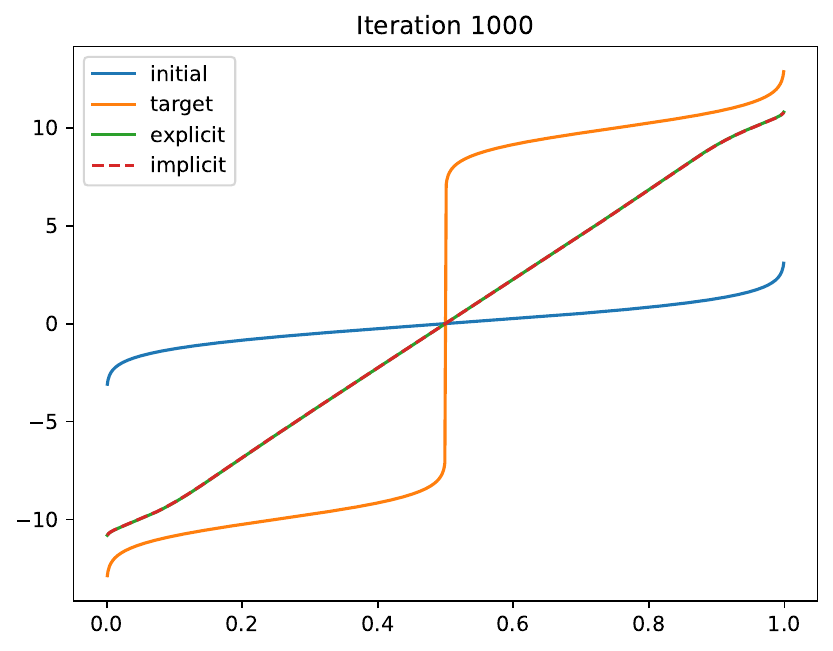}
    \includegraphics[width=.32\textwidth]{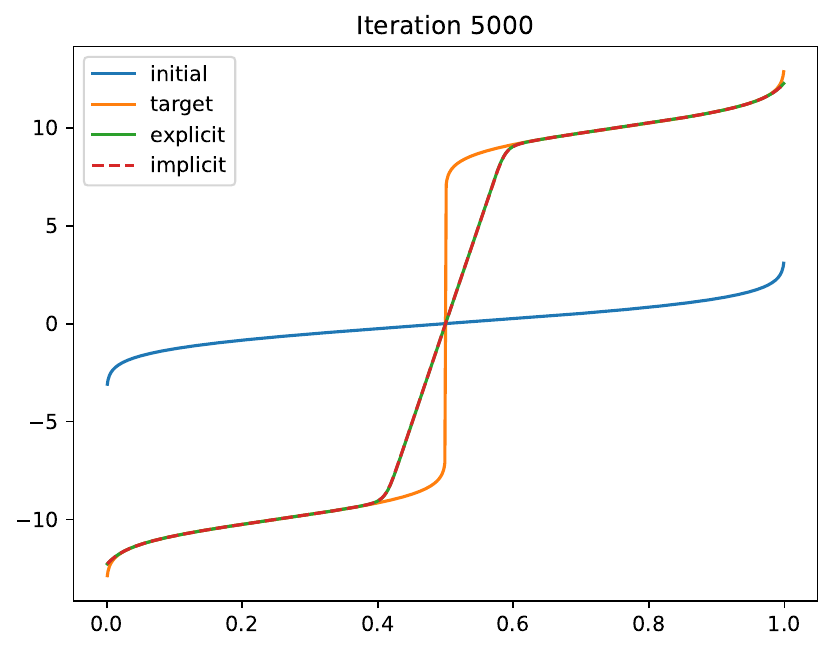}
    \includegraphics[width=.32\textwidth]{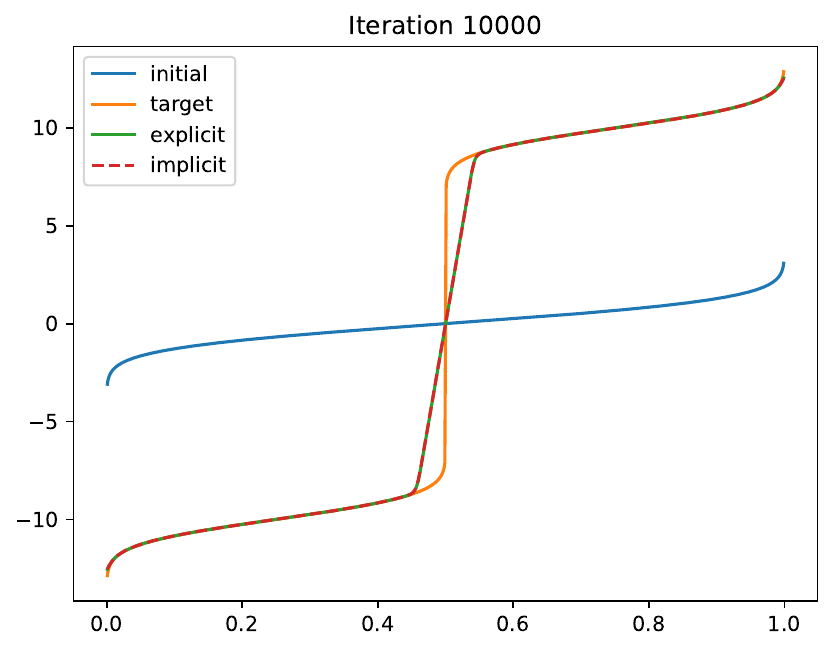}
    \caption{Comparison of implicit (red) and explicit (green) Euler schemes between $\mu_0 \sim \NN(0, 1)$ and $\nu \sim \frac{1}{2} \NN(-10, 1) + \frac{1}{2} \NN(10, 1)$ with $\tau = \tfrac{1}{100}$. For the corresponding densities, see \figref{fig:Gauss_to_Bimodal_Gauss}.}
    \label{fig:Gauss_to_Bimodal_Gauss_q}
\end{figure}

\begin{figure}
    \centering
    \includegraphics[width=.3\textwidth]{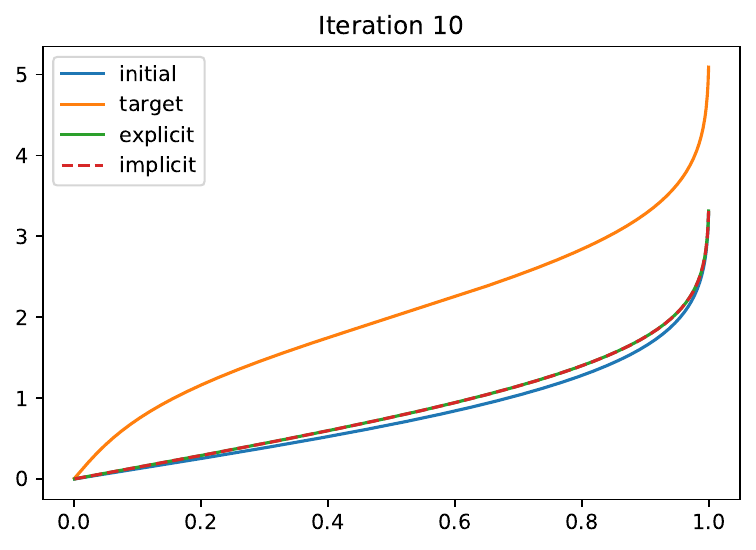}
    \includegraphics[width=.3\textwidth]{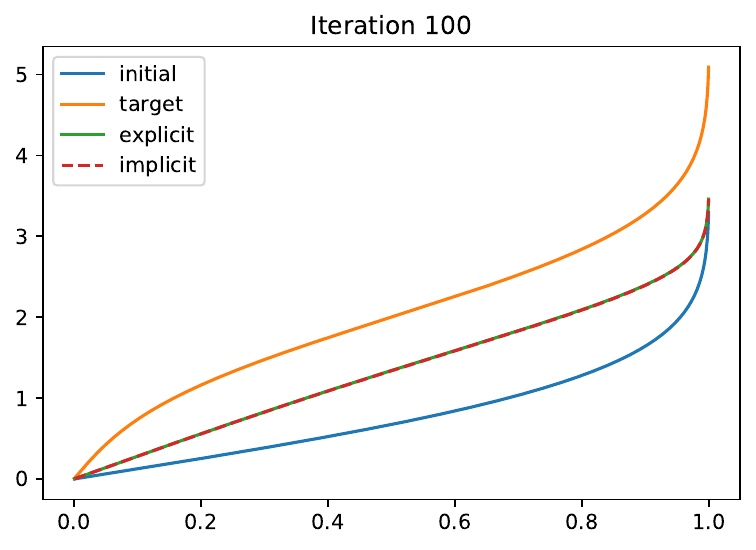}
    \includegraphics[width=.3\textwidth]{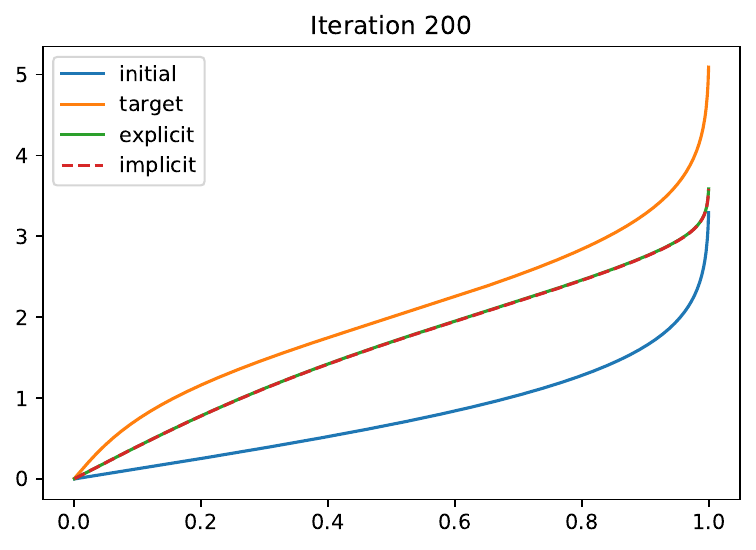}
    \includegraphics[width=.3\textwidth]{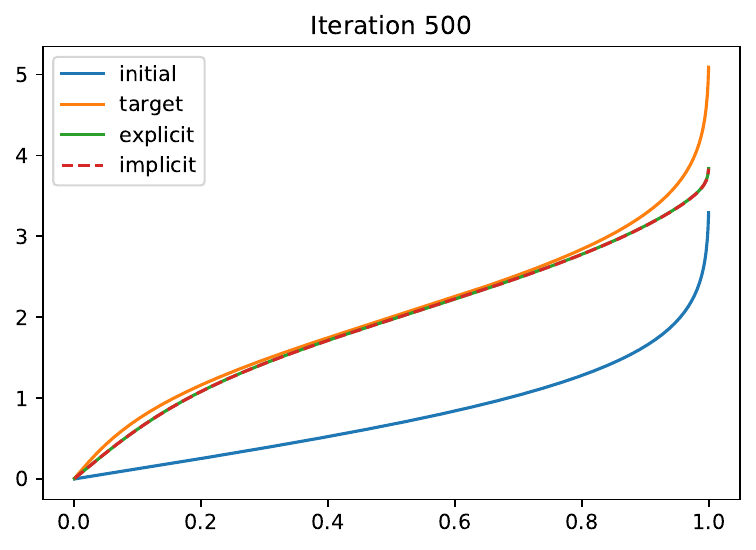}
    \includegraphics[width=.3\textwidth]{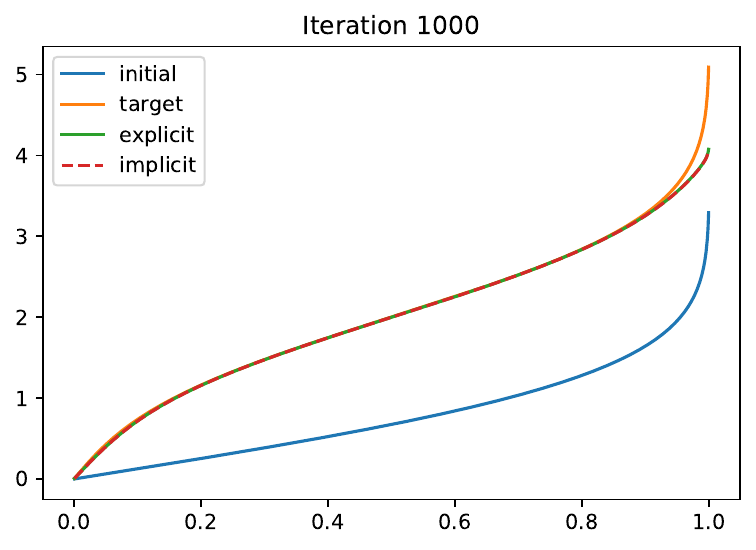}
    \includegraphics[width=.3\textwidth]{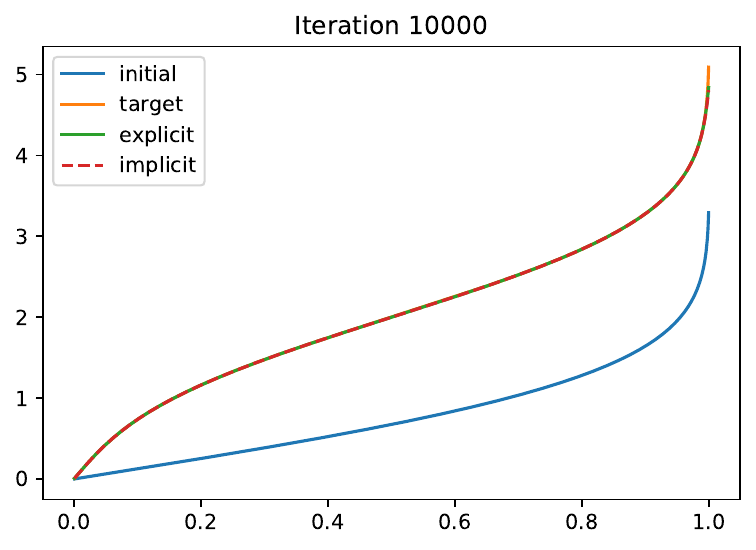}
    \caption{Comparison of the quantile functions belonging to the implicit (red) and explicit (green) Euler schemes between $\mu_0 \sim \mathcal{FN}(0, 1)$ and $\nu \sim \mathcal{FN}(2, 1)$ with $\tau = \tfrac{1}{100}$. For the corresponding densities, see \figref{fig:Folded_Norm_to_Folded_Norm}.}
    \label{fig:Folded_Norm_to_Folded_Norm_q}
\end{figure}

 
\subsection{Implicit Euler Scheme for the Flow Towards \texorpdfstring{$\delta_0$}{a Dirac Measure at Zero} Starting in a Uniform or Gaussian Measure}
\label{app:ex}
 Let $\nu = \delta_0$.
    Then the multi-valued operator $I + 2 \tau [R_{\nu}^-, R_{\nu}^+]$ appearing on the left hand side of \eqref{eq:implicit_Euler_pointwise} is given by
    \begin{equation*}
        \R \ni x \mapsto x + 2 \tau [R_{\nu}^-(x), R_{\nu}^{+}(x)]
        = \begin{cases}
            x, & \text{if } x < 0, \\
            [0, 2 \tau], & \text{if } x = 0, \\
            x + 2 \tau, & \text{if } x > 0,
        \end{cases}
    \end{equation*}
    and its inverse is given by $S_{\tau}(\cdot - \tau)$, where
    \begin{equation*}
        S_{\tau} \colon \R \to \R, \qquad
        x \mapsto \begin{cases}
            x + \tau,   & \text{if } x < - \tau, \\
            0,          & \text{if } - \tau \le x \le \tau, \\
            x - \tau,   & \text{if } x > \tau,
        \end{cases}
    \end{equation*}
    is the soft-shrinkage operator with threshold $\tau > 0$.
    
    Thus the implicit Euler scheme on quantile functions is given by the simple update
    \begin{equation*}
        g_{n + 1}(s) = S_{\tau}(g_n(s) + 2 \tau s - \tau).
    \end{equation*} 
    For a uniform initial distribution and for a Gaussian initial distribution, these iterates are displayed in \figref{fig:implicit_Euler_Dirac_Target}.
    
    \begin{figure}[H]
        \centering
        \includegraphics[width=0.4\textwidth]{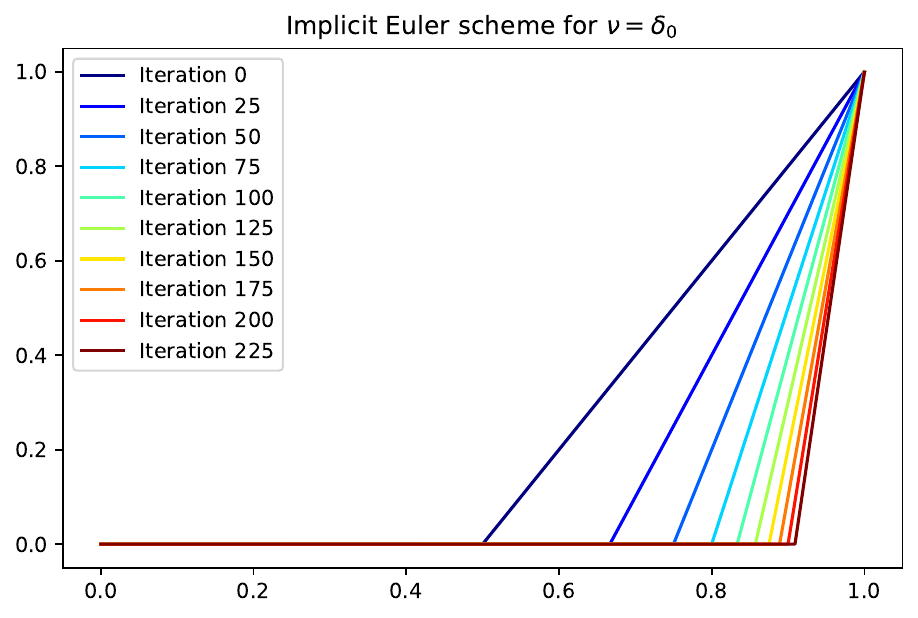}%
        \includegraphics[width=0.4\textwidth]{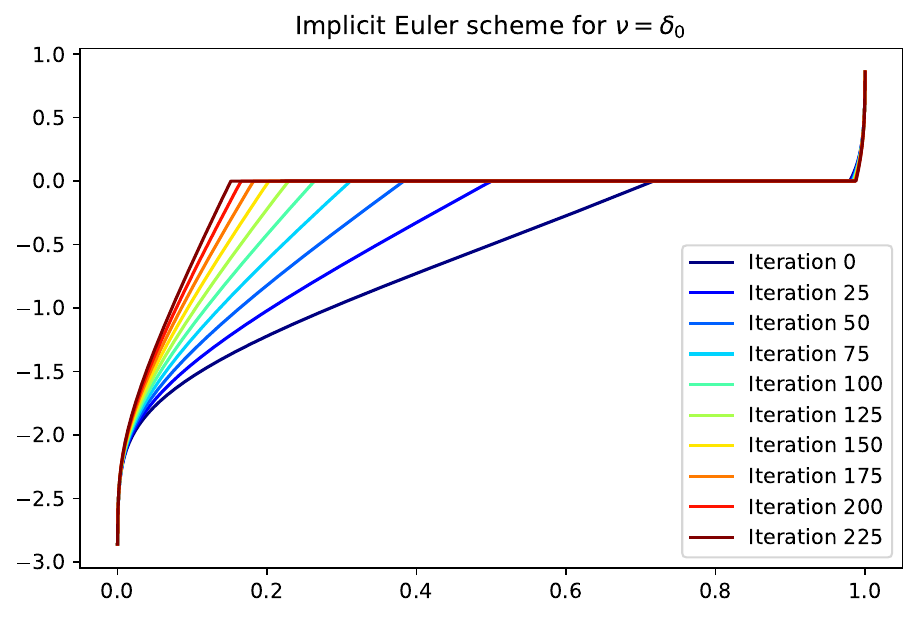}
        \caption{Iterates (quantile functions) of the implicit Euler scheme \eqref{eq:cauchy} with $\nu = \delta_0$ and $\tau = \tfrac{1}{100}$ starting in
        $\mu_0 \sim \mathcal U[0, 1]$ (left) and $\mu_0 \sim \NN(-1, 0.5)$ (right).}
        \label{fig:implicit_Euler_Dirac_Target}
    \end{figure}

\end{document}